\documentclass[a4paper,reqno,10pt]{amsart}

\usepackage{amsthm,amsmath,amssymb,mathtools,extarrows,tasks,enumitem,thmtools,xcolor,multirow,makecell}
\usepackage[all,2cell,cmtip]{xy}
\usepackage[hypertexnames=false,draft]{hyperref}
\usepackage[nameinlink]{cleveref}

\newcommand{\mysmallspace}{\mkern 3mu}
\mathchardef\ordinarycomma=\mathcode`,
\makeatletter
\newcommand{\setupactivecomma}{\mathcode`,="8000\begingroup\lccode`~=`,\lowercase{\endgroup\def~}{\ordinarycomma\mysmallspace}}
\newcommand{\dg}[1]{{
  \setupactivecomma
  \left(\begin{smallmatrix}#1\end{smallmatrix}\right)
}}
\newcommand{\xdashrightarrow}[2][]{\ext@arrow 0359\tofill@@{#1}{#2}}
\newcommand{\shortdash}{
  \mathord{\rule[0.54ex]{0.75ex}{0.37pt}}
}
\def\tofill@@{\arrowfill@@\relax\shortdash\dashrightarrow}
\def\arrowfill@@#1#2#3#4{
$\m@th\thickmuskip0mu\medmuskip\thickmuskip\thinmuskip\thickmuskip
   \relax#4#1
   \xleaders\hbox{$#4\mkern1.6mu#2\mkern1.6mu$}\hfill
   #3$
}
\renewcommand{\tocsection}[3]{%
  \indentlabel{\@ifnotempty{#2}{\ignorespaces#1 #2.\quad}}
  {\mdseries\let\normalfont\relax\let\bfseries\relax #3}%
}
\renewcommand{\tocsubsection}[3]{%
  \indentlabel{\hspace{2.2em}\@ifnotempty{#2}{\ignorespaces#1 #2.}}
  \mdseries #3%
}
\renewcommand{\@tocpagenum}[1]{}
\renewcommand{\@pnumwidth}{0pt}

\makeatother
\newcommand{\mytableofcontents}{%
  \begingroup
  \setcounter{tocdepth}{2}%
  \renewcommand{\contentsname}{}%
  \tableofcontents
  \endgroup
}
\let\oldsection\section
\RenewDocumentCommand{\section}{s o m}{%
  \IfBooleanTF{#1}{%
    \oldsection*{\normalfont\bfseries #3}%
  }{%
    \IfNoValueTF{#2}{%
      \oldsection[#3]{\normalfont\bfseries #3}%
    }{%
      \oldsection[#2]{\normalfont\bfseries #3}%
    }%
  }%
}
\pretocmd{\section}{\addtocontents{toc}{\protect\addvspace{0.7em}}}{}{}

\newcommand{\con}[6]{
  #1 \xrightarrow{#2} #3 \xrightarrow{#4} #5 \xdashrightarrow{#6}{}\!\!\!
}
\newcommand{\tri}[6]{
  #1 \xrightarrow{#2} #3 \xrightarrow{#4} #5 \xrightarrow{#6} \Sigma #1
}
\declaretheoremstyle[headfont=\bfseries,bodyfont=\itshape]{plainstyle}
\declaretheoremstyle[headfont=\bfseries,bodyfont=\normalfont]{defstyle}
\declaretheoremstyle[headfont=\bfseries,bodyfont=\normalfont]{remarkstyle}
\declaretheorem[style=plainstyle,numberwithin=section,name=Theorem]{theorem}
\declaretheorem[style=plainstyle,sibling=theorem,name=Lemma]{lemma}
\declaretheorem[style=plainstyle,sibling=theorem,name=Proposition]{proposition}
\declaretheorem[style=plainstyle,sibling=theorem,name=Corollary]{corollary}
\declaretheorem[style=defstyle,sibling=theorem,name=Definition]{definition}
\declaretheorem[style=defstyle,sibling=theorem,name=Example]{example}
\declaretheorem[style=remarkstyle,sibling=theorem,name=Remark]{remark}

\declaretheorem[style=defstyle,sibling=theorem,name=Notation]{notation}
\declaretheorem[style=defstyle,sibling=theorem,name=Condition]{condition}
\crefname{equation}{}{}
\Crefname{equation}{}{}
\let\OriginalCref\Cref
\renewcommand{\Cref}[1]{\textup{\OriginalCref{#1}}}
\setlist[enumerate]{font=\upshape, align=left, widest=(2'), leftmargin=*, labelindent=1.2em, labelsep=-.3em}
\numberwithin{equation}{section}

\title[Homotopic morphisms and diagram theorems]{Homotopic morphisms and diagram theorems \\ in extriangulated categories}
\author{Chencheng Zhang, \ \  Xue-Song Lu, \ \ Pu Zhang$^\ast$}
\thanks{$^*$ Corresponding author}
\thanks{zhangchencheng@sjtu.edu.cn \ \ leocedar@sjtu.edu.cn \ \ pzhang@sjtu.edu.cn}
\thanks{Supported by the National Key Research and Development Project 2025YFA1017202, and National Natural Science Foundation of China with Grant No. 12131015.}
\address{Chencheng Zhang, School of Mathematical Sciences, Shanghai Jiao Tong University, Shanghai 200240, PR China}
\address{Xue-Song Lu, School of Mathematical Sciences, Shanghai Jiao Tong University, Shanghai 200240, PR China}
\address{Pu Zhang, School of Mathematical Sciences, Shanghai Jiao Tong University, Shanghai 200240, PR China}
\begin{document}

\begin{abstract} \ Homotopic morphisms of $\mathbb E$-triangles in extriangulated categories are introduced.
Any morphism of $\mathbb E$-triangles is a composition of homotopic morphisms. Any morphism $(\alpha_1, \alpha_2, \alpha_3)$ of $\mathbb E$-triangles can be modified to be homotopic, by changing one of $\alpha_i$;
moreover, all the 15 cases where $\alpha_i$ is an $\mathbb E$-inflation ($\mathbb E$-deflation) are analyzed. Some diagram theorems, especially $4\times 4$ Lemma and its $14$ variants, including $3\times 3$ diagram and Horseshoe Lemma, are
investigated. A relation between homotopic morphisms and (middling) good morphisms in triangulated categories are given.
Weakly idempotent complete extriangulated categories are characterized.

\vskip5pt

\vskip5pt

\end{abstract}

\maketitle

\section{Introduction}

The notion of extriangulated category,  introduced by H. Nakaoka and Y. Palu in \cite{NP19}, is a common generalization of exact categories and triangulated categories.
There are many extriangulated categories which are neither exact nor triangulated (\cite[Remark 2.18, Proposition 3.30]{NP19}, \cite[Remark 4.13]{ZZ}; also \cite{Wu23}, \cite{Zho24}).

\vskip5pt

Diagram theorems are widely used and have interests in various categories. 
Since $\mathbb E$-triangles cannot rotate,
a diagram theorem in extriangulated categories has variants. For example,
$4\times 4$ Lemma in an extriangulated category has $14$ variants.

\vskip5pt

In this paper, homotopic morphisms of $\mathbb E$-triangles in an extriangulated category are introduced, and various diagram theorems are studied.

\vskip5pt

\subsection{Homotopic morphisms}  A homotopic morphism of $\mathbb E$-triangles is defined in \Cref{def:htm}. This notion seems to be new even for triangulated categories. 
Comparing with morphisms of $\mathbb E$-triangles, homotopic morphisms enjoy pleasant properties, and play an important role.
They induce new $\mathbb E$-triangles, and any morphism of $\mathbb E$-triangles is a composition of two homotopic morphisms (\Cref{compofhm}).

\vskip5pt

Given a morphism $(\alpha, \beta, \gamma)$ of $\mathbb E$-triangles, it is a question that whether it can be modified to be homotopic; and furthermore, if two of $\alpha, \beta, \gamma$ are in $$\{\text{$\mathbb E$-inflation and $\mathbb E$-deflation, $\mathbb E$-inflation,  $\mathbb E$-deflation}\}$$
whether the third one is still in the set, such that $(\alpha, \beta, \gamma)$ is a homotopic morphism of $\mathbb E$-triangles.
Logically, there are $27$ possible cases, but only $15$ cases admit affirmative answer (the other $12$ cases do not hold in general, as shown in \Cref{ex:no-good-completion}).
The following result claims the $8$ affirmative cases among them (the remaining $7$ cases are dual, and omitted).

\vskip5pt
\begin{theorem} \label{1.1}{\rm (\Cref{thm:hs-morphism}, \Cref{thm:hmor})} \ Let $(\mathcal{C}, \ \mathbb E, \ \mathfrak s)$ be an extriangulated category, and  $(\alpha,\beta,\gamma)$ a morphism of $\mathbb E$-triangles$:$
	\begin{equation*}
\xymatrix@R=0.4cm{
  X\ar[r]^{f}\ar[d]_-{\alpha} & Y\ar[r]^{g}\ar[d]^-{\beta} & Z\ar@{-->}[r]^{\delta}\ar[d]^-{\gamma} & {} \\
  X'\ar[r]^{f'} & Y'\ar[r]^{g'} & Z'\ar@{-->}[r]^{{\delta' }} & {}
}
\end{equation*}
Then there are $\alpha':X\longrightarrow X'$, $\beta':Y\longrightarrow Y'$, and $\gamma':Z\longrightarrow Z'$,  such that $(\alpha',\beta,\gamma)$, $(\alpha,\beta',\gamma)$, and $(\alpha,\beta,\gamma')$ are homotopic morphisms.
Explicitly, one has

\vskip5pt

$(1)$ \  If $\alpha$ and $\gamma$ are $\mathbb E$-inflations and $\mathbb E$-deflations, then so is $\beta'$.

\vskip5pt

$(2)$ \  If $\alpha$ is an $\mathbb E$-inflation and an $\mathbb E$-deflation and $\gamma$ is an $\mathbb E$-inflation, then $\beta'$ is an $\mathbb E$-inflation.

\vskip5pt

$(3)$ \  If $\alpha$ is an $\mathbb E$-inflation and $\gamma$ is an $\mathbb E$-inflation and an $\mathbb E$-deflation, then $\beta'$ is an $\mathbb E$-inflation.

\vskip5pt

$(4)$ \  If $\alpha$ and $\gamma$ are $\mathbb E$-inflations, then so is $\beta'$.

\vskip5pt

$(5)$ \  If $\alpha$ and $\beta$ are $\mathbb E$-inflations and $\mathbb E$-deflations, then $\gamma'$ is an $\mathbb E$-deflation.

\vskip5pt

$(6)$ \  If $\alpha$ is an $\mathbb E$-inflation and an $\mathbb E$-deflation and $\beta$ is an $\mathbb E$-deflation, then $\gamma'$ is an $\mathbb E$-deflation.

\vskip5pt

$(7)$ \  If $\alpha$ is an $\mathbb E$-inflation and $\beta$ is an $\mathbb E$-inflation and an $\mathbb E$-deflation, then $\gamma'$ is an $\mathbb E$-deflation.

\vskip5pt

$(8)$ \  If $\alpha$ is an $\mathbb E$-inflation and $\beta$ is an $\mathbb E$-deflation, then $\gamma'$ is an $\mathbb E$-deflation.

\vskip10pt

In case  $\mathcal{C}$ is weakly idempotent complete, Nakaoka and Palu \cite[Lemma 5.9]{NP19} in fact deals with the following cases$:$
   \vskip5pt
   $(9)$ \ If $\alpha$ is an $\mathbb E$-deflation and $\beta$ is an $\mathbb E$-inflation and an $\mathbb E$-deflation, then $\gamma'$ is
an $\mathbb E$-deflation. In this case, $\alpha$ is an $\mathbb E$-inflation and $\mathbb E$-deflation.
     \vskip5pt
   $(10)$ \ If $\alpha$ and $\beta$ are $\mathbb E$-deflations, then $\gamma'$ is an $\mathbb E$-deflation.
\end{theorem}

\vskip10pt

\subsection {$6$-term exact sequences} A distinguished property of an extriangulated category is that any $\mathbb E$-triangle $\con XfYgZ\delta$ induces two $6$-term exact sequences of functors:
    \begin{equation*}    \resizebox{0.85\textwidth}{!}{
      $\mathcal C(-,X)\xrightarrow{}\mathcal C(-,Y)\xrightarrow{}\mathcal C(-,Z)\xrightarrow{}\mathbb E(-, X)\xrightarrow{}\mathbb E(-, Y)\xrightarrow{}\mathbb E(-, Z),$
      }
\end{equation*}
and
        \begin{equation*}    \resizebox{0.85\textwidth}{!}{
      $\mathcal{C}(Z, -) \xrightarrow{} \mathcal{C}(Y, -) \xrightarrow{} \mathcal{C}(X, -) \xrightarrow{} \mathbb E(Z, -) \xrightarrow{} \mathbb E(Y, -) \xrightarrow{} \mathbb E(X, -).$
      }
\end{equation*}

\vskip5pt

For a morphism $(\alpha, \beta, 1)$ of $\mathbb E$-triangles

\begin{equation*}
\xymatrix@R=0.4cm{
  X\ar[r]^{f}\ar[d]_-{\alpha} & Y\ar[r]^{g}\ar[d]^-{\beta} & Z\ar@{-->}[r]^{\delta}\ar@{=}[d] & {} \\
  X'\ar[r]^{{f'}} & Y'\ar[r]^{{g'}} & Z\ar@{-->}[r]^{{\alpha_\ast \delta}} &. {}
}
\end{equation*}
Although the sequence $\con X{\dg{f \\\alpha}}{Y\oplus X'}{\dg{\beta,-f'}}{Y'}{(g')^\ast\delta}$ is not necessarily an $\mathbb E$-triangle, there is still a $6$-term exact sequence of covariant functors; but the $6$-term sequence of contravariant functors is not necessarily exact (\Cref{ex:contra-les-fails}); however, if $(\alpha, \beta, 1)$ is semi-homotopic (\Cref{shm}) then it is also exact.

\vskip5pt

\begin{theorem} {\rm (\Cref{prop:semi-homotopic-6}, \Cref{twolongexactseq})} \ Let $(\mathcal{C}, \mathbb E, \mathfrak s)$ be an  extriangulated category. Consider the morphism $(\alpha, \beta, 1)$ of $\mathbb E$-triangles and sequence $\con X{\dg{f \\\alpha}}{Y\oplus X'}{\dg{\beta,-f'}}{Y'}{(g')^\ast\delta}$,   as given above.

\vskip5pt

$(1)$ \ There is a $6$-term exact sequence of covariant functors
\begin{equation*}    \resizebox{1\textwidth}{!}{
      $\mathcal C(Y',-)\xrightarrow{\mathcal C(\dg{\beta,-f'},-)}\mathcal C(Y\oplus X',-)\xrightarrow{\mathcal C(\dg{f \\\alpha},-)}\mathcal C(X,-)\xrightarrow{((g')^\ast\delta)^\sharp}\mathbb E(Y',-)\xrightarrow{\dg{\beta,-f'}^\ast}\mathbb E(Y\oplus X',-)\xrightarrow{\dg{f \\\alpha}^\ast}\mathbb E(X,-).$
      }
\end{equation*}
\vskip5pt

$(2)$ \ If $(\alpha, \beta, 1)$ is semi-homotopic, then one has a $6$-term exact sequence of contravariant functors
 \begin{equation*}    \resizebox{1\textwidth}{!}{
      $\mathcal C(-,X)\xrightarrow{\mathcal C(-,\dg{f \\\alpha})}\mathcal C(-,Y\oplus X')\xrightarrow{\mathcal C(-,\dg{\beta,-f'})}\mathcal C(-,Y')\xrightarrow{((g')^\ast\delta)_\sharp}\mathbb E(-,X)\xrightarrow{\dg{f \\\alpha}_\ast}\mathbb E(-,Y\oplus X')\xrightarrow{\dg{\beta,-f'}_\ast}\mathbb E(-,Y').$
      }
\end{equation*}
\vskip5pt

$(3)$ \ If $(\mathcal{C}, \mathbb E, \mathfrak s)$ is a $\mathrm{Hom}$-finite $k$-linear category, then $(\alpha, \beta, 1)$ is semi-homotopic.
\end{theorem}

\vskip5pt

Note that a semi-homotopic morphism is not necessarily homotopic (\Cref{ex:ho-square-not-ho-mor}).

\subsection {$4\times 4$ Lemma and its variants} \ Homotopic morphisms induce many diagram theorems. A main situation is that a homotopic morphism consisting of $\mathbb E$-inflations and $\mathbb E$-deflations induces {\rm $4 \times 4$ Lemma}.

\vskip5pt

\begin{theorem} \ {\rm (\Cref{thm:4x4-sss})} \ Let $(\mathcal{C}, \mathbb E, \mathfrak s)$ be an extriangulated category, and
		\begin{equation*}
\xymatrix@R=0.4cm{
  X\ar[r]^{f}\ar[d]_-{\alpha} & Y\ar[r]^{g} & Z\ar@{-->}[r]^{\delta} \ar[d]^-{\gamma}& {} \\
  L\ar[r]^{u} & M\ar[r]^{v} & N\ar@{-->}[r]^{{\varepsilon }} & {}
}
\end{equation*}
a diagram of $\mathbb E$-triangles such that  $\alpha_\ast \delta=\gamma^\ast \varepsilon$, and $\alpha$ and $\gamma$ are $\mathbb E$-inflations and $\mathbb E$-deflations.	
Then there is an $\mathbb E$-inflation and an $\mathbb E$-deflation $\beta$ such that $(\alpha, \beta, \gamma)$ is a homotopic morphism, and there is a commutative diagram
\begin{equation*}
\xymatrix@R=0.4cm{
  K_\alpha\ar@{..>}[r]^{{f'}}\ar@{..>}[d]_{{i_\alpha }} & K_\beta\ar@{..>}[r]^{{g'}}\ar@{..>}[d]^{{i_\beta}} & K_\gamma\ar@{..>}[d]^{{i_\gamma}}\ar@{..>}[r] & {} \\
  X\ar[r]^{f}\ar[d]_-{\alpha} & Y\ar[r]^{g}\ar@{..>}[d]^-{\beta} & Z\ar@{-->}[r]^{\delta}\ar[d]^-{\gamma} & {} \\
  L\ar[r]^{u}\ar@{..>}[d]_{{p_\alpha}} & M\ar[r]^{v}\ar@{..>}[d]^{{p_\beta}} & N\ar@{-->}[r]^{{\varepsilon }}\ar@{..>}[d]^{{p_\gamma}} & {} \\
  C_\alpha\ar@{..>}[r]^{{u'}} & C_\beta\ar@{..>}[r]^{{v'}} & C_\gamma \ar@{..>}[r] & {}
}
\end{equation*}
with connecting morphism $z: K_\gamma\to C_\alpha$, such that the following are $\mathbb E$-triangles$:$
\vskip5pt
\begin{tasks}[style=enumerate, label=$(\arabic*)$, label-width=17.8pt](2)
		\task $\con {K_\alpha}{f'}{K_\beta}{g'}{K_\gamma}{};$
		\task $\con {K_\beta}{g'}{K_\gamma}z{C_\alpha}{};$
		\task $\con {K_\gamma}z{C_\alpha}{u'}{C_\beta}{};$
		\task $\con {C_\alpha}{u'}{C_\beta}{v'}{C_\gamma}{};$
		\task $\con {K_\alpha}{i_\alpha}X\alpha L{};$
		\task $\con {K_\beta}{i_\beta}Y\beta M{};$
		\task $\con {K_\gamma}{i_\gamma}Z\gamma N{};$
		\task $\con X\alpha L{p_\alpha}{C_\alpha}{};$
		\task $\con Y\beta M{p_\beta}{C_\beta}{};$
		\task $\con Z\gamma N{p_\gamma}{C_\gamma}{}$
	\end{tasks}
\vskip10pt
with morphisms of $\mathbb E$-triangles $(\alpha, \beta, \gamma)$, $(i_\alpha,i_\beta,i_\gamma)$, $(p_\alpha,p_\beta,p_\gamma)$, $(f',f,u)$, $(f,u, u')$, $(g',g,v)$ and $(g,v,v')$. In particular, there is a sequence
	\begin{equation*}		K_\alpha \xrightarrow{f'} K_\beta \xrightarrow{g'} K_\gamma \xrightarrow{z} C_\alpha \xrightarrow{u'} C_\beta \xrightarrow{v'} C_\gamma
\end{equation*}
such that any subsequent three objects form an $\mathbb E$-triangle as in $(1)$-$(4)$.
\end{theorem}

\vskip5pt

Since $\mathbb E$-triangles cannot rotate, $4\times 4$ Lemma has 14 variants (Subsection 5.4). Interesting cases include $3\times 3$ diagram and Horseshoe Lemma.

\vskip5pt

\begin{theorem}{\rm (\Cref{cor:3x3-lemma-inflation}, \Cref{horseshoe})} \ Let $(\mathcal{C}, \mathbb E, \mathfrak s)$ be an extriangulated category. Suppose that
		\begin{equation*}
\xymatrix@R=0.4cm{
  A_1\ar[r]^{f_A}\ar[d]_-{i_1} & A_2\ar[r]^{g_A} & A_3\ar@{-->}[r]^{\delta_A} \ar[d]^-{i_3}& {} \\
  B_1\ar[r]^{f_B} & B_2\ar[r]^{g_B} & B_3\ar@{-->}[r]^{{\delta_B }} & {}
}
\end{equation*}
is a diagram of $\mathbb E$-triangles such that  $(i_1)_\ast \delta_A=(i_3)^\ast \delta_B$, and $i_1$ and $i_3$ are $\mathbb E$-inflations.	
Then there is an $\mathbb E$-inflation $i_2$ such that $(i_1, i_2, i_3)$ is a homotopic morphism of $\mathbb E$-triangles, and there is a $3\times 3$ diagram$:$
	\begin{equation*}
\xymatrix@R=0.4cm{
  A_1\ar[r]^{{f_A}}\ar[d]_{{i_1}} & A_2\ar[r]^{{g_A}}\ar@{..>}[d]^{{i_2}} & A_3\ar@{-->}[r]^{{\delta_A}}\ar[d]^{{i_3}} & {} \\
  B_1\ar[r]^{{f_B}}\ar[d]_{{p_1}} & B_2\ar[r]^{{g_B}}\ar@{..>}[d]^{{p_2}} & B_3\ar@{-->}[r]^{{\delta_B}}\ar[d]^{{p_3}} & {} \\
  C_1\ar@{..>}[r]^{{f_C}}\ar@{-->}[d]_{{\varepsilon _1}} & C_2\ar@{..>}[r]^{{g_C}}\ar@{..>}[d]^{{\varepsilon _2}} & C_3\ar@{..>}[r]^{{\delta_C}}\ar@{-->}[d]^{{\varepsilon _3}} & {} \\
  {} & {} & {} & {}
}
\end{equation*}

\vskip5pt

In particular, if $i_1$ factors through $f_A$ and $\delta_B = 0$, then the $3\times 3$ diagram is of the form
	\begin{equation*}
\xymatrix@R=0.4cm{
  A_1\ar[r]^{f_A}\ar[d]_-{i_1} & A_2\ar[r]^{g_A}\ar@{..>}[d] & A_3\ar@{-->}[r]^{\delta_A}\ar[d]^-{i_3} & {} \\
  B_1\ar[r]^-{{\dg{1 \\ 0}}}\ar[d] & B_1  \oplus B_3\ar[r]^-{{(0,1)}}\ar@{..>}[d] & B_3\ar@{-->}[r]^{0}\ar[d] & {} \\
  C_1\ar@{..>}[r]\ar@{-->}[d] & C_2\ar@{..>}[r]\ar@{..>}[d] & C_3\ar@{..>}[r]\ar@{-->}[d] & {} \\
  {} & {} & {} & {}
}
\end{equation*}
\end{theorem}
\vskip5pt

\vskip5pt

As observed by A. Heller (\cite[Appendix B]{Buh10}), the $\rm (?,i,i)$-variant of $4 \times 4$ Lemma is one of the equivalent characterizations of weakly idempotent complete extriangulated categories (\Cref{thm:wic}).

\subsection{{Homotopic morphisms and (middling) good morphisms}} A. Neeman \cite[Definition 1.9]{Nee91} has introduced the notion of {\it good morphisms} and {\it middling good morphisms} of distinguished triangles in a triangulated category.
By \cite[Theorem 2.3]{Nee91} a good morphism is middling good. A relation between homotopic morphisms and (middling) good morphisms is as follows.

\vskip5pt

\begin{theorem}\label{1.5} \ {\rm (\Cref{middlinggood}, \Cref{good})} \ Let $(\mathcal T, \Sigma, \Delta)$ be a triangulated category. Then

\vskip5pt

$(1)$ \ A homotopic morphism of  distinguished triangles is middling good.

\vskip5pt

$(2)$ \ If $(\alpha, \beta, \gamma)$ is a homotopic morphism of distinguished triangles
\begin{equation*}
\xymatrix@R=0.4cm{
  X\ar[r]^{f}\ar[d]_-{\alpha} & Y\ar[r]^{g}\ar[d]^-{\beta} & Z\ar[r]^{h}\ar[d]^-{\gamma} & \Sigma X\ar[d]^{{\Sigma\alpha}} \\
  X'\ar[r]^{{f'}} & Y'\ar[r]^{{g'}} & Z'\ar[r]^{{h'}} & \Sigma X'
}
\end{equation*}
with $\mathcal T(Z', \Sigma X)=0$, then it is good.
\end{theorem}

\subsection{Organization} The paper is organized as follows.
\vspace{-20pt}

\mytableofcontents

\vspace{-20pt}

\section{Preliminaries}\label{sec:pre}

\subsection{Extriangulated categories} We recall the definition of an extriangulated category from H. Nakaoka and Y. Palu \cite{NP19}.

\vskip5pt

Let $\mathcal{C}$ be an additive category, and $\mathbb E: \mathcal{C}^{\mathrm{op}} \times \mathcal{C} \to \mathbf{Ab}$ an additive bifunctor, where $\mathbf{Ab}$ is the category of abelian groups.
For a morphism $f: X\to X'$, let $f^\ast = \mathbb E(f, -): \mathbb E(X', -)\to \mathbb E(X, -)$ and $f_\ast = \mathbb E(-, f): \mathbb E(-, X)\to \mathbb E(-, X')$
be the natural transformations. For  objects $X, Z\in\mathcal C$, an element $\delta\in \mathbb{E}(Z, X)$
is called {\it an $\mathbb{E}$-extension}. The zero element $0\in\mathbb{E}(Z, X)$ is called {\it the split $\mathbb{E}$-extension}.
Since $\mathbb E$ is a bifunctor, for a morphism $h : Z' \to Z$ and  $\delta \in \mathbb E(Z, X)$, one has
$$(f_\ast)_{Z'} (h^\ast)_{X} (\delta) = (h^\ast)_{X'} (f_\ast)_{Z} (\delta)\in \mathbb E(Z', X').$$

    \vskip5pt

   Let $X, X', Z, Z'\in \mathcal C$. For $\mathbb{E}$-extensions $\delta \in \mathbb E(Z,X)$ and $\delta' \in \mathbb E(Z', X')$, \textit{a morphism of $\mathbb E$-extensions}
   $\delta \to \delta'$ is a pair of morphisms $(\alpha, \gamma)$ with $\alpha: X\to X'$ and $\gamma: Z\to Z'$, such that $\alpha_\ast \delta = \gamma^\ast \delta'\in \mathbb E(Z, X')$.
   Denote by $\delta \oplus \delta' \in \mathbb E(Z \oplus Z', X \oplus X')$ the image of $(\delta, \delta') \in \mathbb E(Z,X) \oplus \mathbb E(Z',X')$ under the natural inclusion
    \begin{equation*}\resizebox{.9\textwidth}{!}{
        $\mathbb E(Z,X) \oplus \mathbb E(Z',X') \hookrightarrow  \mathbb E(Z,X) \oplus \mathbb E(Z',X')\oplus \mathbb E(Z',X) \oplus \mathbb E(Z,X') \cong  \mathbb E(Z \oplus Z', X \oplus X')$.
        }
        \end{equation*}
In particular, if $X=X'$ and $Z=Z'$, then
\begin{equation*}    \dg{1_Z \\ 1_Z}^\ast \dg{1_X,1_X}_\ast (\delta \oplus \delta') =   \dg{1 , 1} \dg{\delta & 0 \\ 0 & \delta'} \dg{1 \\ 1}  = \delta + \delta'.
\end{equation*}

Let $X, Y, Y', Z\in \mathcal C$. By definition two sequences $X \xrightarrow{f} Y \xrightarrow{g} Z$ and $X \xrightarrow{f'} Y' \xrightarrow{g'} Z$ are {\it equivalent}, if there exists an isomorphism $\varphi: Y \to Y'$ such that the following diagram commutes:
\begin{equation*}
\xymatrix@R=0.4cm{
  X\ar[r]^{f}\ar@{=}[d] & Y\ar[r]^{g}\ar[d]^{{\varphi }}_{\cong} & Z\ar@{=}[d] \\
  X\ar[r]^{{f'}} & Y'\ar[r]^{{g'}} & Z
}
\end{equation*}

    A {\it realization} $\mathfrak s$ of $\mathbb E$ is a collection of ``mappings'', sending each $\delta \in \mathbb E(Z,X)$ for $Z, X\in \mathcal C$ to an equivalence class of sequences of the form $[X \xrightarrow{f} Y \xrightarrow{g} Z]$, satisfying that for any morphism of $\mathbb E$-extensions $(\alpha, \gamma): \delta \to \delta'$ and any representative $X \xrightarrow{f} Y \xrightarrow{g} Z$ of $\mathfrak s(\delta)$ and $X' \xrightarrow{f'} Y' \xrightarrow{g'} Z'$ of $\mathfrak s(\delta')$, there exists $\beta: Y \to Y'$ such that the following diagram commutes:
			\begin{equation*}
\xymatrix@R=0.4cm{
  X\ar[r]^{f}\ar[d]_-{\alpha} & Y\ar[r]^{g}\ar@{..>}[d]^-{\beta} & Z\ar[d]^-{\gamma} \\
  X'\ar[r]^{{f'}} & Y'\ar[r]^{{g'}} & Z'
}
\end{equation*}
In this case, one says that an $\mathbb E$-extension $\delta$ is {\it realized} by $X \xrightarrow f Y \xrightarrow g Z$, if
\begin{equation*}  \mathfrak s (\delta) = \left[X \xrightarrow f Y \xrightarrow g Z\right].
\end{equation*}

A {\it realization} $\mathfrak s$ of $\mathbb E$ is {\it additive}, if any split $\mathbb E$-extension $0\in \mathbb E(Y,X)$ is realized by
$X \xrightarrow{\dg{1\\ 0}} X \oplus Y \xrightarrow{\dg{0 , 1}} Y ;$ and if
$\mathfrak s(\delta_i)=[X_i \xrightarrow {f_1}Y_i\xrightarrow {g_i}Z_i]$, $i =1, 2$, then $\delta_1 \oplus \delta_2$ is realized by
		\begin{equation*}			X_1 \oplus X_2 \xrightarrow{\dg{f_1 & 0 \\ 0 & f_2}} Y_1 \oplus Y_2 \xrightarrow{\dg{g_1 & 0 \\ 0 & g_2}} Z_1 \oplus Z_2.
\end{equation*}

\vskip10pt

\begin{definition} \ {\rm (\cite[Definition 2.12]{NP19})} \ An {\it extriangulated category} is a triplet $(\mathcal{C}, \mathbb E, \mathfrak s)$, satisfying the following axioms$:$

\vskip5pt

{\bf ET1.} \ $\mathcal{C}$ is an additive category, and $\mathbb E: \mathcal{C}^{\mathrm{op}}\times \mathcal{C} \to \mathbf{Ab}$ is an additive bifunctor.

\vskip5pt

{\bf ET2.} \ $\mathfrak s$ is an additive realization of $\mathbb E$.

\vskip5pt

If an $\mathbb E$-extension $\delta$ is realized by $X \xrightarrow{f} Y \xrightarrow{g} Z$,  which will be denoted by $\con XfYgZ\delta$, then it is called an \textit{$\mathbb E$-triangle},
$f$ an \textit{$\mathbb E$-inflation} and $g$ an \textit{$\mathbb E$-deflation}. {\it A morphism} of $\mathbb E$-triangles is a triple $(\alpha, \beta, \gamma)$ such that $\alpha_\ast \delta = \gamma ^\ast \delta '$ and the following diagram commutes:
	\begin{equation}\label{morphism}
\xymatrix@R=0.4cm{
  X\ar[r]^{f}\ar[d]_-{\alpha} & Y\ar[r]^{g}\ar[d]^-{\beta} & Z\ar@{-->}[r]^{\delta}\ar[d]^-{\gamma} & {} \\
  X'\ar[r]^{{f'}} & Y'\ar[r]^{{g'}} & Z'\ar@{-->}[r]^{{\delta'}} & {}
}
\end{equation}

{\bf ET3.} \ Let $\con XfYgZ\delta$ and $\con {X'}{f'}{Y'}{g'}{Z'}{\delta'}$ be $\mathbb E$-triangles. If there are $\alpha : X \to X'$ and $\beta : Y \to Y'$ such that $\beta f = f'  \alpha$,
then there exists $\gamma: Z\to Z'$ such that $(\alpha, \beta, \gamma)$ is a morphism of $\mathbb E$-triangles.

\vskip5pt

{\bf ET3$^{\mathrm{op}}$.} \ Let $\con XfYgZ\delta$ and $\con {X'}{f'}{Y'}{g'}{Z'}{\delta'}$ be $\mathbb E$-triangles. If there are $\beta : Y \to Y'$ and $\gamma : Z' \to Z$ such that $\gamma  g = g'  \beta$,
then there exists $\alpha:X\to X'$ such that $(\alpha, \beta, \gamma)$ is a morphism of $\mathbb E$-triangles.

\vskip5pt

{\bf ET4.} \ Let $\con AfBgD\delta$ and $\con BuCvE\varepsilon$ be $\mathbb E$-triangles. Then there exists a commutative diagram
	\begin{equation*}
\xymatrix@R=0.4cm{
  A\ar[r]^{f}\ar@{=}[d] & B\ar[r]^{g}\ar[d]^{u} & D\ar@{-->}[r]^{\delta}\ar@{..>}[d]^{w} & {} \\
  A\ar@{..>}[r]^{m} & C\ar@{..>}[r]^{h}\ar[d]^{v} & F\ar@{..>}[r]^{\theta}\ar@{..>}[d]^{q} & {} \\
  {} & E\ar@{=}[r]\ar@{-->}[d]^{{\varepsilon }} & E\ar@{..>}[d]^{\eta} & {} \\
  {} & {} & {} & {}
}
\end{equation*}
	such that $\con AmChF\theta$ and $\con DwFqE\eta$ are $\mathbb E$-triangles and $w^\ast \theta=\delta$, $g_\ast \varepsilon=\eta$ and $f_\ast \theta= q^\ast\varepsilon$. In particular, $\mathbb E$-inflations are closed under compositions.

\vskip5pt

{\bf ET4$^{\mathrm{op}}$.} \ Let $\con AmChF\theta$ and $\con DwFqE\eta$ be $\mathbb E$-triangles. Then there exists a commutative diagram
	\begin{equation*}
\xymatrix@R=0.4cm{
  A\ar@{..>}[r]^{f}\ar@{=}[d] & B\ar@{..>}[r]^{g}\ar@{..>}[d]^{u} & D\ar@{..>}[r]^{\delta}\ar[d]^{w} & {} \\
  A\ar[r]^{m} & C\ar[r]^{h}\ar@{..>}[d]^{v} & F\ar@{-->}[r]^{\theta}\ar[d]^{q} & {} \\
  {} & E\ar@{=}[r]\ar@{..>}[d]^{{\varepsilon }} & E\ar@{-->}[d]^{\eta} & {} \\
  {} & {} & {} & {}
}
\end{equation*}
	such that $\con AfBgD\delta$ and $\con BuCvE\varepsilon$ are $\mathbb E$-triangles and $w^\ast \theta=\delta$, $g_\ast \varepsilon=\eta$ and $f_\ast \theta= q^\ast\varepsilon$. In particular, $\mathbb E$-deflations are closed under compositions.
\end{definition}

\vskip5pt

\begin{proposition} \label{isomorphisms} {\rm (\cite[Proposition 3.7]{NP19})} \ Let $(\mathcal{C}, \mathbb E, \mathfrak s)$ be an extriangulated category, and $\con XfYgZ\delta$  an $\mathbb E$-triangle.
Then, for any isomorphisms $\alpha: X\to X'$, \ $\beta : Y \to Y'$ and $\gamma: Z\to Z'$,
$$\con {X'}{\beta f \alpha^{-1}}{Y'}{\gamma g \beta^{-1}}{Z'}{\alpha_\ast (\gamma^{-1})^\ast \delta}$$
is also an $\mathbb E$-triangle.	In particular, $\con XfY{-g}Z{-\delta}$, $\con X{-f}YgZ{-\delta}$ and $\con X{-f}Y{-g}Z\delta$ are $\mathbb E$-triangles.
\end{proposition}

\vskip5pt

\begin{proposition} \label{prop:pb-2} {\rm (The dual of \cite[Proposition 3.17]{NP19})} \ Let $(\mathcal{C}, \mathbb E, \mathfrak s)$ be an extriangulated category.
For $\mathbb E$-triangles $\con{A_1}{e_1}{E}{p_2}{B_2}{\eta}$, $\con{A_1}{f_1}{B_1}{g_1}{C}{\delta_1}$ and $\con{A_2}{e_2}{E}{p_1}{B_1}{\varepsilon}$ with $f_1=p_1e_1$, there is a completion of the diagram
\begin{equation}\label{eq:pb-2}
\xymatrix@R=0.4cm{
  {} & A_2\ar@{=}[r]\ar[d]^{{e_2}} & A_2\ar@{..>}[d]^{{f_2}} & {} \\
  A_1\ar[r]^{{e_1}}\ar@{=}[d] & E\ar[r]^{{p_2}}\ar[d]^{{p_1}} & B_2\ar@{-->}[r]^{\eta}\ar@{..>}[d]^{{g_2}} & {} \\
  A_1\ar[r]^{{f_1}} & B_1\ar[r]^{{g_1}}\ar@{-->}[d]^{{\varepsilon }} & C\ar@{-->}[r]^{{\delta_1}}\ar@{..>}[d]^{{\delta_2}} & {} \\
  {} & {} & {} & {}
}
\end{equation}
    such that $\con{A_2}{f_2}{B_2}{g_2}{C}{\delta_2}$ is an $\mathbb E$-triangle, $(g_1)^\ast\delta_2=\varepsilon$, $(g_2)^\ast \delta_1=\eta$ and $(e_1)_\ast\delta_1+(e_2)_\ast\delta_2=0$.
\end{proposition}

\subsection{$6$-term exact sequences} Note that the category of covariant (respectively, contravariant) functors from an additive category to {\bf Ab} is an abelian category.

\vskip5pt

\begin{theorem} \label{lem:long-ext-seq} {\rm (\cite[Corollary 3.12]{NP19})} \ Let $(\mathcal{C}, \mathbb E, \mathfrak s)$ be an extriangulated category. For any $\mathbb E$-triangle $\con XfYgZ\delta$,
there is a $6$-term exact sequences of functors$:$
    \begin{equation*}    \resizebox{0.85\textwidth}{!}{
      $\mathcal C(-,X)\xrightarrow{\mathcal C(-,f)}\mathcal C(-,Y)\xrightarrow{\mathcal C(-,g)}\mathcal C(-,Z)\xrightarrow{\delta_\sharp}\mathbb E(-, X)\xrightarrow{f_\ast}\mathbb E(-, Y)\xrightarrow{g_\ast}\mathbb E(-, Z),$
      }
\end{equation*}
    where $\delta_\sharp : \mathcal{C}(-, Z) \to \mathbb E(-, X)$ is a natural transformation sending $T \xrightarrow{\varphi} Z$ to $\varphi^\ast \delta;$
         and  there is a $6$-term exact sequences of functors$:$  \begin{equation*}     \resizebox{0.85\textwidth}{!}{
 $\mathcal{C}(Z, -) \xrightarrow{\mathcal{C}(g, -)} \mathcal{C}(Y, -) \xrightarrow{\mathcal{C}(f, -)} \mathcal{C}(X, -) \xrightarrow{\delta^\sharp} \mathbb E(Z, -) \xrightarrow{g^\ast} \mathbb E(Y, -) \xrightarrow{f^\ast} \mathbb E(X, -),$
      }
\end{equation*}
where $\delta^\sharp : \mathcal{C}(X, -) \to \mathbb E(Z, -)$ is a natural transformation sending $X \xrightarrow{\psi} T$ to $\psi_\ast \delta$.
\end{theorem}

\vskip5pt

\begin{remark}\label{naturality}  \ 		Moreover, $\delta_\sharp$ and $\delta^\sharp$ are natural at $\delta$, in the sense that given a morphism of $\mathbb E$-triangles $(\alpha, \beta, \gamma)$ as in \Cref{morphism}, the following diagrams commute:
	\begin{equation*}
\xymatrix@R=0.4cm{
  \mathcal C(-, Z)\ar[r]^{{\delta_\sharp}}\ar[d]_{{\mathcal C(-,\gamma)}} & \mathbb E(-, X)\ar[d]^{{\alpha_\ast}} & \mathcal C(X',-)\ar[r]^{{(\delta')^\sharp}}\ar[d]_{{\mathcal C(\alpha,-)}} & \mathbb E(Z',-)\ar[d]^{{\gamma^\ast}} \\
  \mathcal C(-, Z')\ar[r]^{{(\delta')_\sharp}} & \mathbb E(-, X') & \mathcal C(X,-)\ar[r]^{{\delta^\sharp}} & \mathbb E(Z,-)
}
\end{equation*}
\end{remark}

	\begin{proof} \ Let $T\in \mathcal C$. For $\varphi: T \to Z$,  there is
		\begin{equation*}			\resizebox{0.9\textwidth}{!}{
			$(\delta')_\sharp^T( \mathcal C(T, \gamma)(\varphi)) = (\delta')_\sharp^T (\gamma \varphi) = (\gamma \varphi)^\ast \delta' = \varphi^\ast \gamma^\ast \delta' = \varphi^\ast \alpha_\ast \delta = \alpha_\ast \varphi ^\ast \delta = \alpha_\ast \delta_\sharp^T(\varphi).$
			}
\end{equation*}
        For $\psi: X'\to T$,  there is
		\begin{equation*}			\resizebox{0.9\textwidth}{!}{
			$\delta^\sharp_T( \mathcal C(\alpha, T)(\psi)) = \delta^\sharp_T (\psi \alpha) = (\psi \alpha)_\ast \delta = \psi_\ast \alpha_\ast \delta  = \psi_\ast \gamma^\ast \delta'  = \gamma^\ast\psi_\ast  \delta'  = \gamma^\ast (\delta')^\sharp_T(\psi).$
			}
\end{equation*}
	\end{proof}

\begin{corollary}\label{cone}  {\rm (\cite[Corollary 3.6]{NP19})} \  Let $(\mathcal{C}, \mathbb E, \mathfrak s)$ be an extriangulated category, $(\alpha, \beta, \gamma)$ be a morphism of $\mathbb E$-triangles.
If two of $\alpha, \beta, \gamma$ are isomorphisms, then so is the third one.

\end{corollary}

\subsection{Weakly idempotent completeness} \  A morphism $\varphi: A\to B$ is a \textit{section}
(respectively, a \textit{retraction}) if there is a morphism $\psi: B\to A$ such that $\psi \varphi = 1_A$ (respectively, $\varphi \psi = 1_B$).

\vskip5pt

\begin{lemma}\label{kernelcokernel} \ Let $\mathcal A$ be an additive category. For morphisms $f :A\to B$ and $g: B \to A$ with $gf=1_A$, $f$ admits a cokernel if and only if $g$ admits a kernel.
\end{lemma}
\vskip5pt

\begin{lemma}\label{lem:infl-section} \ Let $(\mathcal{C}, \mathbb E, \mathfrak s)$ be an extriangulated category.
\vskip5pt
$(1)$ \ {\rm (\cite[Corollary 3.5]{NP19})} \ Let $\con XfYgZ\delta$ be an $\mathbb E$-triangle. Then $f$ is a section if and only if $\delta = 0$, if and only if $g$ is a retraction.
\vskip5pt
$(2)$ \  An $\mathbb E$-inflation which is a section admits a cokernel, and an $\mathbb E$-deflation which is a retraction admits a kernel.
\vskip5pt
$(3)$ \  A monic $\mathbb E$-deflation is a section, and an epic $\mathbb E$-inflation is a retraction.	
 \end{lemma}

 \begin{proof} \ One only shows (3). Let $\con XfYgZ\delta$ be an $\mathbb E$-triangle. Suppose that $g$ is a monomorphism. Since $gf=0$, $f=0$. It follows that $1_Yf=0$. By \Cref{lem:long-ext-seq} there exists $h: Z \to Y$ such that $hg = 1_Y$, i.e., $g$ is a section. Similarly, one can prove the dual assertion.
\end{proof}

\ An additive category $\mathcal A$ is \textit{weakly idempotent complete},
provided that any section has a cokernel, or equivalently, any retraction has a kernel (see e.g. \cite[Lemma 7.1, Definition 7.2]{Buh10}).
A triangulated category is always weakly idempotent complete. For weakly idempotent complete exact categories one refers to \cite[Section 7]{Buh10}.

\vskip5pt

\begin{condition} \label{con:wic} {\rm (\cite[Condition 5.8]{NP19})} \ An extriangulated category $(\mathcal{C} , \mathbb E, \mathfrak s)$ satisfies WIC condition, if for any composable morphisms $f: X\to Y$ and $g: Y\to Z:$
\vskip5pt
$(1)$ \  If $gf$ is an $\mathbb E$-inflation, then so is $f$;
\vskip5pt
$(2)$ \ If $gf$ is an $\mathbb E$-deflation, then so is $g$.
\end{condition}

\vskip5pt

An extriangulated category is weakly idempotent complete if and only if it satisfies WIC condition (\cite[Proposition C]{Kla22}).
More equivalent conditions will be discussed in \Cref{thm:wic}.

\section{Homotopic morphisms}

Throughout this section we work in an extriangulated category $(\mathcal{C}, \ \mathbb E, \ \mathfrak s)$.
\subsection{Homotopic squares and homotopic morphisms}

Homotopic squares in triangulated categories are introduced in \cite{BN93}, and generalized to extriangulated categories in \cite{He19}.
They have different names, e.g., homotopy Cartesian squares (\cite{Nee01}), homotopy pullback and pushout squares (\cite{May01}), distinguished weak squares (\cite{Kun07}).

\vskip5pt

\begin{definition}\label{def:htsq} \ A \textit{homotopic square} in an extriangulated category is a commutative diagram
	\begin{equation}
\label{eq:htsq}
\xymatrix@R=0.4cm{
  A_1\ar[r]^{f}\ar[d]_{u} & B_1\ar[d]^{v} \\
  A_2\ar[r]^{g} & B_2
}
\end{equation}
		such that $\con{A_1}{\dg{f\\ u}}{B_1 \oplus A_2}{\dg{v , -g}}{B_2}\delta$ is an $\mathbb E$-triangle for some $\mathbb E$-extension $\delta\in \mathbb E(B_2, A_1)$.
\end{definition}

\vskip5pt

The following fact is well-known.

\vskip5pt

\begin{lemma}\label{lem:weak} \ Consider the homotopic square \Cref{eq:htsq}.

 \vskip5pt

$(1)$ \ For any morphisms $m : T \to B_1$ and $n : T \to A_2$ with $v m = g n$, there is $s : T \to A_1$ such that $fs=m$ and $us=n$.
\begin{equation*}
\xymatrix@R=0.4cm{
  T\ar@{..>}[dr]^{s}\ar@/^1pc/[drr]^{m}\ar@/_1pc/[ddr]_{n} & {} & {} \\
  {} & A_1\ar[r]^{f}\ar[d]_{u} & B_1\ar[d]^{v} \\
  {} & A_2\ar[r]^{g} & B_2
}
\end{equation*}

$(2)$ \ For any morphisms $m' : B_1\to T'$ and $n' : A_2\to T'$ with $m' f = n' u$, there is $s' : B_2 \to T'$ such that $s'v=m'$ and $s'g=n'$.
	\begin{equation*}
\xymatrix@R=0.4cm{
  A_1\ar[r]^{f}\ar[d]_{u} & B_1\ar[d]^{v}\ar@/^1pc/[ddr]^{{m'}} & {} \\
  A_2\ar[r]^{g}\ar@/_1pc/[drr]_{{n'}} & B_2\ar@{..>}[dr]^{{s'}} & {} \\
  {} & {} & T'
}
\end{equation*}
\end{lemma}

\begin{lemma}\label{thm:hs-composition} {\rm (\cite[Theorem 3.2]{HXZ23})} \ Suppose that two squares in the following commutative diagram are homotopic squares
	\begin{equation*}
\xymatrix@R=0.4cm{
  A\ar[r]^{f}\ar[d]_-{\alpha} & B\ar[r]^{g}\ar[d]^-{\beta} & C\ar[d]^-{\gamma} \\
  D\ar[r]^{u} & E\ar[r]^{v} & F
}
\end{equation*}
i.e., $\con{A}{\dg{f\\ \alpha}}{B \oplus D}{\dg{\beta,-u}}E\kappa$ and $\con B{\dg{g \\ \beta}}{C\oplus E}{(\gamma,-v)}F\varepsilon$ are $\mathbb E$-triangles.
Then the outside rectangle is also a homotopic square with an $\mathbb E$-triangle $\con A{\dg{gf \\ \alpha}}{C\oplus D}{\dg{\gamma ,-vu}}F\delta$, such that $v^\ast \delta = \kappa$ and $\varepsilon=f_\ast\delta$.
\end{lemma}

\vskip10pt

A main notion defined in this paper is as follows.

\vskip5pt

\begin{definition}\label{def:htm} \  A morphism of $\mathbb E$-triangles $(\alpha, \beta, \gamma)$ in \Cref{morphism} is \textit{homotopic}, provided that there exists a commutative diagram
\begin{equation*}
\xymatrix@R=0.4cm{
  X\ar[r]^{f}\ar[d]_-{\alpha} & Y\ar[r]^{g}\ar@{..>}[d]_{{\beta_1}} & Z\ar@{-->}[r]^{\delta} & {} \\
  X'\ar[r]^{s}\ar@{=}[d] & E\ar[r]^{t}\ar@{..>}[d]_{{\beta_2}} & Z\ar@{=}[u]\ar@{-->}[r]^{\eta}\ar[d]^-{\gamma} & {} \\
  X'\ar[r]^{{f'}} & Y'\ar[r]^{{g'}} & Z'\ar@{-->}[r]^{{\delta'}} & {}
}
\end{equation*}
such that $(\alpha, \beta_1,1_Z)$ and $(1_{X'},\beta_2,\gamma)$ are morphisms of $\mathbb E$-triangles with $\beta = \beta_2 \beta_1$, and there are  $\mathbb E$-triangles
\begin{equation*}	\con X{\dg{f\\\alpha}}{Y\oplus X'}{\dg{\beta_1,-s}}E{t^\ast\delta}\ \text{and} \  \  \con E{\dg{-t\\\beta_2}}{Z\oplus Y'}{\dg{\gamma,g'}}{Z'}{s_\ast \delta'}.
\end{equation*}
\end{definition}

\vskip5pt

\begin{remark} \ In an exact category, any morphism of conflations is homotopic (see e.g. \cite[Proposition 3.1]{Buh10}). However, this is not true in a triangulated category: 
an example can be found in \cite[Section 4]{CK11}.
\end{remark}

\vskip5pt

\begin{proposition}\label{prop:mor-of-confl} \  $(1)$ \ Suppose that the following is a morphism of $\mathbb E$-triangles$:$

\[\xymatrix@R=0.4cm{
  X\ar[r]^{f}\ar[d]_-{\alpha} & Y\ar[r]^{g}\ar[d]^-{\beta} & Z\ar@{-->}[r]^{\delta}\ar@{=}[d] & {} \\
  X'\ar[r]^{{f'}} & Y'\ar[r]^{{g'}} & Z\ar@{-->}[r]^{{\alpha_\ast \delta}} &}\]
Then $(\alpha,\beta,1_Z)$ is a homotopic morphism if and only if $\con X{\dg{f\\ \alpha}}{Y\oplus X'}{\dg{\beta,- f'}}{Y'}{(g')^\ast \delta}$ is an $\mathbb E$-triangle. If this is the case, then the left square is homotopic.

\vskip5pt

$(1)'$ \ Suppose that the following is a morphism of $\mathbb E$-triangles$:$
\[\xymatrix@R=0.4cm{
  X\ar[r]^{f}\ar@{=}[d] & Y\ar[r]^{g}\ar[d]_-{\beta} & Z\ar@{-->}[r]^{{\gamma^*\varepsilon}}\ar[d]^-{\gamma} & {} \\
  X\ar[r]^{f'} & Y'\ar[r]^{g'} & Z'\ar@{-->}[r]^{\varepsilon} & }\]
Then $(1_X,\beta,\gamma)$ is a homotopic morphism if and only if $\con Y{\dg{-g\\ \beta}}{Z\oplus Y'}{\dg{\gamma,g'}}{Z'}{f_\ast \varepsilon}$ is an $\mathbb E$-triangle. If this is the case, then the right square is homotopic.
\end{proposition}
\begin{proof} \ We only justify $(1)$. If $\con X{\dg{f\\ \alpha}}{Y\oplus X'}{\dg{\beta,- f'}}{Y'}{(g')^\ast \delta}$ is an $\mathbb E$-triangle, then the following diagram shows that $(\alpha,\beta,1_Z)$ is a homotopic morphism:
\begin{equation*}
\xymatrix@R=0.4cm{
  X\ar[r]^{f}\ar[d]_-{\alpha} & Y\ar[r]^{g}\ar[d]^-{\beta} & Z\ar@{-->}[r]^{\delta} & {} \\
  X'\ar[r]^{{f'}}\ar@{=}[d] & Y'\ar[r]^{{g'}}\ar@{=}[d] & Z\ar@{=}[u]\ar@{-->}[r]^{{\alpha_\ast \delta}}\ar@{=}[d] & {} \\
  X'\ar[r]^{{f'}} & Y'\ar[r]^{{g'}} & Z\ar@{-->}[r]^{{\alpha_\ast\delta}} & {}
}
\end{equation*}
where $\con {Y'}{\dg{-g' \\ 1_{Y'}}}{Z \oplus Y'}{\dg{1_Z, g'}}{Z}{0}$ is a split $\mathbb E$-triangle.

\vskip5pt

Conversely, suppose that $(\alpha,\beta,1_Z)$ is a homotopic morphism. By definition there is a commutative diagram
\begin{equation*}
\xymatrix@R=0.5cm{
  X\ar[r]^{f}\ar[d]_-{\alpha} & Y\ar[r]^{g}\ar[d]^(.4){{\sigma^{-1}\beta}} & Z\ar@{-->}[r]^{\delta} & {} \\
  X'\ar[r]^{{\sigma^{-1}f'}}\ar@{=}[d] & Y''\ar[r]^{{g'\sigma}}\ar[d]^{\sigma} & Z\ar@{=}[u]\ar@{-->}[r]^{{\alpha_\ast \delta}}\ar@{=}[d] & {} \\
  X'\ar[r]^{{f'}} & Y'\ar[r]^{{g'}} & Z\ar@{-->}[r]^{{\alpha_\ast\delta}} & {}
}
\end{equation*}
where $\sigma$ is an isomorphism by \Cref{cone}, and $\con X{\dg{f\\ \alpha}}{Y\oplus X'}{\dg{\sigma^{-1}\beta,- \sigma^{-1}f'}}{Y''}{(g'\sigma)^\ast \delta}$ is an $\mathbb E$-triangle. By \Cref{isomorphisms} the following commutative diagram shows that $\con X{\dg{f\\ \alpha}}{Y\oplus X'}{\dg{\beta,- f'}}{Y'}{\theta}$ is an $\mathbb E$-triangle, where $\theta=(\sigma^{-1})^\ast(g'\sigma)^{\ast}\delta =(\sigma^{-1})^\ast\sigma^\ast(g')^{\ast}\delta=(g')^\ast \delta $.
\begin{equation*}
\xymatrix@R=0.4cm{
  X\ar[r]^(0.4){{\dg{f \\ \alpha}}}\ar@{=}[d] & Y\oplus X'\ar[rr]^{{\dg{\beta, f'}}}\ar@{=}[d] & {} & Y'\ar@{-->}[rr]^{\theta}\ar[d]^(.4){{\sigma^{-1}}} & & \\
  X\ar[r]^(0.4){{\dg{f \\ \alpha}}} & Y\oplus X'\ar[rr]^(0.55){{\dg{\sigma^{-1}\beta,\sigma^{-1}f'}}} & {} & Y''\ar@{-->}[rr]^-{{(g'\sigma)^\ast\delta}} & &
}
\end{equation*}
\end{proof}

\vskip5pt

\begin{remark}\label{rmk:mor-vs-con} \ Even if the left square of the diagram in \Cref{prop:mor-of-confl}$(1)$ is homotopic,
$\con X{\dg{f\\ \alpha}}{Y\oplus X'}{\dg{\beta,- f'}}{Y'}{(g')^\ast \delta}$ is not necessarily an $\mathbb E$-triangle, and $(\alpha, \beta, 1_Z)$ is not necessarily a homotopic morphism of $\mathbb E$-triangles (\Cref{ex:ho-square-not-ho-mor}).

\end{remark}

\vskip5pt

It is natural to ask if a composition of two homotopic morphisms of $\mathbb{E}$-triangles is itself homotopic. The answer is negative, as shown by the following result.

\vskip5pt

\begin{proposition} \label{compofhm} \  Any morphism of $\mathbb E$-triangles is a composition of two homotopic morphisms.
	\begin{proof} \ Let $(\alpha,\beta,\gamma)$ be a morphism of $\mathbb E$-triangles as in \Cref{morphism}. Consider
\begin{equation*}
\xymatrix@C=1.5cm@R=.5cm{
  X\ar[r]^-{f}\ar[d]_-{ \scalebox{0.8}{$\dg{1 \\ \alpha}$}} & Y\ar[r]^-{g}\ar[d]^-{ \dg{1 \\ \beta}}  & Z\ar@{-->}[r]^-{\delta}\ar[d]^-{ \scalebox{0.8}{$\dg{1 \\ \gamma}$}} & {} \\
  X \oplus X'\ar[r]^-{ \scalebox{0.8}{$\dg{f&0\\ 0&f'}$}}\ar[d]_-{(0,1)} &  Y \oplus Y'\ar[r]^-{ \scalebox{0.8}{$\dg{g&0\\ 0&g'}$}}\ar[d]^-{(0,1)}  & Z \oplus Z'\ar@{-->}[r]^-{\delta \oplus \delta'}\ar[d]^-{(0,1)} & {} \\
  X'\ar[r]^-{f'}  & Y'\ar[r]^-{g'} & Z'\ar@{-->}[r]^-{\delta'} & {}
}
\end{equation*}

\vskip5pt

\noindent We claim that the following commutative diagram gives a homotopic morphism of $\mathbb E$-triangles $(\dg{1 \\ \alpha},\dg{1 \\ \beta},\dg{1 \\ \gamma})$.
\begin{equation*}
\xymatrix@C=1.5cm@R=.5cm{
  X\ar[r]^-{f}\ar[d]_-{ \scalebox{0.8}{$\dg{1 \\ \alpha}$}} & Y\ar[r]^-{g}\ar[d]^-{ \scalebox{0.8}{$\dg{1 \\ 0}$}} & Z\ar@{-->}[r]^-{\delta}\ar@{=}[d] & {} \\
  X \oplus X'\ar[r]^-{ \scalebox{0.8}{$\dg{f & 0 \\ -\alpha & 1}$}}\ar@{=}[d] & Y \oplus X'\ar[r]^-{\dg{g,0}}\ar[d]^-{ \scalebox{0.8}{$\dg{1 & 0 \\ \beta & f'}$}} & Z\ar@{-->}[r]^-{ \scalebox{0.8}{$\dg{1 \\ \alpha}_\ast \delta$}}\ar[d]^-{ \scalebox{0.8}{$\dg{1 \\ \gamma}$}} & {} \\
  X \oplus X'\ar[r]^-{ \scalebox{0.8}{$\dg{f & 0 \\ 0 & f'}$}} & Y \oplus Y'\ar[r]^(.6){ \scalebox{0.8}{$\dg{g & 0 \\ 0 & g'}$}} & Z \oplus Z'\ar@{-->}[r]^-{\delta \oplus \delta'} & {}
}
\end{equation*}

\vskip5pt

\noindent In fact $\dg{1 \\ \gamma}^\ast (\delta \oplus \delta') = \dg{1 \\ \alpha}_\ast \delta$, since $\dg{1,0}_\ast \dg{1 \\ \gamma}^\ast (\delta \oplus \delta') = \delta = \dg{1,0}_\ast \dg{1 \\ \alpha}_\ast \delta$ and $$\dg{0,1}_\ast \dg{1 \\ \gamma}^\ast (\delta \oplus \delta') = \gamma ^\ast \delta' = \alpha_\ast \delta = \dg{0,1}_\ast \dg{1 \\ \alpha}_\ast \delta.$$ Note that $g^\ast \delta = 0$, the following is a split $\mathbb E$-triangle
\begin{equation*}
  \con X {\scalebox{0.8}{$\dg{f \\ 1 \\ \alpha}$}}{ Y \oplus X \oplus X'} {\scalebox{0.8}{$\dg{1 & -f & 0 \\ 0 & \alpha & -1}$}} {Y \oplus X'}{(g,0)^\ast \delta}.
\end{equation*}

In the commutative diagram with vertical isomorphisms:
\begin{equation*}
\xymatrix@C=1cm{
  Y \oplus X'\ar[rr]^-{ \scalebox{0.8}{$\dg{0&0\\1&0\\0&f'}$}}\ar@{=}[d] & {} & Z \oplus Y \oplus Y'\ar[rr]^-{ \dg{1&0&0\\0&0&g'}}\ar[d]_-{ \scalebox{0.8}{$\dg{1 & -g & 0 \\ 0 & 1 & 0 \\ 0 & \beta & 1} $}} & {} & Z \oplus Z'\ar@{-->}[rr]^-{0 \oplus \delta ' }\ar[d]_-{ \dg{1 & 0 \\ \gamma & 1}} & & {} \\
  Y \oplus X'\ar[rr]^-{ \scalebox{0.8}{$\dg{-g&0\\1&0\\\beta&f'}$}} & {} & Z \oplus Y \oplus Y'\ar[rr]^-{ \scalebox{0.8}{$\dg{1&g&0\\\gamma&0&g'}$}} & {} & Z \oplus Z'\ar@{-->}[rr]^-{ \scalebox{0.8}{$\dg{f & 0\\ -\alpha & 1}_\ast (\delta \oplus \delta')$}} & & {}
}
\end{equation*}
the first row is the direct sum of a split $\mathbb E$-triangle and $\con {X'}{f'}{Y'}{g'}{Z'}{\delta'}$. By \Cref{isomorphisms} the second row is also an $\mathbb E$-triangle. This proves the claim.

\vskip5pt

In particular  $(\dg{1 \\ 0},\dg{1 \\ 0},\dg{1 \\ 0})$ is a homotopic morphism of $\mathbb E$-triangles. Similarly, $(\dg{0,1}, \dg{0,1}, \dg{0,1})$ is a homotopic morphism of $\mathbb E$-triangles, and then
$(\alpha, \beta, \gamma) = (\dg{0,1}, \dg{0,1}, \dg{0,1}) \circ (\dg{1 \\ \alpha},\dg{1 \\ \beta},\dg{1 \\ \gamma})$.
	\end{proof}

\end{proposition}

\subsection{From morphisms of $\mathbb E$-triangles to homotopic morphisms}
Comparing with a morphism of $\mathbb E$-triangles, a homotopic morphism of $\mathbb E$-triangles enjoys pleasant properties.
For example, it induces two new $\mathbb E$-triangles:
$$\con X{\dg{f\\\alpha}}{Y\oplus X'}{\dg{\beta_1,-s}}E{t^\ast\delta},  \ \ \ \  \  \con E{\dg{-t\\\beta_2}}{Z\oplus Y'}{\dg{\gamma,g'}}{Z'}{s_\ast \delta'}.$$
This subsection aims to obtain homotopic morphisms of $\mathbb E$-triangles, from a morphism $(\alpha, \beta, \gamma)$  of $\mathbb E$-triangles, by changing one of $\alpha, \beta, \gamma$.
This is in fact a good fill-in of ET2, ET3 and ET3$^{\rm op}$.

\vskip5pt

\begin{lemma}\label{lem:hs-f?1}\label{lem:hs-?g1} \label{lem:hs-1?h}\label{lem:hs-1g?}
    Let $\con XfYgZ\delta$, $\con{X'}{f'}{Y'}{g'}Z{\delta'}$ and $\con{X}{f''}{Y''}{g''}{Z''}{\delta''}$ be $\mathbb E$-triangles.
\vskip5pt
$(1)$ \ {\rm (\cite[Proposition 1.20]{LN19}; \cite[Lemma 3.2]{He19})} \ Suppose that there exists $\alpha:X\to X'$ such that $\delta'=\alpha_\ast \delta$. Then there exists $\beta: Y\to Y'$ such that
$(\alpha, \beta, 1)$ is a homotopic morphism of $\mathbb E$-triangles$:$
        \begin{equation*}
\xymatrix@R=0.4cm{
  X\ar[r]^{f}\ar[d]_-{\alpha} & Y\ar[r]^{g}\ar@{..>}[d]^-{\beta} & Z\ar@{-->}[r]^{\delta}\ar@{=}[d] & {} \\
  X'\ar[r]^{{f'}} & Y'\ar[r]^{{g'}} & Z\ar@{-->}[r]^{{\delta'}} & {}
}
\end{equation*}
        \vskip5pt
$(2)$ \ {\rm (\cite[Theorem 3.3]{Kon24})} \ Suppose that there exists $\beta: Y\to Y'$ such that $g'\beta=g$. Then there exists $\alpha:X\to X'$ such that $(\alpha, \beta, 1)$ is a homotopic morphism of $\mathbb E$-triangles$:$
\begin{equation*}
\xymatrix@R=0.4cm{
  X\ar[r]^{f}\ar@{..>}[d]_-{\alpha} & Y\ar[r]^{g}\ar[d]^-{\beta} & Z\ar@{-->}[r]^{\delta}\ar@{=}[d] & {} \\
  X'\ar[r]^{{f'}} & Y'\ar[r]^{{g'}} & Z\ar@{-->}[r]^{{\delta'}} & {}
}
\end{equation*}

\vskip5pt
$(1)'$ \ Suppose that there exists $\gamma:Z\to Z''$ such that $\delta=\gamma^\ast \delta''$. Then there exists $\beta: Y\to Y''$ such that $(1, \beta, \gamma)$ is a homotopic morphism of $\mathbb E$-triangles$:$
    \begin{equation*}
\xymatrix@R=0.4cm{
  X\ar[r]^{f}\ar@{=}[d] & Y\ar[r]^{g}\ar@{..>}[d]^-{\beta} & Z\ar@{-->}[r]^{\delta}\ar[d]^-{\gamma} & {} \\
  X\ar[r]^{{f''}} & Y''\ar[r]^{{g''}} & Z''\ar@{-->}[r]^{{\delta''}} & {}
}
\end{equation*}
\vskip5pt
$(2)'$ \ Suppose that there exists $\beta: Y\to Y''$ such that $\beta f=f''$. Then there exists $\gamma:Z\to Z''$ such that $(1, \beta, \gamma)$ is a homotopic morphism of $\mathbb E$-triangles$:$
    \begin{equation*}
\xymatrix@R=0.4cm{
  X\ar[r]^{f}\ar@{=}[d] & Y\ar[r]^{g}\ar[d]^-{\beta} & Z\ar@{-->}[r]^{\delta}\ar@{..>}[d]^-{\gamma} & {} \\
  X\ar[r]^{{f''}} & Y''\ar[r]^{{g''}} & Z''\ar@{-->}[r]^{{\delta''}} & {}
}
\end{equation*}

\end{lemma}

\vskip5pt

Similar results in triangulated categories for the following statement can be found in Neeman \cite[Theorem 1.8]{Nee91} and Christensen and Frankland \cite[Remark 2.4]{CF22}.

\vskip5pt

\begin{theorem}\label{thm:hs-morphism} \ Let $(\alpha,\beta,\gamma)$ be a morphism of $\mathbb E$-triangles
	\begin{equation*}
\xymatrix@R=0.4cm{
  X\ar[r]^{f}\ar[d]^-{\alpha} & Y\ar[r]^{g}\ar[d]^-{\beta} & Z\ar@{-->}[r]^{\delta}\ar[d]^-{\gamma} & \ \\
  X'\ar[r]^{f'} & Y'\ar[r]^{g'} & Z'\ar@{-->}[r]^{{\delta' }} &  .
}
\end{equation*}
Then there exist $\alpha':X\longrightarrow X'$, $\beta':Y\longrightarrow Y'$, and $\gamma':Z\longrightarrow Z'$,  such that $(\alpha',\beta,\gamma)$, $(\alpha,\beta',\gamma)$, and $(\alpha,\beta,\gamma')$ are homotopic morphisms of $\mathbb E$-triangles.
\end{theorem}
\begin{proof} \ We first show the existence of $\beta'$. There exists an $\mathbb E$-triangle $\con {X'}sEtZ{\alpha_\ast \delta}$ realizing $\alpha_\ast \delta = \gamma^\ast \delta'$. One then takes $\beta_1$ and $\beta_2$ by \Cref{lem:hs-f?1}$(1)$ and \Cref{lem:hs-1?h}$(1)'$ respectively, as in the following diagram:
		\begin{equation*}
\xymatrix@R=0.4cm{
  X\ar[r]^{f}\ar[d]^-{\alpha} & Y\ar[r]^{g}\ar@{..>}[d]^-{{\beta_1}} & Z\ar@{-->}[r]^{\delta}\ar@{=}[d] & {} \\
  X'\ar[r]^{s}\ar@{=}[d] & E\ar[r]^{t}\ar@{..>}[d]^-{{\beta_2}} & Z\ar@{-->}[r]^{{\alpha_\ast \delta}}\ar[d]^-{\gamma} & {} \\
  X'\ar[r]^{{f'}} & Y'\ar[r]^{{g'}} & Z'\ar@{-->}[r]^{{\delta' }} & {}
}
\end{equation*}
That is, $(\alpha, \beta'=\beta_2\beta_1, \gamma)$ is a homotopic morphism.

\vskip5pt

We next show the existence of $\alpha'$. By \Cref{lem:hs-1?h}$(1)'$ and \Cref{lem:weak}$(1)$, there is a commutative diagram
		\begin{equation*}
\xymatrix@R=0.4cm{
  X\ar[r]^{f} & Y\ar[r]^{g}\ar@{..>}[d]^-{\beta_1} & Z\ar@{-->}[r]^{\delta}\ar@{=}[d] & {} \\
  X'\ar[r]^{s}\ar@{=}[d] & E\ar[r]^{t}\ar@{..>}[d]^-{\beta_2} & Z\ar@{-->}[r]^{{\gamma^\ast \delta' }}\ar[d]^-{\gamma} & {} \\
  X'\ar[r]^{f'} & Y'\ar[r]^{g'} & Z'\ar@{-->}[r]^{{\delta' }} & {}
}
\end{equation*}
		where $\beta_2\beta_1 = \beta$ and $t\beta_1 = g$. The desired $\alpha' : X \to X'$ is given by \Cref{lem:hs-?g1}$(2)$.

\vskip5pt

The existence of $\gamma'$ is similar to that of $\alpha'$.
	\end{proof}

\begin{lemma} \label{prop:hs-inflation} \  Consider the homotopic square in \Cref{eq:htsq}.
\vskip5pt
$(1)$ \  {\rm (\cite[Lemma 3.1]{HXZ23})} \  $f$ is an $\mathbb E$-inflation if and only if $g$ is an $\mathbb E$-inflation. In this case, there exists a homotopic morphism of $\mathbb E$-triangles$:$
\begin{equation*}
\xymatrix@R=0.4cm{
  A_1\ar[r]^{f}\ar[d]_{u} & B_1\ar[r]\ar[d]^{v} & C\ar@{-->}[r]^{{\delta_1}}\ar@{=}[d] & {} \\
  A_2\ar[r]^{g} & B_2\ar[r] & C\ar@{-->}[r]^{{u_\ast\delta_1}} & {}
}
\end{equation*}
\vskip5pt
$(2)$ \ {\rm (The dual of \cite[Lemma 3.1]{HXZ23})} \ $f$ is an $\mathbb E$-deflation if and only if $g$ is an $\mathbb E$-deflation. In this case, there exists a homotopic morphism of $\mathbb E$-triangles$:$
\begin{equation*}
\xymatrix@R=0.4cm{
  K\ar[r]\ar@{=}[d] & A_1\ar[r]^{f}\ar[d]_{u} & B_1\ar@{-->}[r]^{{v^\ast \varepsilon _2}}\ar[d]^{v} & {} \\
  K\ar[r] & A_2\ar[r]^{g} & B_2\ar@{-->}[r]^{{\varepsilon_2}} & {}
}
\end{equation*}
\end{lemma}

\vskip5pt

The following result is an analog of K{\"u}nzer's axiom in exact categories (see \cite{Kun07} and \cite[Exercise 3.11]{Buh10}).

\vskip5pt

\begin{proposition}\label{prop:comp-infldefl} \ Let  $f :X \to Y$ and $g : Y \to Z$ be morphisms.

\vskip5pt

$(1)$ \  If $gf$ is an $\mathbb E$-inflation with $g$ an $\mathbb E$-deflation, then $f$ is an $\mathbb E$-inflation.

\vskip5pt

In particular, if $f$ is an $\mathbb E$-inflation, then $\dg{f \\ s}: X\rightarrow Y\oplus W$ is an $\mathbb E$-inflation for any morphism $s: X\rightarrow W$.

\vskip5pt

$(1)'$ \ If $gf$ is an $\mathbb E$-deflation with $f$ an $\mathbb E$-inflation, then $g$ is an $\mathbb E$-deflation.

\vskip5pt

In particular, if $g$ is an $\mathbb E$-deflation, then $(g, t): Y\oplus W \rightarrow Z$ is an $\mathbb E$-deflation for any morphism $t: W\rightarrow Z$.
\end{proposition}
\begin{proof} \ One only justifies $(1)$. Since $gf$ is an $\mathbb E$-inflation, there is an $\mathbb E$-triangle $\con X{gf}Z h C\delta$. Since both $g$ and $h$ are $\mathbb E$-deflations, $hg$ is an $\mathbb E$-deflation. It follows that there is an $\mathbb E$-triangle $\con StY{hg}C{\varepsilon}$. By \Cref{lem:hs-?g1}$(2)$ there is $p: S\to X$ such that $\con S {\dg{t\\ p}} {Y\oplus X} {\dg{g , -gf}} Z {h^\ast \varepsilon}$ is an $\mathbb E$-triangle.
\begin{equation*}
\xymatrix@R=0.4cm{
  S\ar[r]^{t}\ar@{..>}[d]_{p} & Y\ar[r]^{hg}\ar[d]^{g} & C\ar@{=}[d] \\
  X\ar[r]^{gf} & Z\ar[r]^{h} & C
}
\end{equation*}

Since $g$ is an $\mathbb E$-deflation, $p$ is an $\mathbb E$-deflation by \Cref{prop:hs-inflation}$(2)$. By \Cref{lem:weak}$(1)$ there is $s: X\to S$ such that $ts=f$ and $ps=1_X$.

        \begin{equation*}
\xymatrix@R=0.4cm{
  {} & K\ar@{=}[r]\ar@{..>}[d] & K\ar[d] & {} \\
  X\ar@{..>}[r]^{s}\ar@/^1.4pc/[rr]^(0.3){f}\ar@{=}@/_0.6pc/[dr] & S\ar@{..>}[r]^(0.4){t}\ar@{..>}[d]_{p} & Y\ar@{..>}[r]^{hg}\ar[d]^{g} & C\ar@{=}[d] \\
  {} & X\ar[r]^{gf} & Z\ar[r]^{h} & C
}
\end{equation*}
        Since $p$ is an $\mathbb E$-deflation and a retraction, $g$ admits a kernel by \Cref{lem:infl-section}. Thus, $s$ admits a cokernel by \Cref{kernelcokernel}. Then $s$ is an $\mathbb E$-inflation since every split short exact sequence is an $\mathbb E$-triangle. Finally, $f=ts$ is an $\mathbb E$-inflation by ET4.
    \end{proof}

\Cref{thm:hs-morphism} modifies a morphism  $(\alpha,\beta,\gamma)$ of $\mathbb E$-triangles into a homotopic one.
In applications, we meet such a question that if two of $\alpha, \beta, \gamma$ are in the set $$\{\text{$\mathbb E$-inflation and $\mathbb E$-deflation, $\mathbb E$-inflation,  $\mathbb E$-deflation}\}$$
whether the third one is still in the set, such that $(\alpha, \beta, \gamma)$ is a homotopic morphism of $\mathbb E$-triangles.
Logically, there are $27$ possible cases, but only $15$ cases admit affirmative answer (the other $12$ cases do not hold in general, as shown in \Cref{ex:no-good-completion}).
The following result claims the $8$ affirmative cases among them (the remaining $7$ cases are dual, and omitted).

\vskip5pt
    \begin{theorem}\label{thm:hmor} \ Let $(\mathcal{C}, \ \mathbb E, \ \mathfrak s)$ be an extriangulated category, and  $(\alpha,\beta,\gamma)$ a morphism of $\mathbb E$-triangles$:$
	\begin{equation*}
\xymatrix@R=0.4cm{
  X\ar[r]^{f}\ar[d]^-{\alpha} & Y\ar[r]^{g}\ar[d]^-{\beta} & Z\ar@{-->}[r]^{\delta}\ar[d]^-{\gamma} & {} \\
  X'\ar[r]^{f'} & Y'\ar[r]^{g'} & Z'\ar@{-->}[r]^{{\delta' }} & {}
}
\end{equation*}
Then there are $\alpha':X\longrightarrow X'$, $\beta':Y\longrightarrow Y'$, and $\gamma':Z\longrightarrow Z'$,  such that $(\alpha',\beta,\gamma)$, $(\alpha,\beta',\gamma)$, and $(\alpha,\beta,\gamma')$ are homotopic morphisms of $\mathbb E$-triangles.
Explicitly one has

\vskip5pt

$(1)$ \  If $\alpha$ and $\gamma$ are $\mathbb E$-inflations and $\mathbb E$-deflations, then so is $\beta'$.

\vskip5pt

$(2)$ \  If $\alpha$ is an $\mathbb E$-inflation and an $\mathbb E$-deflation and $\gamma$ is an $\mathbb E$-inflation, then $\beta'$ is an $\mathbb E$-inflation.

\vskip5pt

$(3)$ \  If $\alpha$ is an $\mathbb E$-inflation and $\gamma$ is an $\mathbb E$-inflation and an $\mathbb E$-deflation, then $\beta'$ is an $\mathbb E$-inflation.

\vskip5pt

$(4)$ \  If $\alpha$ and $\gamma$ are $\mathbb E$-inflations, then so is $\beta'$.

\vskip5pt

$(5)$ \  If $\alpha$ and $\beta$ are $\mathbb E$-inflations and $\mathbb E$-deflations, then $\gamma'$ is an $\mathbb E$-deflation.

\vskip5pt

$(6)$ \  If $\alpha$ is an $\mathbb E$-inflation and an $\mathbb E$-deflation and $\beta$ is an $\mathbb E$-deflation, then $\gamma'$ is an $\mathbb E$-deflation.

\vskip5pt

$(7)$ \  If $\alpha$ is an $\mathbb E$-inflation and $\beta$ is an $\mathbb E$-inflation and an $\mathbb E$-deflation, then $\gamma'$ is an $\mathbb E$-deflation.

\vskip5pt

$(8)$ \  If $\alpha$ is an $\mathbb E$-inflation and $\beta$ is an $\mathbb E$-deflation, then $\gamma'$ is an $\mathbb E$-deflation.

\vskip10pt

In case  $\mathcal{C}$ is weakly idempotent complete, Nakaoka and Palu \cite[Lemma 5.9]{NP19} in fact deals with the following cases$:$
   \vskip5pt
   $(9)$ \ If $\alpha$ is an $\mathbb E$-deflation and $\beta$ is an $\mathbb E$-inflation and an $\mathbb E$-deflation, then $\gamma'$ is
an $\mathbb E$-deflation. In this case, $\alpha$ is an $\mathbb E$-inflation and $\mathbb E$-deflation.
     \vskip5pt
   $(10)$ \ If $\alpha$ and $\beta$ are $\mathbb E$-deflations, then $\gamma'$ is an $\mathbb E$-deflation.
\end{theorem}

\begin{proof} \ By \Cref{thm:hs-morphism} we only need to prove the explicit cases. It suffices to prove $(1)$, $(5)$ and $(9)$, all the other cases can be similarly proved.

\vskip5pt

$(1)$ \ Assume that $(\alpha, \beta, \gamma)$ is a morphism of $\mathbb E$-triangles such that $\alpha$ and $\gamma$ are $\mathbb E$-inflations and $\mathbb E$-deflations.
By \Cref{thm:hs-morphism} there is $\beta': Y\to Y'$ such that $(\alpha, \beta', \gamma)$ is a homotopic morphism. That is, there is a decomposition
		\begin{equation*}
\xymatrix@R=0.4cm{
  X\ar[r]^{f}\ar[d]_-{\alpha} & Y\ar[r]^{g}\ar[d]_{{\beta_1}} & Z\ar@{-->}[r]^{\delta}\ar@{=}[d] & {} \\
  X'\ar[r]^{s}\ar@{=}[d] & E\ar[r]^{t}\ar[d]_{{\beta_2}} & Z\ar@{-->}[r]^{{\eta}}\ar[d]^-{\gamma} & {} \\
  X'\ar[r]^{{f'}} & Y'\ar[r]^{{g'}} & Z'\ar@{-->}[r]^{{\delta'}} & {}
}
\end{equation*}
    where $\con{X}{\dg{f\\\alpha}}{Y \oplus X'}{\dg{\beta_1,-s}}{E}{t^\ast \delta}$ and $\con{E}{\dg{-t\\\beta_2}}{Z \oplus Y'}{\dg{\gamma,g'}}{Z'}{s_\ast \delta'}$ are $\mathbb E$-triangles, $\eta = \alpha_\ast \delta=\gamma^\ast \delta'$, and $\beta'=\beta_2\beta_1$. By \Cref{prop:hs-inflation}, $\beta_1$ and $\beta_2$ are $\mathbb E$-inflations and $\mathbb E$-deflations. Thus, $\beta'=\beta_2\beta_1$ is an $\mathbb E$-inflation and an $\mathbb E$-deflation, by ET4 and ET4$^{\rm op}$, respectively.

\vskip5pt

$(5)$ \ \ Assume that $(\alpha, \beta, \gamma)$ is a morphism of $\mathbb E$-triangles such that $\alpha$ and $\beta$ are $\mathbb E$-inflations and $\mathbb E$-deflations. By \Cref{thm:hs-morphism} there is $\gamma': Z\to Z'$ such that $(\alpha,\beta,\gamma')$ is a homotopic morphism. That is, there is a decomposition
		\begin{equation*}
\xymatrix@R=0.4cm{
  X\ar[r]^{f}\ar[d]_-{\alpha} & Y\ar[r]^{g}\ar[d]_{{\beta_1}} & Z\ar@{-->}[r]^{\delta}\ar@{=}[d] & {} \\
  X'\ar[r]^{s}\ar@{=}[d] & E\ar[r]^{t}\ar[d]_{{\beta_2}} & Z\ar@{-->}[r]^{{\eta}}\ar[d]^-{\gamma'} & {} \\
  X'\ar[r]^{{f'}} & Y'\ar[r]^{{g'}} & Z'\ar@{-->}[r]^{{\delta'}} & {}
}
\end{equation*}
    where $\con{X}{\dg{f\\\alpha}}{Y \oplus X'}{\dg{\beta_1,-s}}{E}{t^\ast \delta}$ and $\con{E}{\dg{-t\\\beta_2}}{Z \oplus Y'}{\dg{\gamma' ,g'}}{Z'}{s_\ast \delta'}$ are $\mathbb E$-triangles, $\eta = \alpha_\ast \delta=(\gamma')^\ast \delta'$, and $\beta=\beta_2\beta_1$. By \Cref{prop:hs-inflation}, $\beta_1$ is an $\mathbb E$-inflation. By assumption $\beta=\beta_2\beta_1$ is an $\mathbb E$-deflation.
    It follows from \Cref{prop:comp-infldefl}$(1)'$ that $\beta_2$ is an $\mathbb E$-deflation. By \Cref{prop:hs-inflation}, $\gamma'$ is an $\mathbb E$-deflation.
    \vskip5pt

$(9)$ \ Let $\mathcal{C}$ be a weakly idempotent complete extriangulated category.
Assume that $(\alpha, \beta, \gamma)$ is a morphism of $\mathbb E$-triangles such that $\alpha$ is an $\mathbb E$-deflation and $\beta$ is an $\mathbb E$-inflation and an $\mathbb E$-deflation.
By \Cref{thm:hs-morphism} there is $\gamma': Z\to Z'$ such that $(\alpha,\beta,\gamma')$ is a homotopic morphism. That is, there is a decomposition
		\begin{equation*}
\xymatrix@R=0.4cm{
  X\ar[r]^{f}\ar[d]_-{\alpha} & Y\ar[r]^{g}\ar[d]_{{\beta_1}} & Z\ar@{-->}[r]^{\delta}\ar@{=}[d] & {} \\
  X'\ar[r]^{s}\ar@{=}[d] & E\ar[r]^{t}\ar[d]_{{\beta_2}} & Z\ar@{-->}[r]^{{\eta}}\ar[d]^-{\gamma'} & {} \\
  X'\ar[r]^{{f'}} & Y'\ar[r]^{{g'}} & Z'\ar@{-->}[r]^{{\delta'}} & {}
}
\end{equation*}
    where $\con{X}{\dg{f\\\alpha}}{Y \oplus X'}{\dg{\beta_1,-s}}{E}{t^\ast \delta}$ and $\con{E}{\dg{-t\\\beta_2}}{Z \oplus Y'}{\dg{\gamma' ,g'}}{Z'}{s_\ast \delta'}$ are $\mathbb E$-triangles, $\eta = \alpha_\ast \delta=(\gamma')^\ast \delta'$, and $\beta=\beta_2\beta_1$. By \Cref{con:wic}, $\beta_2$ is an $\mathbb E$-deflation. By \Cref{prop:hs-inflation}, $\gamma'$ is an $\mathbb E$-deflation.
    Since $\beta =\beta_2\beta_1$ is an $\mathbb E$-inflation, it follows from \Cref{con:wic} that $\beta_1$ is an $\mathbb E$-inflation, and hence $\alpha$ is an $\mathbb E$-inflation, by \Cref{prop:hs-inflation}.
    \end{proof}

\vskip5pt

As promised, we give counterexamples for the remaining $12$ cases: we will only write down the $7$ cases among them, the other $5$ dual cases are omitted.

\vskip5pt

\begin{example}\label{ex:no-good-completion} \ Let $(\alpha,\beta,\gamma)$ be a morphism of $\mathbb E$-triangles, as in \Cref{thm:hmor}.
	
\vskip5pt

${\rm (1)}$ \  Even if $\alpha$ is an $\mathbb E$-inflation and $\gamma$ is an $\mathbb E$-deflation, there may not exist an $\mathbb E$-inflation nor an $\mathbb E$-deflation $\beta': Y\to Y'$ such that $(\alpha,\beta',\gamma)$ is a morphism of $\mathbb E$-triangles.

  \vskip5pt

${\rm (2)}$ \  Even if $\alpha$ is an $\mathbb E$-deflation and $\gamma$ is an $\mathbb E$-inflation, there may not exist an $\mathbb E$-inflation nor an $\mathbb E$-deflation $\beta': Y\to Y'$ such that $(\alpha,\beta',\gamma)$ is a morphism of $\mathbb E$-triangles.

\vskip5pt

We only illustrate ${\rm (1)}$, and ${\rm (2)}$ is similar. The following is an example in $\mathbf{Ab}:$
	\begin{equation*}
		\xymatrix@R=0.4cm{
  0\ar[r]\ar[d] & \mathbb Z / 2 \mathbb Z\ar@{=}[r]\ar@{..>}[d] & \mathbb Z / 2 \mathbb Z\ar@{-->}[r]^-{0}\ar[d] & {} \\
  \mathbb Z\ar@{=}[r] & \mathbb Z\ar[r] & 0\ar@{-->}[r]^-{0} & {}
		}
	\end{equation*}

\vskip5pt

${\rm (3)}$ \ Even if $\alpha$ is an $\mathbb E$-inflation and an $\mathbb E$-deflation and $\beta$ is an $\mathbb E$-inflation, there may not exist an $\mathbb E$-inflation nor an $\mathbb E$-deflation $\gamma': Z\to Z'$ such that $(\alpha,\beta,\gamma')$ is a morphism of $\mathbb E$-triangles.
\vskip5pt

${\rm (4)}$ \ Even if $\alpha$ and $\beta$ are $\mathbb E$-inflations, there may not exist an $\mathbb E$-inflation nor an $\mathbb E$-deflation $\gamma': Z\to Z'$ such that $(\alpha,\beta,\gamma')$ is a morphism of $\mathbb E$-triangles.
\vskip5pt

${\rm (5)}$ \ Even if $\alpha$ is an $\mathbb E$-deflation and $\beta$ is an $\mathbb E$-inflation, there may not exist an $\mathbb E$-inflation nor an $\mathbb E$-deflation $\gamma': Z\to Z'$ such that $(\alpha,\beta,\gamma')$ is a morphism of $\mathbb E$-triangles.
\vskip5pt

We only show  ${\rm (3)}$:  $(4)$ and $(5)$ are consequences of $(3)$. Let $D^b(R\mbox{-}{\rm Mod})$ be the bounded derived category of a ring $R$, and $\mathcal{D}$ the full subcategory of complexes with cohomology concentrated in degrees $0$ and $-1$.
Then $\mathcal D$ is an extriangulated category (\cite[Remark 2.18]{NP19}). Note that $\Sigma R \oplus \Sigma^{-1}R \xrightarrow {\dg{1,0}} \Sigma R \xrightarrow 0 R \xrightarrow {\dg{0 \\ 1}} \Sigma^2R \oplus R$ is a distinguished triangle in $D^b(R\mbox{-}{\rm Mod})$, where neither $\Sigma R \oplus \Sigma^{-1}R$ nor $\Sigma^2R \oplus R$ belongs to $\mathcal{D}$. Hence $0: \Sigma R\to R$ is neither an $\mathbb E$-inflation nor an $\mathbb E$-deflation in $\mathcal D$.

\vskip5pt

Consider the morphism of $\mathbb E$-triangles given by the following diagram:
\begin{equation}\label{eg:3x3-fail}
\xymatrix@R=0.4cm{
  0\ar[r]\ar[d] & \Sigma R\ar@{=}[r]\ar[d]^{{\binom 10}} & \Sigma R\ar@{-->}[r]^{0}\ar[d]^{0} & {} \\
  \Sigma R\ar[r]^(0.4){{\binom 10}} & \Sigma R \oplus  R\ar[r]^(0.6){{(0,1)}} & R\ar@{-->}[r]^{0} & {}
}
\end{equation}
In this case, $0: 0\to \Sigma R$ is an $\mathbb E$-inflation and an $\mathbb E$-deflation, $\binom 10 :\Sigma R\to \Sigma R\oplus R$ is an $\mathbb E$-inflation; but any candidate $\gamma': \Sigma R\to R$ must be $0$.

\vskip5pt

${\rm (6)}$  \ Even if $\alpha$ is an $\mathbb E$-deflation and $\beta$ is an $\mathbb E$-inflation and $\mathbb E$-deflation, there may not exist an $\mathbb E$-inflation nor an $\mathbb E$-deflation $\gamma': Z\to Z'$ such that $(\alpha,\beta,\gamma')$ is a morphism of $\mathbb E$-triangles.
\vskip5pt

${\rm (7)}$  \ Even if $\alpha$ and $\beta$ are $\mathbb E$-deflations, there may not exist an $\mathbb E$-inflation nor an $\mathbb E$-deflation $\gamma': Z\to Z'$ such that $(\alpha,\beta,\gamma')$ is a morphism of $\mathbb E$-triangles.
\vskip5pt

We only show ${\rm (6)}$: $(7)$ is a consequence of $(6)$. Let $\mathcal A$ be a semisimple abelian category with only two indecomposable objects $X$ and $Y$. Let $D^b (\mathcal A)$ be its bounded derived category, and 
$\mathcal D=\{ Z\in D^b (\mathcal A)\mid {\rm H}^0(Z)\ne X^m, m>0,  {\rm H}^1(Z)=X^n, n\geq 0, {\rm H}^i(Z)=0, i\ne 0,1\}$. Then $\mathcal{D}$ is closed under extensions, and thus an extriangulated category (\cite[Remark 2.18]{NP19}). Consider the morphism of split $\mathbb E$-triangles:
  \begin{equation*}
    \xymatrix@R=0.8cm@C=1.2cm{
X\oplus Y \ar[r]^-{\scalebox{0.8}{$\dg{1&0\\0&1\\0&0\\0&0}$}} \ar[d]^-\alpha_-{\scalebox{0.8}{$\dg{1&0\\0&1\\0&0}$}} & X\oplus Y\oplus X\oplus Y \ar[r]^-{\scalebox{0.8}{$\dg{0&0&1&0\\0&0&0&1}$}} \ar@{=}[d]^-\beta & X\oplus Y \ar@{-->}[r]^-{0} \ar@{.>}[d]^{\gamma'} & {} \\
X\oplus Y\oplus X \ar[r]^-{\scalebox{0.8}{$\dg{1&0&0\\0&1&0\\0&0&1\\0&0&0}$}} & X\oplus Y\oplus X\oplus Y  \ar[r]^-{\scalebox{0.8}{$\dg{0&0&0&1}$}} & Y \ar@{-->}[r]^-{0} & {}
    }
  \end{equation*}
  Since $\Sigma^{-1} X\in \mathcal D$, $\alpha$ is an $\mathbb E$-deflation. Note that $\gamma'$ has to be a split epimorphism, and $X, \Sigma X\notin \mathcal D$. Thus $\gamma'$ is neither an $\mathbb E$-inflation nor an $\mathbb E$-deflation.

\vskip5pt

\end{example}

\vskip5pt

For a better global view, we combine \Cref{thm:hmor}, \Cref{ex:no-good-completion}, and their duals,  in the following table of homotopic morphisms. Let $(\alpha, \beta, \gamma)$ be a homotopic morphism of $\mathbb E$-triangles.

\begin{equation*}
\xymatrix@R=0.4cm{
  X\ar[r]^{f}\ar[d]_-{\alpha} & Y\ar[r]^{g}\ar[d]_-{\beta} & Z\ar@{-->}[r]^{\delta}\ar[d]_-{\gamma} & {} \\
  X'\ar[r]^{{f'}} & Y'\ar[r]^{{g'}} & Z'\ar@{-->}[r]^{{\delta'}} & .
}
\end{equation*}
In the following table, i, d and i-d are abbreviations of $\mathbb E$-inflation, $\mathbb E$-deflation, and $\mathbb E$-inflation and $\mathbb E$-deflation respectively. For instance,
(i,?,d) means that in a homotopic morphism $(\alpha, \beta, \gamma)$, $\alpha$ is an $\mathbb E$-inflation and $\gamma$ is an $\mathbb E$-deflation, whether $\beta$ is in the set $$\{\text{$\mathbb E$-inflation and $\mathbb E$-deflation, $\mathbb E$-inflation,  $\mathbb E$-deflation}\}.$$
The answer ``implicit'' means that $\beta$ is neither an $\mathbb E$-inflation nor an $\mathbb E$-deflation, in general.

\vskip10pt
\centerline {\bf Table of homotopic morphisms  $(\alpha, \beta, \gamma)$}
\vskip5pt
\begin{center}
    \renewcommand{\arraystretch}{1.3}
    \scriptsize
    \begin{tabular}{|c|c|c|c|c|}
      \hline
      \textbf{Prerequisite} & \textbf{Case} & $\boldsymbol{(\alpha, \beta, \gamma)}$ & \textbf{Completion} & \textbf{Reference} \\
      \hline
           \multirow{9}{*}{\makecell{Given $(\alpha,\gamma)$ \\\\ with $\alpha_\ast \delta = \gamma^\ast \delta'$}}
           & $(1)$ & (i-d, ?, i-d) & $\beta$ is i-d & \multirow{4}{*}{\Cref{thm:hmor}} \\
            \cline{2-3}\cline{4-4}
           & $(2)$ & (i-d, ?, i) & \multirow{3}{*}{$\beta$ is i} &  \\
            \cline{2-3}
           & $(3)$ & (i, ?, i-d) & & \\
            \cline{2-3}
           & $(4)$ & (i, ?, i) & &  \\
            \cline{2-3}\cline{4-4}\cline{5-5}
           & $(2)'$ & (d, ?, i-d) &  \multirow{3}{*}{$\beta$ is d} &  \multirow{3}{*}{The dual of \Cref{thm:hmor}} \\
            \cline{2-3}
           & $(3)'$ & (i-d, ?, d) & & \\
            \cline{2-3}
           & $(4)'$ & (d, ?, d) & & \\
            \cline{2-3}\cline{4-4}\cline{5-5}
           & $(5)$ & (i, ?, d) & \multirow{2}{*}{Implicit} & \multirow{2}{*}{\Cref{ex:no-good-completion}} \\
            \cline{2-3}
           & $(6)$ & (d, ?, i) &  &  \\
            \hline
          \multirow{9}{*}{\makecell{Given $(\alpha, \beta)$ \\\\ with $f'\alpha=\beta f$}}
           & $(7)$ & (i-d, i-d, ?) & \multirow{4}{*}{$\gamma$ is d} & \multirow{4}{*}{\Cref{thm:hmor}} \\
            \cline{2-3}
           & $(8)$ & (i-d, d, ?) &  &  \\
            \cline{2-3}
           & $(9)$ & (i, i-d, ?) &  &  \\
            \cline{2-3}
           & $(10)$ & (i, d, ?) &  & \\
            \cline{2-3}\cline{4-4}\cline{5-5}
           & $(11)$ & (i-d, i, ?) & \multirow{5}{*}{Implicit} & \multirow{5}{*}{\Cref{ex:no-good-completion}} \\
            \cline{2-3}
           & $(12)$ & (i, i, ?) &  & \\
            \cline{2-3}
           & $(13)$ & (d, i, ?) &  &\\
            \cline{2-3}
           & $(14)$ & (d, i-d, ?) &  &  \\
            \cline{2-3}
           & $(15)$ & (d, d, ?) &  &  \\
            \hline
          \multirow{9}{*}{\makecell{Given $(\beta, \gamma)$ \\\\with $g'\beta=\gamma g$}}
           & $(7)'$ & (?, i-d, i-d) & \multirow{4}{*}{$\alpha$ is i} & \multirow{4}{*}{The dual of \Cref{thm:hmor}} \\
            \cline{2-3}
           & $(8)'$ & (?, i, i-d) &  &  \\
            \cline{2-3}
           & $(9)'$ & (?, i-d, d) &  &\\
            \cline{2-3}
           & $(10)'$ & (?, i, d) &  &\\
            \cline{2-3}\cline{4-4}\cline{5-5}
           & $(11)'$ & (?, d, i-d) & \multirow{5}{*}{Implicit} &  \multirow{5}{*}{The dual of \Cref{ex:no-good-completion}} \\
            \cline{2-3}
           & $(12)'$ & (?, d, d) &  & \\
            \cline{2-3}
           & $(13)'$ & (?, d, i) &  &\\
            \cline{2-3}
           & $(14)'$ & (?, i-d, i) &  &   \\
            \cline{2-3}
           & $(15)'$ & (?, i, i) &  &  \\
            \hline
    \end{tabular}
  \end{center}
\vskip5pt

\subsection{More properties of homotopic morphisms} Different from ET4, in the following result, the morphism $w$ is given and $(1, u, w)$ is a homotopic morphism of $\mathbb E$-triangles.

\begin{proposition}\label{prop:strict-et4-h}\label{prop:strict-et4-2-h} \ Let $(1, u, w)$ be a homotopic morphism of $\mathbb E$-triangles$:$
\begin{equation*}
\xymatrix@R=0.4cm{A\ar[r]^{f}\ar@{=}[d] & B\ar[r]^{g}\ar[d]_{u} & D\ar@{-->}[r]^{{w^*\theta}}\ar[d]^{w} & {} \\
  A\ar[r]^{uf} & C\ar[r]^{h} & F\ar@{-->}[r]^{\theta} & {}
}
\end{equation*}
such that both $u$ and $w$ are $\mathbb E$-inflations.

\vskip5pt
$(1)$ \  For any $\mathbb E$-triangle $\con {B}u{C}vE\varepsilon$, there is an $\mathbb E$-triangle $\con {D}w{F}{q}E{\eta}$ such that $v=qh$, $\eta=g_\ast \varepsilon$ and $f_\ast \theta =q^\ast \varepsilon$.
\vskip5pt
$(2)$ \  For any $\mathbb E$-triangle $\con {D}w{F}{q}E{\eta}$, there is an $\mathbb E$-triangle $\con {B}u{C}vE\varepsilon$ such that $v=qh$, $\eta=g_\ast \varepsilon$ and $f_\ast \theta =q^\ast \varepsilon$.

\begin{equation*}
\xymatrix@R=0.4cm{
  A\ar[r]^{f}\ar@{=}[d] & B\ar[r]^{g}\ar[d]_{u} & D\ar@{-->}[r]^{{w^*\theta}}\ar[d]^{w} & {} \\
  A\ar[r]^{uf} & C\ar[r]^{h}\ar[d]_{v} & F\ar@{-->}[r]^{\theta}\ar[d]^{q} & {} \\
  {} & E\ar@{=}[r]\ar@{-->}[d]_{\varepsilon} & E\ar@{-->}[d]^{\eta} & {} \\
  {} & {} & {} & {}
}
\end{equation*}
\end{proposition}

	\begin{proof} \ $(1)$ \  By \Cref{prop:pb-2} there is an $\mathbb E$-triangle $\con D{w'}FqE\eta$ which fits the following commutative diagram, satisfying $f_\ast \theta=q^\ast \varepsilon$ and $\dg{1 \\ 0}_\ast \eta + \dg{-g \\ u}_\ast \varepsilon = 0$.
\begin{equation*}
\xymatrix@R=0.6cm{
  {} & D\ar@{=}[r]\ar[d]_{{\binom 10}} & D\ar@{..>}[d]^{{w'}} & {} \\
  B\ar[r]^(0.4){\scalebox{0.8}{${\binom {-g}u}$}}\ar@{=}[d] & D\oplus C\ar[r]^-{{(w, h)}}\ar[d]_{{(0,1)}} & F\ar@{-->}[r]^{{f_\ast \theta}}\ar@{..>}[d]^{q} & {} \\
  B\ar[r]^{u} & C\ar[r]^{v}\ar@{-->}[d]_{0} & E\ar@{-->}[r]^{\varepsilon}\ar@{..>}[d]^{{\eta }} & {} \\
  {} & {} & {} & {}
}
\end{equation*}
Now $w'=\dg{w , h}\dg{1\\0}=w$. And $q \dg{w , h}= v \dg{0 , 1}$ implies that $v=qh$. Finally, $\eta -g_\ast \varepsilon=\dg{1 , 0}_\ast \left(\dg{1 \\ 0}_\ast \eta + \dg{-g \\ u}_\ast \varepsilon \right) = 0$. This completes the proof.

\vskip5pt
$(2)$ \  By \Cref{prop:pb-2} there is an $\mathbb E$-triangle $\con B{u'}CvE\varepsilon$ which fits the following commutative diagram satisfying $f_\ast \theta=q^\ast \varepsilon$ and $\dg{1 \\ 0}_\ast \eta + \dg{-g \\ u}_\ast \varepsilon = 0$.
\begin{equation*}
\xymatrix@R=0.6cm{
  {} & B\ar@{=}[r]\ar[d]_-{\scalebox{0.8}{${\dg {{-g} \\ u}}$}} & B\ar@{..>}[d]^-{{u'}} & {} \\
  D\ar[r]^(0.4){\scalebox{0.8}{${\dg {1\\ 0}}$}}\ar@{=}[d] & D \oplus C\ar[r]^(0.6){{\dg{0,1}}}\ar[d]_-{{\dg{w, h}}} & C \ar@{-->}[r]^{0}\ar@{..>}[d]^{v} & {} \\
  D\ar[r]^-{w} & F\ar[r]^{q}\ar@{-->}[d]_{{f_\ast \theta}} & E\ar@{-->}[r]^{{\eta }}\ar@{..>}[d]^{\varepsilon} & {} \\
  {} & {} & {} & {}
}
\end{equation*}
Now $u'=\dg{0 , 1}\dg{-g\\u}=u$. And $q \dg{w , h}= v \dg{0 , 1}$ implies that $v=qh$. Finally, $\eta -g_\ast \varepsilon=\dg{1 , 0}_\ast \left(\dg{1 \\ 0}_\ast \eta + \dg{-g \\ u}_\ast \varepsilon \right) = 0$. This completes the proof.
\end{proof}

Different from ET4$^{\rm op}$, in the dual of \Cref{prop:strict-et4-h}, the morphism $g$ is given and $(g, h, 1)$ is a homotopic morphism of $\mathbb E$-triangles.

\vskip5pt
\begin{proposition}\label{prop:strict-et4op-h}\label{prop:strict-et4op-2-h}
	Let  $(g, h, 1)$ be a homotopic morphism of $\mathbb E$-triangles$:$
\begin{equation*}
\xymatrix@R=0.4cm{
   B\ar[r]^{g}\ar[d]_{u} & D\ar[d]^{w} \\
   C\ar[r]^{h}\ar[d]_{qh} & F\ar[d]^{q} \\
   E\ar@{=}[r]\ar@{-->}[d]_{\varepsilon} & E\ar@{-->}[d]^{g_\ast\varepsilon} \\
   {} & {}
}
\end{equation*}
	such that both $g$ and $h$ are $\mathbb E$-deflations.
\vskip5pt
$(1)$ \  For any $\mathbb E$-triangle $\con AmChF\theta$, there is an $\mathbb E$-triangle $\con AfBgD\delta$ such that $m=uf$, $\delta=w^\ast \theta$ and $f_\ast \theta =q^\ast \varepsilon$.
\vskip5pt
$(2)$ \  For any $\mathbb E$-triangle $\con AfBgD\delta$, there is an $\mathbb E$-triangle $\con AmChF\theta$ such that $m=uf$, $\delta=w^\ast \theta$ and $f_\ast \theta =q^\ast \varepsilon$.

\begin{equation*}
\xymatrix@R=0.4cm{
  A\ar[r]^{f}\ar@{=}[d] & B\ar[r]^{g}\ar[d]_{u} & D\ar@{-->}[r]^{{\delta}}\ar[d]^{w} & {} \\
  A\ar[r]^{m} & C\ar[r]^{h}\ar[d]_{qh} & F\ar@{-->}[r]^{\theta}\ar[d]^{q} & {} \\
  {} & E\ar@{=}[r]\ar@{-->}[d]_{\varepsilon} & E\ar@{-->}[d]^{g_\ast\varepsilon} & {} \\
  {} & {} & {} & {}
}
\end{equation*}
\end{proposition}

\vskip5pt
Different from \cite[Proposition 3.17]{NP19}, in the following result, the morphism $f_2$ is given and $(f_2, e_1, 1)$ is a homotopic morphism of $\mathbb E$-triangles.
\vskip5pt
\begin{proposition}\label{prop:po-1-h}\label{prop:po-2-h}
	Let $(f_2, e_1, 1)$ be a homotopic morphism of $\mathbb E$-triangles$:$
\begin{equation*}
\xymatrix@R=0.4cm{
  A\ar[r]^{f_1}\ar[d]_{f_2} & B_1\ar[r]^{p_1 e_1}\ar[d]^{e_1} & C_1\ar@{-->}[r]^{\delta_1}\ar@{=}[d] & {} \\
  B_2\ar[r]^{e_2} & F\ar[r]^{p_1} & C_1\ar@{-->}[r]^{{(f_2)_\ast \delta_1}} & {}
}
\end{equation*}
such that both $f_2$ and $e_1$ are $\mathbb E$-inflations.

\vskip5pt
$(1)$ \  For any $\mathbb E$-triangle $\con {B_1}{e_1}{F}{p_2}{C_2}{\varepsilon}$, there is an $\mathbb E$-triangle $\con {A}{f_2}{B_2}{g_2}{C_2}{\delta_2}$ such that $g_2=p_2e_2$, $\varepsilon=(f_1)_\ast \delta_2$ and $(p_1)^\ast \delta_1+(p_2)^\ast \delta_2 =0$.

\vskip5pt
$(2)$ \  For any $\mathbb E$-triangle $\con {A}{f_2}{B_2}{g_2}{C_2}{\delta_2}$, there is an $\mathbb E$-triangle $\con {B_1}{e_1}{F}{p_2}{C_2}{\varepsilon}$, such that $g_2=p_2e_2$, $\varepsilon=(f_1)_\ast \delta_2$ and $(p_1)^\ast \delta_1+(p_2)^\ast \delta_2 =0$.

\begin{equation*}
\xymatrix@R=0.4cm{
 A\ar[r]^{f_1}\ar[d]_{f_2} & B_1\ar[r]^{p_1 e_1}\ar[d]^{e_1} & C_1\ar@{-->}[r]^{\delta_1}\ar@{=}[d] & {} \\
  B_2\ar[r]^{e_2}\ar[d]_{g_2} & F\ar[r]^{p_1}\ar[d]^{p_2} & C\ar@{-->}[r]^{{(f_2)_\ast \delta_1}} & {} \\
  C_2\ar@{=}[r]\ar@{-->}[d]_{\delta_2} & C_2\ar@{-->}[d]^{\varepsilon} & {} & {} \\
  {} & {} & {} & {}
}
\end{equation*}
\end{proposition}

	\begin{proof} \ $(1)$ \  By \Cref{prop:pb-2} there is an $\mathbb E$-triangle $\con {A}{f_2'}{B_2}{-g_2}{C_2}{-\delta_2}$ which fits the following commutative diagram satisfying
        $(p_1)^\ast \delta_1=(p_2)^\ast (-\delta_2)$ and $\dg{1 \\ 0}_\ast \varepsilon + \dg{f_1 \\ f_2}_\ast (-\delta_2) = 0$
\begin{equation*}
\xymatrix@R=0.6cm{
  {} & A\ar@{=}[r]\ar[d]_-{\scalebox{0.8}{${\dg{f_1 \\ f_2}}$}} & A\ar@{..>}[d]^-{{f_2'}} & {} \\
  B_1\ar[r]^(0.4){\scalebox{0.8}{${\dg{1 \\ 0}}$}}\ar@{=}[d] & B_1\oplus B_2\ar[r]^(0.6){{\dg{0,1}}}\ar[d]_-{{\dg{e_1,-e_2}}} & B_2\ar@{-->}[r]^{0}\ar@{..>}[d]^-{{-g_2}} & {} \\
  B_1\ar[r]^{e_1} & F\ar[r]^{p_2}\ar@{-->}[d]_{{(p_1)^\ast \delta_1}} & C_2\ar@{-->}[r]^{{\varepsilon }}\ar@{..>}[d]^{{-\delta_2}} & {} \\
  {} & {} & {} & {}
}
\end{equation*}
Now $f_2'=\dg{0 , 1}\dg{f_1\\f_2}=f_2$. Then $-g_2 \dg{0 , 1}= p_2 \dg{e_1 , -e_2}$ implies that $g_2=p_2e_2$, and $(p_1)^\ast \delta_1=(p_2)^\ast (-\delta_2)$ implies that $(p_1)^\ast \delta_1+(p_2)^\ast \delta_2=0$. Finally, $\varepsilon-(f_1)_\ast \delta_2=\dg{1 , 0}_\ast \left(\dg{1 \\ 0}_\ast \varepsilon + \dg{f_1 \\ f_2}_\ast (-\delta_2)\right) = 0$. This completes the proof.

\vskip5pt
$(2)$ \
		By \Cref{prop:pb-2} there is an $\mathbb E$-triangle $\con {B_1}{e_1'}{F}{-p_2}{C_2}{-\varepsilon}$ which fits the following commutative diagram satisfying
        $(p_1)^\ast \delta_1=(p_2)^\ast (-\delta_2)$ and $\dg{f_1 \\ f_2}_\ast \delta_2 + \dg{1 \\ 0}_\ast (-\varepsilon) = 0:$
\begin{equation*}
\xymatrix@R=0.7cm{
  {} & B_1\ar@{=}[r]\ar[d]_-{{\dg{1 \\ 0}}} & B_1\ar@{..>}[d]^-{{e_1'}} & {} \\
  A\ar[r]^(0.4){{\dg{f_1 \\ f_2}}}\ar@{=}[d] & B_1\oplus B_2\ar[r]^(0.6){{\dg{e_1,-e_2}}}\ar[d]_-{{\dg{0,1}}} & F\ar@{-->}[r]^{{(p_1)^\ast \delta_1}}\ar@{..>}[d]^-{{-p_2}} & {} \\
  A\ar[r]^{f_2} & B_2\ar[r]^{g_2}\ar@{-->}[d]_{0} & C_2\ar@{-->}[r]^{\delta_2}\ar@{..>}[d]^{{-\varepsilon }} & {} \\
  {} & {} & {} & {}
}
\end{equation*}
		Now $e_1' = \dg{e_1, -e_2} \dg{1 \\ 0}=e_1$. Then $g_2 \dg{0 , 1}=(- p_2) \dg{e_1 , -e_2}$ implies that $g_2=p_2e_2$ and $(p_1)^\ast \delta_1=(p_2)^\ast (-\delta_2)$ implies that $(p_1)^\ast \delta_1+(p_2)^\ast \delta_2=0$. Finally, $\varepsilon-(f_1)_\ast \delta_2=\dg{1 , 0}_\ast \left(\dg{1 \\ 0}_\ast \varepsilon + \dg{f_1 \\ f_2}_\ast (-\delta_2)\right) = 0$. This completes the proof.
	\end{proof}
\vskip5pt
Different from \Cref{prop:pb-2}, in the following result, the morphism $g_2$ is given and $(1, p_1, g_2)$ is a homotopic morphism of $\mathbb E$-triangles.
\vskip5pt

\begin{proposition}\label{prop:pb-1-h}\label{prop:pb-2-h}
	Let $(1, p_1, g_2)$ be a homotopic morphism of $\mathbb E$-triangles$:$
	\begin{equation*}
\xymatrix@R=0.4cm{
  A_1\ar[r]^{e_1}\ar@{=}[d] & E\ar[r]^{p_2}\ar[d]_{p_1} & B_2\ar@{-->}[r]^{{(g_2)^\ast \delta_1}}\ar[d]^{g_2} & {} \\
  A_1\ar[r]^{p_1e_1} & B_1\ar[r]^{g_1} & C\ar@{-->}[r]^{\delta_1} & {}
}
\end{equation*}
	such that both $p_1$ and $g_2$ are $\mathbb E$-deflations.
\vskip5pt
$(1)$ \  For any $\mathbb E$-triangle $\con {A_2}{e_2}{E}{p_1}{B_1}\varepsilon$, there is an $\mathbb E$-triangle $\con {A_2}{f_2}{B_2}{g_2}{C}{\delta_2}$ such that $f_2=p_2e_2$, $\varepsilon=(g_1)^\ast \delta_2$ and $(e_1)_\ast\delta_1+(e_2)_\ast\delta_2= 0$.
\vskip5pt
$(2)$ \  For any $\mathbb E$-triangle $\con {A_2}{f_2}{B_2}{g_2}{C}{\delta_2}$, there is an $\mathbb E$-triangle $\con {A_2}{e_2}{E}{p_1}{B_1}\varepsilon$ such that $f_2=p_2e_2$, $\varepsilon=(g_1)^\ast \delta_2$ and $(e_1)_\ast\delta_1+(e_2)_\ast\delta_2= 0$.

	\begin{equation*}
\xymatrix@R=0.4cm{
  {} & A_2\ar@{=}[r]\ar[d]_{e_2} & A_2\ar[d]_{f_2} & {} \\
  A_1\ar[r]^{e_1}\ar@{=}[d] & E\ar[r]^{p_2}\ar[d]_{p_1} & B_2\ar@{-->}[r]^{{(g_2)^\ast \delta_1}}\ar[d]_{g_2} & {} \\
  A_1\ar[r]^{f_1} & B_1\ar[r]^{g_1}\ar@{-->}[d]_{{\varepsilon }} & C\ar@{-->}[r]^{\delta_1}\ar@{-->}[d]_{\delta_2} & {} \\
  {} & {} & {} & {}
}
\end{equation*}
\end{proposition}

\section{Semi-homotopic morphisms and $6$-term exact sequences}\label{sec_ce}

\subsection{$6$-term exact sequences}\label{sec:hierarchy} To investigate the behavior of morphisms of $\mathbb E$-triangles which are not necessarily homotopic, we concentrate on the morphisms of the form $(\alpha, \beta, 1):$
\begin{equation}\label{eq:hs-fg1}
\xymatrix@R=0.4cm{
  X\ar[r]^{f}\ar[d]_-{\alpha} & Y\ar[r]^{g}\ar[d]^-{\beta} & Z\ar@{-->}[r]^{\delta}\ar@{=}[d] & {} \\
  X'\ar[r]^{{f'}} & Y'\ar[r]^{{g'}} & Z\ar@{-->}[r]^{{\alpha_\ast \delta}} &. {}
}
\end{equation}

\vskip5pt

If $(\alpha,\beta,1)$ is a homotopic morphism,  then by \Cref{prop:mor-of-confl} there is an $\mathbb E$-triangle  $\con X{\dg{f\\\alpha}}{Y\oplus X'}{\dg{\beta,-f'}}{Y'}{(g')^*\delta}$, and hence by \Cref{lem:long-ext-seq} one has the $6$-term exact sequences$:$
    \begin{equation*}    \resizebox{1\textwidth}{!}{
      $\mathcal C(Y',-)\xrightarrow{\mathcal C(\dg{\beta,-f'},-)}\mathcal C(Y\oplus X',-)\xrightarrow{\mathcal C(\dg{f \\\alpha},-)}\mathcal C(X,-)\xrightarrow{((g')^*\delta)^\sharp}\mathbb E(Y',-)\xrightarrow{\dg{\beta,-f'}^\ast}\mathbb E(Y \oplus X',-)\xrightarrow{\dg{f \\\alpha}^\ast}\mathbb E(X,-)$
      }
      \end{equation*}
      and
    \begin{equation*}    \resizebox{1\textwidth}{!}{
      $\mathcal C(-,X)\xrightarrow{\mathcal C(-,\dg{f \\\alpha})}\mathcal C(-,Y\oplus X')\xrightarrow{\mathcal C(-,\dg{\beta,-f'})}\mathcal C(-,Y')\xrightarrow{((g')^*\delta)_\sharp}\mathbb E(-,X)\xrightarrow{\dg{f \\\alpha}_\ast}\mathbb E(-,Y\oplus X')\xrightarrow{\dg{\beta,-f'}_\ast}\mathbb E(-,Y').$}
\end{equation*}
In general, $\con X{\dg{f \\\alpha}}{Y\oplus X'}{\dg{\beta,-f'}}{Y'}{(g')^\ast\delta}$ is not an $\mathbb E$-triangle; however, we will see that
it still induces a $6$-term exact sequence of contravariant functors;
while the other sequence of covariant functors is not exact in general.

\vskip5pt

\begin{lemma}\label{lem:hs-fg1} \ Let $(\mathcal{C}, \mathbb E, \mathfrak s)$ be an extriangulated category, $(\alpha, \beta, 1)$ a morphism of $\mathbb E$-triangles as in \Cref{eq:hs-fg1}, and $\con X{\dg{f \\\alpha}}{Y\oplus X'}{\dg{\beta,-f'}}{Y'}{(g')^\ast\delta}$ the associated sequence. Then $\dg{\beta,-f'}\dg{f\\\alpha} = 0$, $\dg{\beta,-f'}^\ast((g')^\ast\delta) = 0$ and $\dg{f\\\alpha}_\ast((g')^\ast\delta) = 0$.
\end{lemma}

	\begin{proof} \ The commutative diagram \Cref{eq:hs-fg1} shows the first identity $\dg{\beta,-f'}\dg{f\\\alpha} = 0$. By \Cref{lem:hs-f?1}$(1)$ there exists $\beta': Y\to Y'$ such that $\con X{\dg{f \\\alpha}}{Y \oplus X'}{\dg{\beta',-f'}}{Y'}{(g')^\ast \delta}$ is an $\mathbb E$-triangle realizing $(g')^\ast \delta$. By \Cref{lem:long-ext-seq} one has $\dg{ f\\\alpha}_\ast ((g')^\ast \delta) = 0$.
Similarly, by \Cref{lem:hs-f?1}$(2)$ there exists $\alpha': X \to X'$ such that $\con X{\dg{f \\\alpha'}}{Y \oplus X'}{\dg{\beta,-f'}}{Y'}{(g')^\ast \delta}$ is an $\mathbb E$-triangle realizing $(g')^\ast \delta$.
Again by \Cref{lem:long-ext-seq} one has $\dg{\beta,-f'}^\ast ((g')^\ast \delta) = 0$.
	\end{proof}

\vskip5pt

\begin{proposition}\label{prop:quasi-long-ext} \  Let $(\mathcal{C}, \mathbb E, \mathfrak s)$ be an extriangulated category.

\vskip5pt

$(1)$ \ Let $(\alpha, \beta, 1)$ be a morphism of $\mathbb E$-triangles as in \Cref{eq:hs-fg1}. Then the sequence  $\con X{\dg{f \\\alpha}}{Y\oplus X'}{\dg{\beta,-f'}}{Y'}{(g')^\ast\delta}$ induces a $6$-term exact sequence of functors
    \begin{equation*}    \resizebox{1\textwidth}{!}{
      $\mathcal C(Y',-)\xrightarrow{\mathcal C(\dg{\beta,-f'},-)}\mathcal C(Y\oplus X',-)\xrightarrow{\mathcal C(\dg{f \\\alpha},-)}\mathcal C(X,-)\xrightarrow{((g')^\ast \delta)^\sharp}\mathbb E(Y',-)\xrightarrow{\dg{\beta,-f'}^\ast}\mathbb E(Y \oplus X',-)\xrightarrow{\dg{f \\\alpha}^\ast}\mathbb E(X,-).$
      }
\end{equation*}

\vskip5pt

$(1)'$ \ Let $(1, \beta, \gamma)$ be a morphism of $\mathbb E$-triangles$:$
    \begin{equation}\label{eq:con-mor-o}
\xymatrix@R=0.4cm{
  X\ar[r]^{f}\ar@{=}[d] & Y\ar[r]^{g}\ar[d]_-{\beta} & Z\ar@{-->}[r]^{{\gamma^*\varepsilon}}\ar[d]^-{\gamma} & {} \\
  X\ar[r]^{f'} & Y'\ar[r]^{g'} & Z'\ar@{-->}[r]^{\varepsilon} &. {}
}
\end{equation}
Then the sequence $\con Y{\dg{-g \\\beta}}{Z\oplus Y'}{\dg{\gamma,g'}}{Z'}{f_\ast\varepsilon}$ induces a $6$-term exact sequence of contravariant functors
    \begin{equation*}    \resizebox{1\textwidth}{!}{
      $\mathcal C(-,Y)\xrightarrow{\mathcal C(-,\dg{-g \\\beta})}\mathcal C(-,Z\oplus Y')\xrightarrow{\mathcal C(-,\dg{\gamma,g'})}\mathcal C(-,Z')\xrightarrow{(f_\ast\varepsilon)_\sharp}\mathbb E(-,Y)\xrightarrow{\dg{-g \\\beta}_\ast}\mathbb E(-,Z\oplus Y')\xrightarrow{\dg{\gamma,g'}_\ast}\mathbb E(-,Z').$
      }
\end{equation*}

\end{proposition}

\begin{proof} \ We only prove $(1)$ by checking exactness at each position.
\vskip5pt
 By \Cref{lem:hs-fg1} one has $\dg{\beta,-f'}\dg{f \\\alpha} =0$. Thus $\operatorname{Ker} \mathcal{C}(\dg{f \\\alpha}, -) \supseteq \operatorname{Im} \mathcal{C}(\dg{\beta,-f'}, -)$. Conversely, one takes $\dg{a,b}: Y\oplus X'\to T$ such that $\dg{a,b}\dg{f\\\alpha} = 0$. Since $b_\ast (\alpha_\ast \delta) = (b\alpha)_\ast \delta = (-af)_\ast \delta = -a_\ast(f_\ast \delta )= 0$, there exists $s:Y'\to T$ such that $sf' = b$ by \Cref{lem:long-ext-seq}, as shown in the following commutative diagram:
			\begin{equation*}
\xymatrix@R=0.4cm{
  X\ar[r]^{f}\ar[d]_-{\alpha} & Y\ar[r]^{g}\ar[d]^-{\beta}\ar@/^1pc/[ddr]^(0.4)*+<-1.5ex>{\scalebox{0.8}{$-a$}} & Z\ar@{-->}[r]^{\delta}\ar@{=}[d] & {} \\
  X'\ar[r]^{{f'}}\ar@/_1pc/[drr]^{b} & Y'\ar[r]^{{g'}}\ar@{..>}[dr]^{s} & Z\ar@{-->}[r]^{{f_\ast \delta}}\ar@{..>}@/^0pc/[d]^{t} & {} \\
  {} & {} & T & {}
}
\end{equation*}
Note that $(a+s\beta)f = af + sf'\alpha = af+ b\alpha = 0$. By \Cref{lem:long-ext-seq} there exists $t: Z\to T$ such that $tg = a + s\beta$. A direct verification shows $\dg{a,b}=(tg'-s)\dg{\beta,-f'}$. Hence $\operatorname{Ker} \mathcal{C}(\dg{f\\\alpha}, -) \subseteq \operatorname{Im} \mathcal{C}(\dg{\beta,-f'}, -)$. Thus one obtains the exactness at $\mathcal{C}(Y \oplus X', -)$.

            \vskip5pt

By \Cref{lem:hs-f?1}$(1)$ there exists $\beta' : Y \to Y'$ such that $\con X{\dg{f\\\alpha}}{Y\oplus X'}{\dg{\beta',-f'}}{Y'}{(g')^\ast\delta}$ is an $\mathbb E$-triangle. By \Cref{lem:long-ext-seq}, ${\mathcal{C}(Y \oplus X', -)} \xrightarrow{\mathcal{C}(\dg{ f\\\alpha}, -)}  {\mathcal{C}(X, -)} \xrightarrow{((g')^\ast \delta)^\sharp} {\mathbb E(Y', -)}$ is exact at $\mathcal{C}(X, -)$. Thus one obtains the exactness at $\mathcal{C}(X, -)$.

            \vskip5pt
			
By \Cref{lem:hs-?g1}$(2)$ there exists $\alpha' : X \to X'$ such that $\con X{\dg{f\\\alpha'}}{Y\oplus X'}{\dg{\beta,-f'}}{Y'}{(g')^\ast\delta}$ is an $\mathbb E$-triangle. By \Cref{lem:long-ext-seq}, ${\mathcal{C}(X, -)} \xrightarrow{((g')^\ast \delta)^\sharp} {\mathbb E(Y', -)} \xrightarrow{\dg{\beta,-f'}^\ast}  {\mathbb E(Y \oplus X', -)}$ is exact at $\mathbb E(Y', -)$. Thus one obtains the exactness at $\mathbb E(Y', -)$.

            \vskip5pt
By \Cref{lem:hs-fg1} one has $\dg{\beta,-f'}\dg{f\\\alpha} = 0$. Thus $\operatorname{Ker} (\dg{f\\\alpha}^\ast) \supseteq \operatorname{Im}(\dg{\beta,-f'}^\ast)$. Conversely, take $\varepsilon \in \mathbb E(Y \oplus X', T)$ such that $\dg{u\\ f}^\ast \varepsilon = 0$. By \Cref{lem:hs-f?1}$(1)$ there exists $\beta': Y \to Y'$ such that $\con X{\dg{f\\\alpha}}{Y\oplus X'}{\dg{\beta',-f'}}{Y'}{(g')^\ast\delta}$ is an $\mathbb E$-triangle.
By \Cref{lem:long-ext-seq}, ${\mathbb E(Y', -)} \xrightarrow{\dg{\beta',-f'}^\ast}  {\mathbb E(Y \oplus X', -)} \xrightarrow{\dg{f\\\alpha}^\ast} \mathbb E(X, -)$ is exact at $\mathbb E(Y \oplus X', -)$. Thus there exists $\eta\in \mathbb E(Y', T)$ such that $\varepsilon = \dg{\beta',-f'}^\ast \eta$.  Since $\dg{\beta',-f'}\dg{f \\\alpha}=0$, $\dg{\beta',-f'}$ factors through $(\beta,-f')$ by $(1)$ above. Say with $\dg{\beta',-f'}=s\dg{\beta,-f'}$ where $s: Y'\to Y'$. Therefore, $\varepsilon = \dg{\beta',-f'}^\ast \eta=\varepsilon = (s\dg{\beta,-f'})^\ast \eta=\dg{\beta,-f'}^\ast(s^\ast \eta)$. That is, $\varepsilon\in \operatorname{Im}(\dg{\beta,-f'}^\ast)$. Thus one obtains the exactness at $\mathbb E(Y \oplus X', -)$.
	\end{proof}

\vskip5pt
\begin{definition} \label{shm} \ A morphism $(\alpha, \beta, 1)$ of $\mathbb E$-triangles, as in \Cref{eq:hs-fg1}, is \emph{semi-homotopic},  if the left square in \Cref{eq:hs-fg1} is homotopic, i.e.,
$\con{X}{\dg{f\\ \alpha}}{Y \oplus X'}{\dg{\beta , -f'}}{Y'}\eta$ is an $\mathbb E$-triangle for some $\mathbb E$-extension $\eta\in \mathbb E(Y', X)$.
\vskip5pt Similarly, one has the concept of semi-homotopic for a morphism $(1,\beta,\gamma)$ of $\mathbb E$-triangles.
\end{definition}

There are strict inclusions:
\begin{equation*}    \resizebox{1\textwidth}{!}{
  $\quad \{\text{morphism of }\mathbb E\text{-triangles}\} \supsetneqq \{\text{semi-homotopic morphism}\} \supsetneqq \{\text{homotopic morphism}\}.$
  }
\end{equation*}
By \Cref{prop:mor-of-confl}$(1)$, a morphism $(\alpha, \beta, 1)$ of $\mathbb E$-triangles is homotopic if and only if  $\con X{\dg{f \\\alpha}}{Y\oplus X'}{\dg{\beta,-f'}}{Y'}{(g')^\ast\delta}$ is an $\mathbb E$-triangle. The strictness of the inclusions are illustrated by \Cref{ex:contra-les-fails} and \Cref{ex:ho-square-not-ho-mor} in Subsection 4.2.

\vskip5pt

\begin{proposition}\label{prop:semi-homotopic-6} \ Let $(\mathcal{C}, \mathbb E, \mathfrak s)$ be an extriangulated category.
\vskip5pt
$(1)$ \ Suppose that $(\alpha,\beta,1)$ is a semi-homotopic morphism of $\mathbb E$-triangles$:$
	\begin{equation*}
\xymatrix@R=0.4cm{
  X\ar[r]^{f}\ar[d]_-{\alpha} & Y\ar[r]^{g}\ar[d]^-{\beta} & Z\ar@{-->}[r]^{\delta}\ar@{=}[d] & {} \\
  X'\ar[r]^{{f'}} & Y'\ar[r]^{{g'}} & Z\ar@{-->}[r]^{{f_\ast \delta}} & .{}
}
\end{equation*}
    There is a $6$-term exact sequence of covariant functors
    \begin{equation*}    \resizebox{1\textwidth}{!}{
      $\mathcal C(-,X)\xrightarrow{\mathcal C(-,\dg{f \\\alpha})}\mathcal C(-,Y\oplus X')\xrightarrow{\mathcal C(-,\dg{\beta,-f'})}\mathcal C(-,Y')\xrightarrow{((g')^\ast\delta)_\sharp}\mathbb E(-,X)\xrightarrow{\dg{f \\\alpha}_\ast}\mathbb E(-,Y\oplus X')\xrightarrow{\dg{\beta,-f'}_\ast}\mathbb E(-,Y').$
      }
\end{equation*}

\vskip5pt

$(1)'$ \  Suppose that $(1,\beta,\gamma)$ is a semi-homotopic morphism of $\mathbb E$-triangles$:$
	\begin{equation*}
\xymatrix@R=0.4cm{
  X\ar[r]^{f}\ar@{=}[d] & Y\ar[r]^{g}\ar[d]_-{\beta} & Z\ar@{-->}[r]^{{\gamma^\ast \delta'}}\ar[d]^-{\gamma} & {} \\
  X\ar[r]^{{f'}} & Y'\ar[r]^{{g'}} & Z'\ar@{-->}[r]^{{\delta'}} & .{}
}
\end{equation*}
	There is a $6$-term exact sequence of contravariant functors
    \begin{equation*}    \resizebox{1\textwidth}{!}{
      $\mathcal C(Z',-)\xrightarrow{\mathcal C(\dg{\gamma,g'},-)}\mathcal C(Z\oplus Y',-)\xrightarrow{\mathcal C(\dg{-g \\\beta},-)}\mathcal C(Y,-)\xrightarrow{(f_\ast\delta')_\sharp}\mathbb E(Z',-)\xrightarrow{\dg{\gamma,g'}^\ast}\mathbb E(Z\oplus Y',-)\xrightarrow{\dg{-g \\\beta}^\ast}\mathbb E(Y,-).$
      }
      \end{equation*}
\end{proposition}

\begin{proof}
  One only justifies $(1)$. By assumption there is an $\mathbb E$-triangle $\con X{\dg{f\\\alpha}}{Y\oplus X'}{\dg{\beta,-f'}}{Y'}{\eta}$ for some $\eta \in \mathbb E(Y', X)$. The exactness at $\mathcal{C}(-, Y \oplus X')$ and $\mathbb E(-, Y \oplus X')$ are justified by applying \Cref{lem:long-ext-seq} to this $\mathbb E$-triangle.

  By \Cref{lem:hs-f?1}$(1)$ there exists $\beta' : Y \to Y'$ such that $\con X{\dg{f\\\alpha}}{Y\oplus X'}{\dg{\beta',-f'}}{Y'}{(g')^\ast\delta}$ is an $\mathbb E$-triangle. The exactness at $\mathbb E(-, X)$ is justified by applying \Cref{lem:long-ext-seq} to this $\mathbb E$-triangle.

  Similarly, by \Cref{lem:hs-f?1}$(2)$ there exists $\alpha' : X \to X'$ such that $\con X{\dg{f\\\alpha'}}{Y\oplus X'}{\dg{\beta,-f'}}{Y'}{(g')^\ast\delta}$ is an $\mathbb E$-triangle. The exactness at $\mathcal{C}(-, Y')$ is justified by applying \Cref{lem:long-ext-seq} to this $\mathbb E$-triangle.
\end{proof}

\vskip5pt
\begin{remark} \  Suppose $(\alpha,\beta,1)$ is a semi-homotopic morphism of $\mathbb E$-triangles. By \Cref{prop:mor-of-confl} there exists an $\mathbb E$-triangle  $\con X{\dg{f\\\alpha}}{Y\oplus X'}{\dg{\beta,-f'}}{Y'}{\eta}$ for some $\eta \in \mathbb E(Y', X)$. By \Cref{lem:long-ext-seq} one has the $6$-term exact sequences$:$
    \begin{equation*}    \resizebox{1\textwidth}{!}{
      $\mathcal C(Y',-)\xrightarrow{\mathcal C(\dg{\beta,-f'},-)}\mathcal C(Y\oplus X',-)\xrightarrow{\mathcal C(\dg{f \\\alpha},-)}\mathcal C(X,-)\xrightarrow{\eta^\sharp}\mathbb E(Y',-)\xrightarrow{\dg{\beta,-f'}^\ast}\mathbb E(Y \oplus X',-)\xrightarrow{\dg{f \\\alpha}^\ast}\mathbb E(X,-)$
      }
      \end{equation*}
      and
    \begin{equation*}    \resizebox{1\textwidth}{!}{
      $\mathcal C(-,X)\xrightarrow{\mathcal C(-,\dg{f \\\alpha})}\mathcal C(-,Y\oplus X')\xrightarrow{\mathcal C(-,\dg{\beta,-f'})}\mathcal C(-,Y')\xrightarrow{\eta_\sharp}\mathbb E(-,X)\xrightarrow{\dg{f \\\alpha}_\ast}\mathbb E(-,Y\oplus X')\xrightarrow{\dg{\beta,-f'}_\ast}\mathbb E(-,Y')$
      }
\end{equation*}
  which are different from \Cref{prop:quasi-long-ext}$(1)$ and \Cref{prop:semi-homotopic-6}$(1)$ respectively.
\end{remark}
\vskip5pt

Under some good conditions, $(\alpha, \beta, 1)$ and $(1, \beta, \gamma)$ are always semi-homotopic. The following conditions are given by A. Canonaco and M. K\"unzer \cite[Lemma 2.1]{CK11}.

\vskip5pt

\begin{condition}\label{con:c} \ Let $R$ be a unital ring. Consider the following conditions {\rm (C1)} and  {\rm (C2)}.

\vskip5pt

{\rm (C1)} \ For any $r \in R$, there exists $a\in R$ such that $1+r+ar^2$ is a unit in $R$.

\vskip5pt

{\rm (C2)} \ For any $r \in R$, there exists $b \in R$ such that $1+r+ r^2 b$ is a unit in $R$.

\end{condition}

\vskip5pt

For instance, a finite-dimensional algebra over a field satisfies both ({C1}) and ({C2}).
The following result generalizes \cite[Proposition 3.1]{CK11}, which concerns triangulated categories.

\vskip5pt

\begin{proposition}\label{prop:C1-ho} \ Let $(\mathcal{C}, \mathbb E, \mathfrak s)$ be an extriangulated category.
\vskip5pt
$(1)$ \ Suppose that $(\alpha,\beta,1_Z)$ is a morphism of $\mathbb E$-triangles$:$
	\begin{equation*}
\xymatrix@R=0.4cm{
  X\ar[r]^{f}\ar[d]_-{\alpha} & Y\ar[r]^{g}\ar[d]^-{\beta} & Z\ar@{-->}[r]^{\delta}\ar@{=}[d] & {} \\
  X'\ar[r]^{{f'}} & Y'\ar[r]^{{g'}} & Z\ar@{-->}[r]^{{f_\ast \delta}} & {}
}
\end{equation*}
    If $\mathrm{End}_\mathcal C (Y')$ satisfies {\rm Condition(C1)}, then $(\alpha, \beta, 1)$ is semi-homotopic. Moreover, there is an automorphism $\theta \in \mathrm{Aut}_\mathcal C(Y')$ such that
      \begin{equation}\label{eq:modified}
        \con X{\dg{f\\\alpha}}{Y\oplus X'}{\dg{\beta,-f'}}{Y'}{\theta^\ast(g')^\ast\delta}
      \end{equation}
      is an $\mathbb E$-triangle.

\vskip5pt

$(1)'$ \  Suppose that $(1_X,\beta,\gamma)$ is a morphism of $\mathbb E$-triangles$:$
	\begin{equation*}
\xymatrix@R=0.4cm{
  X\ar[r]^{f}\ar@{=}[d] & Y\ar[r]^{g}\ar[d]_-{\beta} & Z\ar@{-->}[r]^{{\gamma^\ast \delta'}}\ar[d]^-{\gamma} & {} \\
  X\ar[r]^{{f'}} & Y'\ar[r]^{{g'}} & Z'\ar@{-->}[r]^{{\delta'}} & {}
}
\end{equation*}
	If $\mathrm{End}_\mathcal C (Y)$ satisfies {\rm Condition(C2)}, then $(1_X,\beta,\gamma)$ is semi-homotopic. Moreover, there is an automorphism $\theta \in \mathrm{Aut}_\mathcal C(Y)$ such that
      \begin{equation}\label{eq:modified-2}
        \con Y{\dg{-g\\\beta}}{Z\oplus Y'}{\dg{\gamma,g'}}{Z'}{\theta_\ast f_\ast\delta'}
      \end{equation}
      is an $\mathbb E$-triangle.
\end{proposition}

    \begin{proof} \ One only justifies $(1)$. By \Cref{lem:hs-f?1}$(1)$ there exists $\beta' : Y \to Y'$ such that $\beta'f=f'\alpha$, $g'\beta=g$, and $\con X{\dg{f \\\alpha}}{Y\oplus X'}{\dg{\beta',-f'}}{Y'}{(g')^\ast\delta}$ is an $\mathbb E$-triangle. Note that $(\beta-\beta') f = 0$. By \Cref{lem:long-ext-seq} there exists $\varphi : Z \to Y'$ such that $\varphi g = \beta'-\beta$. Since $\mathrm{End}_\mathcal C (Y')$ satisfies Condition{\rm (C1)}, there exists some $a: Y'\to Y'$ such that $\theta:=1_Y + (\varphi g') + a(\varphi g')^2$ is an automorphism. A direct verification shows that
\begin{equation*}	\theta f'=(1_Y + (\varphi g') + a(\varphi g')^2) f' = f' + (\varphi + a \varphi g' \varphi)  (g' f') = f',
\end{equation*}
        and
\begin{equation*}	\theta \beta=(1_Y + (\varphi g') + a(\varphi g')^2) \beta = \beta + \varphi g + a \varphi g' \varphi g' \beta = \beta' + a \varphi g' (\beta' - \beta) = \beta'.
\end{equation*}
    Hence $\theta^{-1}\dg{\beta',-f'}=\dg{\beta,f'}$. By \Cref{isomorphisms} one applies an isomorphism $(1, 1,\theta)$ to the $\mathbb E$-triangle $\con X{\dg{f\\\alpha}}{Y\oplus X'}{\dg{\beta',-f'}}{Y'}{(g')^\ast\delta}$ and obtains the $\mathbb E$-triangle $\con X{\dg{f\\\alpha}}{Y\oplus X'}{\dg{\beta,-f'}}{Y'}{\theta^\ast(g')^\ast\delta}$.
	\end{proof}

\vskip5pt

Recall that a category $\mathcal{C}$ is $k$-linear, if for any objects $X$ and $Y$, $\mathcal C(X,Y)$ is a $k$-vector space, and the composition of morphisms is $k$-bilinear.
A $k$-linear category is $\mathrm{Hom}$-finite if any $\mathcal C(X,Y)$ is finite-dimensional.

\vskip5pt

\begin{theorem} \label{twolongexactseq} \ Let $(\mathcal{C}, \mathbb E, \mathfrak s)$ be an  extriangulated category.

\vskip5pt

$(1)$ \ Consider the sequence $\con X{\dg{f \\\alpha}}{Y\oplus X'}{\dg{\beta,-f'}}{Y'}{(g')^\ast\delta}$ associated to the morphism of $\mathbb E$-triangles$:$
\begin{equation*}
\xymatrix@R=0.4cm{
  X\ar[r]^{f}\ar[d]_-{\alpha} & Y\ar[r]^{g}\ar[d]^-{\beta} & Z\ar@{-->}[r]^{\delta}\ar@{=}[d] & {} \\
  X'\ar[r]^{{f'}} & Y'\ar[r]^{{g'}} & Z\ar@{-->}[r]^{{\alpha_\ast \delta}} &. {}
}
\end{equation*}
There is a $6$-term exact sequence of covariant functors
\begin{equation*}    \resizebox{1\textwidth}{!}{
      $\mathcal C(Y',-)\xrightarrow{\mathcal C(\dg{\beta,-f'},-)}\mathcal C(Y\oplus X',-)\xrightarrow{\mathcal C(\dg{f \\\alpha},-)}\mathcal C(X,-)\xrightarrow{((g')^\ast\delta)^\sharp}\mathbb E(Y',-)\xrightarrow{\dg{\beta,-f'}^\ast}\mathbb E(Y\oplus X',-)\xrightarrow{\dg{f \\\alpha}^\ast}\mathbb E(X,-).$
      }
\end{equation*}

\vskip5pt

Moreover, if $(\mathcal{C}, \mathbb E, \mathfrak s)$ is a $\mathrm{Hom}$-finite $k$-linear  extriangulated category, then there is also a $6$-term exact sequence of contravariant functors
 \begin{equation*}    \resizebox{1\textwidth}{!}{
      $\mathcal C(-,X)\xrightarrow{\mathcal C(-,\dg{f \\\alpha})}\mathcal C(-,Y\oplus X')\xrightarrow{\mathcal C(-,\dg{\beta,-f'})}\mathcal C(-,Y')\xrightarrow{((g')^\ast\delta)_\sharp}\mathbb E(-,X)\xrightarrow{\dg{f \\\alpha}_\ast}\mathbb E(-,Y\oplus X')\xrightarrow{\dg{\beta,-f'}_\ast}\mathbb E(-,Y').$
      }
\end{equation*}

\vskip5pt

$(1)'$  \ Consider the sequence $\con Y{\dg{-g \\\beta}}{Z\oplus Y'}{\dg{\gamma,g'}}{Z'}{f_\ast\varepsilon}$ associated to the morphism of $\mathbb E$-triangles$:$
    \begin{equation*}
\xymatrix@R=0.4cm{
  X\ar[r]^{f}\ar@{=}[d] & Y\ar[r]^{g}\ar[d]_-{\beta} & Z\ar@{-->}[r]^{{\gamma^*\varepsilon}}\ar[d]^-{\gamma} & {} \\  X\ar[r]^{f'} & Y'\ar[r]^{g'} & Z'\ar@{-->}[r]^{\varepsilon} &. {}
}
\end{equation*}
There is a $6$-term exact sequence of contravariant functors
\begin{equation*}    \resizebox{1\textwidth}{!}{
      $\mathcal C(-,Y)\xrightarrow{\mathcal C(-,\dg{-g \\\beta})}\mathcal C(-,Z\oplus Y')\xrightarrow{\mathcal C(-,\dg{\gamma,g'})}\mathcal C(-,Z')\xrightarrow{(f_\ast\varepsilon)_\sharp}\mathbb E(-,Y)\xrightarrow{\dg{-g \\\beta}_\ast}\mathbb E(-,Z\oplus Y')\xrightarrow{\dg{\gamma,g'}_\ast}\mathbb E(-,Z').$
      }
\end{equation*}

Moreover, if $(\mathcal{C}, \mathbb E, \mathfrak s)$ is a $\mathrm{Hom}$-finite $k$-linear extriangulated category, then there is also a $6$-term exact sequence of covariant functors

\begin{equation*}    \resizebox{1\textwidth}{!}{
      $\mathcal C(Z',-)\xrightarrow{\mathcal C(\dg{\gamma,g'},-)}\mathcal C(Z\oplus Y',-)\xrightarrow{\mathcal C(\dg{-g \\\beta},-)}\mathcal C(Y,-)\xrightarrow{(f_\ast\varepsilon)^\sharp}\mathbb E(Z',-)\xrightarrow{\dg{\gamma,g'}^\ast}\mathbb E(Z\oplus Y',-)\xrightarrow{\dg{-g \\\beta}^\ast}\mathbb E(Y,-).$
      }
\end{equation*}
\end{theorem}

\begin{proof} \  One only justifies $(1)$. The assertion $(1)'$ can be dually proved. The $6$-term exact sequence of covariant functors
follows from \Cref{prop:quasi-long-ext}$(1)$. If $\mathcal C$ is a $\mathrm{Hom}$-finite $k$-linear category,
then the endomorphism ring of any object is a finite-dimensional $k$-algebra, which satisfies {\rm Condition(C1)} and {\rm (C2)}, hence the $6$-term exact sequence of contravariant functors follows from \Cref{prop:semi-homotopic-6}$(1)$ and \Cref{prop:C1-ho}$(1)$.
\end{proof}

\subsection{Strict inclusions}\label{sec:strict-implication} \ We show the strict inclusions:
\begin{equation*}    \resizebox{1\textwidth}{!}{
  $\{\text{morphisms of }\mathbb E\text{-triangles}\} \supsetneqq \{\text{semi-homotopic morphisms}\} \supsetneqq \{\text{homotopic morphisms}\}.$
  }
\end{equation*}
\vskip5pt
\begin{example}\label{ex:contra-les-fails} \ Let $(\mathcal{C}, \mathbb E, \mathfrak s)$ be an extriangulated category.
   \vskip5pt
  $(1)$ \ Let $\con X{\dg{f \\\alpha}}{Y\oplus X'}{\dg{\beta,-f'}}{Y'}{(g')^\ast\delta}$ be a sequence associated to the commutative diagram \Cref{eq:hs-fg1}. Then the sequence of covariant functors
  \begin{equation}\label{eq:contra-les-fails}
    \mathcal C(-, X) \xrightarrow{\mathcal C(-,\dg{f \\ \alpha})}\mathcal C(-, Y \oplus X') \xrightarrow{\mathcal C(-,\dg{\beta, -f'})}\mathcal C(-, Y')
  \end{equation}
  is not always exact at $\mathcal C(-, Y \oplus X')$. If this is the case, then $(\alpha, \beta, 1)$ fails to be semi-homotopic.
  \vskip5pt
  $(1)'$ \ Let $\con Y{\dg{-g \\\beta}}{Z\oplus Y'}{\dg{\gamma,g'}}{Z'}{f_\ast\varepsilon}$ be a sequence associated to the commutative diagram \Cref{eq:con-mor-o}. Then the sequence of contravariant functors
\begin{equation*}
  \mathcal C(Z',-)\xrightarrow{\mathcal C(\dg{\gamma,g'},-)}\mathcal C(Z\oplus Y',-)\xrightarrow{\mathcal C(\dg{-g \\\beta},-)}\mathcal C(Y,-)
\end{equation*}
is not always exact at $\mathcal C(Z\oplus Y',-)$. If this is the case, then $(1, \beta, \gamma)$ fails to be semi-homotopic.
  \vskip5pt

   Here is a counterexample of $(1)$ inspired by diagram$(\ast)$ in \cite[Lemma 4.1]{CK11}. Let $\mathcal{C} = K^b(\mathrm{proj}(\mathbb{Z}))$. Consider the following morphisms of distinguished triangles$:$
\begin{equation*}\label{eq:big-diagram}
  \resizebox{1\textwidth}{!}{
      $
  \xymatrix@R=0.4cm{
  {} & {} & \mathbb Z^2\ar[rrrr]^{ \dg{9&0\\ 0&1}}\ar@{=}[ddd] & {} & {} & {} & \mathbb Z^2\ar[rrrr]^{{\dg{-1,0}}}\ar[ddd]^{{\dg{0,1}}} & {} & {} & {} & \mathbb Z\ar@{.>}[rrrr]\ar@{.>}[ddd] & {} & {} & {} & 0\ar@{=}[ddd] \\
  {} & \mathbb Z\ar[ur]^(.6){ \dg{-3 \\ 9}}\ar[rrrr]^{1}\ar@{=}[ddd] & {} & {} & {} & \mathbb Z\ar[ur]^(0.3){ \dg{-27 \\ 9}}\ar[rrrr]^{3}\ar[ddd]^{1} & {} & {} & {} & \mathbb Z\ar[ur]^{{9}}\ar[rrrr]^{ \dg{-1\\0}}\ar[ddd]^{{4}} & {} & {} & {} & \mathbb Z^2\ar[ur]\ar@{=}[ddd] & {}  \\
  0\ar[ur]\ar@{.>}[rrrr]\ar@{=}[ddd] & {} & {} & {} & 0\ar[ur]\ar@{.>}[rrrr]\ar@{.>}[ddd] & {} & {} & {} & 0\ar[ur]\ar@{.>}[rrrr]\ar@{.>}[ddd] & {} & {} & {} & \mathbb Z\ar[ur]_(.3){ \dg{3 \\ -9}}\ar@{=}[ddd] & {} & {}\\
  {} & {} & \mathbb Z^2\ar[rrrr]^(.35){{\dg{0,1}}} & {} & {} & {} & \mathbb Z\ar@{.>}[rrrr] & {} & {} & {} & 0\ar@{.>}[rrrr] & {} & {} & {} &  \\
  {} & \mathbb Z\ar[ur]^(.6){ \dg{-3 \\ 9}}\ar[rrrr]^{1} & {} & {} & {} & \mathbb Z\ar[ur]^{{9}}\ar[rrrr]^{3} & {} & {} & {} & \mathbb Z\ar[ur]\ar[rrrr]^{ \dg{-1\\0}} & {} & {} & {} & \mathbb Z^2\ar[ur] & {}  \\
  0\ar[ur]\ar@{.>}[rrrr] & {} & {} & {} & 0\ar[ur]\ar@{.>}[rrrr] & {} & {} & {} & 0\ar[ur]\ar@{.>}[rrrr] & {} & {} & {} & \mathbb Z\ar[ur]_(.3){ \dg{3 \\ -9}} & {} & {}
}
$}
\end{equation*}
\vskip5pt

For simplicity, we denote the above diagram as follows:
\begin{equation*}
  \xymatrix@R=0.4cm{
  \Sigma^{-1} Z\ar[r]^-{h}\ar@{=}[d] & X\ar[r]^{f}\ar[d]_-{\alpha} & Y\ar[r]^{g}\ar[d]^-{\beta} & Z\ar@{=}[d] \\
  \Sigma^{-1} Z\ar[r]^-{{h'}} & X'\ar[r]^{{f'}} & Y'\ar[r]^{{g'}} & Z
}
\end{equation*}
One fixes the distinguished triangles $\tri LiX\alpha{X'}p$ and $\tri KjY\beta{Y'}q$ induced by $\alpha$ and $\beta$, respectively. Suppose that the sequence $\mathcal C(-, X) \xrightarrow{\mathcal C(-,\dg{f \\ \alpha})}\mathcal C(-, Y \oplus X') \xrightarrow{\mathcal C(-,\dg{\beta, -f'})}\mathcal C(-, Y')$ is exact. Then there must exist a morphism $s: K \to X$ such that $fs = j$ and $\alpha s = 0$. By \Cref{lem:long-ext-seq} there exists a morphism $\varphi: K \to L$ such that the following diagram commutes:
\begin{equation*}
  \xymatrix@R=0.5cm{
  {} & L\ar[d]_{i} & K\ar@{..>}[l]_{{\varphi }}\ar@{..>}[dl]_{s}\ar[d]^{j} & {} \\
  \Sigma^{-1}Z\ar[r]^{h}\ar@{=}[d] & X\ar[r]^{f}\ar[d]_-{\alpha} & Y\ar[r]^{g}\ar[d]^-{\beta} & Z\ar@{=}[d] \\
  \Sigma^{-1}Z\ar[r]^{{h'}} & X'\ar[r]^{{f'}} & Y'\ar[r]^{{g'}} & Z
}
\end{equation*}
\vskip5pt
The mapping cone of $fi$ is constructed as follows:
\begin{equation*}
\xymatrix@R=0.5cm{
  L\ar[d]_{i} & 0\ar[r]\ar[d] & \mathbb Z\ar[r]^-{\scalebox{0.8}{$\dg{-27\\ 9\\ -1}$}}\ar[d]_{1}  & \mathbb Z\oplus \mathbb Z\oplus \mathbb Z\ar[r]^-{\scalebox{0.8}{$-\dg{0,1,9}$}}\ar[d]^{\scalebox{0.8}{$\dg{1&0&0\\0&1&0}$}} & \mathbb Z\ar[d] \\
  X\ar[d]_{f} & 0\ar[r]\ar[d] & \mathbb Z\ar[r]^-{ \dg{-27 \\ 9}}\ar[d]_{3}  & \mathbb Z \oplus \mathbb Z\ar[r]\ar[d]^{{\dg{-1,0}}} & 0\ar[d] \\
  Y\ar@{>->}[d] & 0\ar[r]\ar[d] & \mathbb Z\ar[r]^-{{9}}\ar[d]_-{\scalebox{0.6}{$\dg{0\\0\\0\\1}$}} & \mathbb Z\ar[r]\ar[d]^-{ \dg{0\\1}} & 0\ar[d] \\
  \mathrm{Cone}(fi) & \mathbb Z\ar[r]^-{ \scalebox{0.6}{$\dg{27\\ -9\\ 1\\ 3}$}} & \mathbb Z\oplus \mathbb Z\oplus \mathbb Z\oplus \mathbb Z\ar[r]^-{ \scalebox{0.8}{$\dg{0&1&9&0\\ -1&0&0&9}$}} & \mathbb Z \oplus \mathbb Z\ar[r] & 0
}
\end{equation*}

Since $j$ factors through $fi$, the composition $K \xrightarrow {j} Y \rightarrowtail \mathrm{Cone}(fi)$ is zero in the homotopy category. Consequently, there are $s : \mathbb Z \to \mathbb Z$ and $t = \dg{x'&y'\\z'&w'\\x&y\\z&w}: \mathbb Z\oplus \mathbb Z\to \mathbb Z\oplus \mathbb Z\oplus \mathbb Z\oplus \mathbb Z$ such that the following is a null-homotopic morphism of complexes:

\begin{equation*}
  \xymatrix@R=1cm@C=1.5cm{
  0\ar[r]\ar[d] & \mathbb Z\ar[r]^-{\dg{9 \\ 4}}\ar@/_.5pc/@{..>}[dl]_(.6){s}\ar[d]_-{\scalebox{0.6}{$\dg{0\\0\\0\\1}$}} & \mathbb Z\oplus \mathbb Z\ar[r]\ar@{..>}[dl]_-{t}\ar[d]_-{\scalebox{0.8}{$\dg{0&0\\1&0}$}} & 0\ar[dl]\ar[d] \\
  \mathbb Z\ar[r]^(.4){\scalebox{0.6}{$\dg{27\\-9\\1\\3}$}} &  \mathbb Z\oplus \mathbb Z\oplus \mathbb Z\oplus \mathbb Z\ar[r]^-{\scalebox{0.8}{$\dg{0&1&9&0\\-1&0&0&3}$}} & \mathbb Z\oplus \mathbb Z\ar[r] & 0
}
\end{equation*}

Since $\dg{0&1&9&0\\-1&0&0&3}\dg{x'&y'\\z'&w'\\x&y\\z&w} = \dg{0&0\\1&0}$, $\dg{x'&y'\\z'&w'\\x&y\\z&w}$ takes the form $\dg{-1+3z&3w\\-9x&-9y\\x&y\\z&w}$. Then
\vskip5pt
\begin{equation*}
  \dg{-1+3z&3w\\-9x&-9y\\x&y\\z&w} \dg{9 \\ 4} + \dg{27 \\ -9 \\ 1 \\ 3} s = \dg{-9+27z+12w+27s\\-81x-36y-9s\\ 9x+4y+s\\ 9z + 4w + 3s } = \dg{0\\0\\0\\1}.
\end{equation*}
From rows $1$ and $4$, one has $3s = 1$. This yields a contradiction. \hfill $\square$

\end{example}

\vskip5pt

\begin{example}\label{ex:ho-square-not-ho-mor} \ Let $(\mathcal{C}, \mathbb E, \mathfrak s)$ be an extriangulated category.
   \vskip5pt
  $(1)$ \ Let $\con X{\dg{f \\\alpha}}{Y\oplus X'}{\dg{\beta,-f'}}{Y'}{(g')^\ast\delta}$ be a sequence associated to the commutative diagram \Cref{eq:hs-fg1}. Even though $(\alpha, \beta, 1)$ is semi-homotopic, it may fail to be homotopic.
  \vskip5pt
  $(1)'$ \ Let $\con Y{\dg{-g \\\beta}}{Z\oplus Y'}{\dg{\gamma,g'}}{Z'}{f_\ast\varepsilon}$ be a sequence associated to the commutative diagram \Cref{eq:con-mor-o}. Even though $(1, \beta, \gamma)$ is semi-homotopic, it may fail to be homotopic.
  \vskip5pt

  Here is a counterexample of $(1)$ based on \cite[Example 6.19]{CF22}. Let $k := \mathbb F_2$ be a field, $R := k[\mathbb Z / 4 \mathbb Z]$ be a group algebra. For $r \in R$, let $\mu_r : R/x^i \to R/x^j$ denote the morphism $[1] \mapsto [r]$ when it is well-defined. The category of finite-dimensional $R$-modules $\mathrm{mod}_R$ is a Frobenius exact category, and its stable category is a $\mathrm{Hom}$-finite $k$-linear triangulated category. By \cite[Example 6.19]{CF22} there exists a morphism between distinguished triangles
  \begin{equation*}
    \xymatrix@R=0.7cm{
R/x^3 \oplus R/x^2 \ar@{->}[r]^-{\scalebox{0.8}{$\dg{1&0\\0 & \mu_x}$}} \ar@{=}[d] & R/x^3 \oplus R/x^3 \ar@{->}[r]^-{\scalebox{0.8}{$\dg{0, \mu_1}$}} \ar@{->}[d]_-{\dg{\mu _1, 0}} & R/x \ar@{->}[r]^-{\dg{0 \\ \mu_x}} \ar@{->}[d]^{\dg{0 \\ \mu_x}} & R/x \oplus R/x^2 \ar@{=}[d] \\
R/x^3 \oplus R/x^2 \ar@{->}[r]^-{\dg{\mu_1, 0}} & R/x \ar@{->}[r]^-{\dg{\mu _x \\ 0}} & R/x^2 \oplus R/x^2  \ar@{->}[r]^-{\scalebox{0.8}{$\dg{\mu _1 & \mu _1 \\ 0 & 1}$}} & R/x \oplus R/x^2
}
  \end{equation*}
  By \Cref{twolongexactseq}, the middle square is homotopic. We claim that $(\dg{\mu_1,0}, \dg{0 \\ \mu_x}, 1)$ is not a homotopic morphism. For sake of contradiction, we assume ${R/x^3\oplus R/x^3 }\xrightarrow {\scalebox{0.8}{$\dg{\mu_1&0\\0&\mu_1}$}} {R/x\oplus R/x} \xrightarrow{\scalebox{0.8}{$\dg{\mu_x&0\\0&-\mu_x}$}}{R/x \oplus R/x}\xrightarrow {\scalebox{0.8}{$\dg{\mu_1&\mu_1\\0&\mu_1}$}} {R/x \oplus R/x^2}$ is a distinguished triangle.
  Since $R/x^3 \xrightarrow {\mu_1} R/x \xrightarrow {\mu_x} R/x^2 \xrightarrow {\mu_1} R/x$ is a distinguished triangle, there exists $\varphi \in \mathrm{End} (R/x^3 \oplus R/x^3)$ such that the following diagram commutes:
  \begin{equation*}
    \xymatrix@R=0.5cm@C=1.3cm{
R/x^3\oplus R/x^3  \ar@{..>}[d]^-{\varphi} \ar@{->}[r]^-{\scalebox{0.8}{$\dg{\mu_1&0\\0&\mu_1}$}} & R/x\oplus R/x \ar@{=}[d] \ar@{->}[r]^-{\scalebox{0.8}{$\dg{\mu_x&0\\0&-\mu_x}$}} & R/x^2\oplus R/x^2 \ar@{=}[d]  \ar@{->}[r]^-{\scalebox{0.8}{$\dg{\mu_1&\mu_1\\0&\mu_1}$}} & R/x\oplus R/x \ar@{..>}[d]^-{\Sigma \varphi} \\
R/x^3\oplus R/x^3 \ar@{->}[r]^-{\scalebox{0.8}{$\dg{\mu_1&0\\0&\mu_1}$}} & R/x\oplus R/x \ar@{->}[r]^-{\scalebox{0.8}{$\dg{\mu_x&0\\0&-\mu_x}$}} & R/x^2\oplus R/x^2 \ar@{->}[r]^-{\scalebox{0.8}{$\dg{\mu_1&0\\0&-\mu_1}$}} & R/x\oplus R/x
} \end{equation*}
  Since $R/x \cong \mathbb F_2$, one has $\Sigma \varphi = \dg{1 & 1 \\ 0 & 1}$. Since $\Sigma$ is an auto-equivalence, $\varphi = \dg{1 & 1 \\ 0 & 1}$. The left square fails to commute, which yields a contradiction. \hfill $\square$
\end{example}

\vskip5pt

\begin{remark}
  In this case, the triangle associated to a semi-homotopic morphism takes the form $\tri XuYvZ{\alpha w}$, where $\tri XuYvZw$ is a distinguished triangle and $\alpha \in \mathrm{Aut}(\Sigma X)$ (\Cref{prop:C1-ho}$(2)$).
\end{remark}

\section{Homotopic morphisms and $4 \times 4$ Lemma}

This section aims to generalize $4 \times 4$ Lemma in triangulated categories to extriangulated categories.
Since the class of $\mathbb E$-triangles is not closed under rotations, $4 \times 4$ Lemma in extriangulated categories has several variants.

\subsection{Morphisms which are $\mathbb E$-inflations and $\mathbb E$-deflations}

To study {\rm $4 \times 4$ Lemma}, one inevitably needs to rotate some $\mathbb E$-triangles. Thus one should consider the class of morphisms which are $\mathbb E$-inflations and $\mathbb E$-deflations.

\vskip5pt

\begin{notation}\label{def:slr} \ Let $(\mathcal{C}, \mathbb{E}, \mathfrak{s})$ be an extriangulated category.
Denote by $S$ the class of morphisms which are $\mathbb E$-inflations and $\mathbb E$-deflations,  by $\mathcal{L}$ the class of objects $L$ such that
$L \to 0$ is an $\mathbb E$-inflation, and by $\mathcal{R}$ the class of objects $R$ such that $0 \to R$ is an $\mathbb E$-deflation.
\end{notation}

\vskip5pt

In an exact category, $S$ consists of all the isomorphisms, $\mathcal{L} = \{0\} = \mathcal{R}.$
In a triangulated category, $S$ consists of all the morphisms, $\mathcal{L}  = \mathcal{R}$ consists of all the objects.

\vskip5pt

\begin{lemma}\label{prop:slr}
	Let $(\mathcal{C}, \mathbb E, \mathfrak s)$ be an extriangulated category. Then
\vskip5pt
$(1)$ \ $S$ contains all isomorphisms, and $S$ is closed under finite coproducts and compositions.
\vskip5pt
$(2)$ \ $\mathcal L$ is closed under isomorphisms and finite coproducts. Moreover, let $\con XfYgZ\delta$ be an $\mathbb E$-triangle. If $Y \in \mathcal{L}$, then $X \in \mathcal{L};$ if $X,Z\in\mathcal{L}$, then $Y \in \mathcal{L}$. In particular, $\mathcal L$ is closed under direct summands.
\vskip5pt
$(2)'$ \  $\mathcal R$ is closed under isomorphisms and finite coproducts. Moreover, let $\con XfYgZ\delta$ be an $\mathbb E$-triangle in $\mathcal{C}$. If $Y \in \mathcal{R}$, then $Z \in \mathcal{R};$ if $X,Z\in\mathcal{R}$, then $Y \in \mathcal{R}$. In particular, $\mathcal L$ is closed under direct summands.
\end{lemma}

\begin{proof} \ $(1)$ \ This is clear by ET2, ET4, and ET4$^{\mathrm{op}}$.

\vskip5pt

$(2)$ \  We only justify the last assertion. Let $\con XfYgZ\delta$ be an $\mathbb E$-triangle
with $X, Z\in \mathcal L$. Then there is a homotopic square
    \begin{equation*}
\xymatrix@R=0.4cm{
  X\ar[r]\ar[d]_{f} & 0\ar[d] \\
  Y\ar[r]^{g} & Z
}
\end{equation*}
    Since $X\in \mathcal L$ and $X\to 0$ is an $\mathbb E$-inflation, by \Cref{prop:hs-inflation}$(1)$, $g$ is also an $\mathbb E$-inflation. Note that $Y \to 0$ is the composition of $g$ and $Z\to 0$, which is also an $\mathbb E$-inflation, i.e.,  $Y\in \mathcal L$.
\end{proof}

\vskip5pt

\begin{lemma}\label{prop:slr-2} \ Let $(\mathcal{C}, \mathbb E, \mathfrak s)$ be an extriangulated category. Then

\vskip5pt
$(1)$ \ $X\in \mathcal L$ if and only if there is an $\mathbb E$-triangle $\con{X}{}{A}{f}{B}{}$ with $f\in S$.
\vskip5pt
$(1)'$ \  $Y\in \mathcal R$ if and only if there is an $\mathbb E$-triangle $\con{A}{f}{B}{}{Y}{}$ with $f\in S$.
\vskip5pt
$(2)$ \  $f \in S$ if and only if there is an $\mathbb E$-triangle $\con{X}{}{A}{f}{B}{}$ with $X\in \mathcal L$.
\vskip5pt
$(2)'$ \  $f \in S$ if and only if there is an $\mathbb E$-triangle $\con{A}{f}{B}{}{Y}{}$ with $Y\in \mathcal R$.
\end{lemma}
\begin{proof} \ We only justify $(1)$ and $(2)$.
\vskip5pt
$(1)$ \ Suppose that there is an $\mathbb E$-triangle $\con {X}kAfB{}$ with $f \in S$. Since $f$ is an $\mathbb E$-inflation, there is an $\mathbb E$-triangle $\con A f B c{Y}{}$.
One writes those $\mathbb E$-triangles as homotopic squares:
		\begin{equation*}
\xymatrix@R=0.4cm{
  X\ar[r]^{k}\ar[d] & A\ar[r]\ar[d]^-{f} & 0\ar[d] \\
  0\ar[r] & B\ar[r]^{c} & Y
}
\end{equation*}
		By \Cref{thm:hs-composition} the outer rectangle is also a homotopic square, which follows that $\con {X}{}0{}{Y}{}$ is an $\mathbb E$-triangle. Hence $X \in \mathcal{L}$.
Conversely, if $X\in \mathcal L$, by definition there is an $\mathbb E$-triangle $\con X{}0fR{}$ with $f\in S$.

\vskip5pt
$(2)$ \ Let $f: A\to B$ be a morphism in $S$. Since $f$ is an $\mathbb E$-deflation, there is $\mathbb E$-triangle $\con {X}kAfB{}$. By $(1)$, $X\in \mathcal L$.
Conversely, let $\con X{}AfB{}$ be an $\mathbb E$-triangle with $X \in \mathcal{L}$. By definition
        \begin{equation*}
\xymatrix@R=0.4cm{
  X\ar[r]\ar[d] & 0\ar[d] \\
  A\ar[r]^{f} & B
}
\end{equation*}
is a homotopic square.  Since $X\in \mathcal L$,  $X\to 0$ is an $\mathbb E$-inflation. By \Cref{prop:hs-inflation}$(1)$, $f$ is also an $\mathbb E$-inflation. Thus $f\in S$.
	\end{proof}

\vskip5pt

\begin{lemma}\label{Ff}
    Let $(\mathcal{C}, \mathbb E, \mathfrak s)$ be an extriangulated category, $\con X{}0{}{X'}{\delta_X}$ and $\con Y{}0{}{Y'}{\delta_Y}$ be $\mathbb E$-triangles.
\vskip5pt
$(1)$ \  For any $f: X\to Y$, there is a unique $g:X'\to Y'$ such that $f_\ast \delta_X=g^\ast \delta_Y$.

\vskip5pt
$(1)'$ \ For any $g:X'\to Y'$, there is a unique $f: X\to Y$ such that $f_\ast \delta_X=g^\ast \delta_Y$.
    \begin{equation*}
\xymatrix@R=0.4cm{
  X\ar[r]\ar[d]_{f} & 0\ar[r]\ar[d] & X'\ar@{-->}[r]^{{\delta_X}}\ar[d]^{g} & {} \\
  Y\ar[r] & 0\ar[r] & Y'\ar@{-->}[r]^{{\delta_Y}} & {}
}
\end{equation*}
\end{lemma}

        \begin{proof} \ One only shows $(1)$. By ET3 there is $g:X'\to Y'$ such that $f_\ast \delta_X=g^\ast \delta_Y$. Suppose that there is also $g':X'\to Y'$ such that $f_\ast \delta_X=(g')^\ast \delta_Y$. Then $(g-g')^\ast \delta_Y=0$. By \Cref{lem:long-ext-seq}, $(g-g')$ factors through $0\to Y'$. Hence $g=g'$.
\end{proof}

\vskip5pt

A connection between $\mathcal{L}$ and $\mathcal{R}$ is given as follows. For each $X \in \mathcal{L}$, one fixes an $\mathbb E$-triangle $\con X{}0{}{FX}{\delta_X}$. By definition $FX\in \mathcal R$. For each morphism $f:X\to Y$ with $X, Y\in \mathcal L$, one defines $Ff: FX\to FY$ to be the unique morphism such that $f_\ast \delta_X={(Ff)}^\ast \delta_Y$. Thus $F : \mathcal L \to \mathcal R$ is a functor. For another morphism $h: X\to Y$, one has $$(F(f+h))^\ast\delta_Y=(f+h)_\ast\delta_X=f_\ast\delta_X+h_\ast\delta_X=(Ff+Fh)^*\delta_Y.$$ By the uniqueness, one has $F(f+h) = Ff + Fh$, i.e., $F$ is an additive functor.

\vskip5pt

\begin{lemma}\label{F}
    Let $(\mathcal{C}, \mathbb E, \mathfrak s)$ be an extriangulated category. The additive functor $F: \mathcal L\to \mathcal R$ is an equivalence.
	\end{lemma}
	\begin{proof} \ $F$ is fully faithful by \Cref{Ff}. Let $Y\in \mathcal R$. By definition there exists an $\mathbb E$-triangle $\con X{}0{}Y{}$ where $X\in \mathcal L$. It follows from ET3 and \Cref{cone} that $Y\cong FX$. Hence, $F$ is dense.
	\end{proof}

\vskip5pt

\begin{lemma}\label{thm:lr-equivalence}\label{thm:rl-equivalence}
	Let $(\mathcal{C}, \mathbb E, \mathfrak s)$ be an extriangulated category, $X \in \mathcal{L}$ and $M \in \mathcal{C}$.
\vskip5pt
$(1)$ \  There is an isomorphism $\ell_{M,X} : \mathbb E(M,X) \cong \mathcal{C}(M, FX)$, which is functorial in both arguments.

\vskip5pt

$(2)$ \  There is an isomorphism $\rho _{X,M} : \mathbb E(FX,M) \cong \mathcal{C}(X,M)$, which is functorial in both arguments.
\end{lemma}

\begin{proof} \ $(1)$ \  Since $X\in \mathcal L$, there is an $\mathbb E$-triangle $X\to 0\to FX\overset{\delta_X}\dashrightarrow$.
By \Cref{lem:long-ext-seq}  there is an isomorphism $(\delta_X)_\sharp(M): \mathcal{C}(M, FX)  \cong \mathbb E(M,X)$. Take $\ell_{M,X}=((\delta_X)_\sharp(M))^{-1}: \mathbb E(M,X) \cong \mathcal{C}(M, FX)$.
\vskip5pt
$(2)$ \ This can be similarly justified. \end{proof}

\begin{notation}\label{binaturality} The natural isomorphisms  $\ell$ and $\rho$ are in fact inspired by ``the connecting morphism'' by Nakaoka and Palu \cite[Definition 6.9]{NP19}.

\vskip5pt

    Let $(\mathcal{C}, \ \mathbb E, \ \mathfrak s)$ be an extriangulated category,  $X\in \mathcal L$ and $M\in \mathcal C$.
In the remaining part of this section, for $\delta \in \mathbb E(M,X)$ (respectively, $\varepsilon \in \mathbb E(FX,M)$) and $\varphi\in \mathcal C(M, FX)$ (respectively, $\psi \in \mathcal C(X, M)$), one writes $\ell(\delta):= \ell_{M, X}(\delta)$ (respectively, $\rho(\varepsilon):=\rho_{X, M}(\varepsilon)$) and $\ell^{-1}(\varphi):= (\ell_{M, X})^{-1}(\varphi)$ (respectively, $\rho^{-1}(\psi):=(\rho_{X, M})^{-1}(\psi)$) for simplicity. Then
\vskip5pt
$(1)$ \ For $f: X\to X'$ and $\delta\in \mathbb E(M, X)$, there holds $(Ff) \ell(\delta) =\ell(f_{\ast}\delta)$.

\vskip5pt
$(2)$ \ For $g: M\to M'$ and $\delta'\in \mathbb E(M', X)$, there holds $\ell(\delta') g=\ell(g^{\ast}\delta')$.

\vskip5pt
$(3)$ \ For $f: X\to X'$ and $\varepsilon'\in \mathbb E(FX', M)$, there holds $\rho(\varepsilon')f =\rho((Ff)^{\ast}\varepsilon')$.

\vskip5pt
$(4)$ \  For $g: M\to M'$ and $\varepsilon\in \mathbb E(FX, M)$, there holds $g(\rho(\varepsilon)) =\rho(g_{\ast}\varepsilon)$.

\end{notation}

\vskip5pt

\begin{lemma}\label{lem:rotation}
	Let $(\mathcal{C}, \ \mathbb E, \ \mathfrak s)$ be an extriangulated category, $F:\mathcal{L}\to \mathcal{R}$ the functor in \Cref{F}, and $f:A\to B$  a morphism in $S$.
\vskip5pt

 \vskip5pt
$(1)$ \  Let $\con KiAfB\eta$ be an $\mathbb E$-triangle. Then $\con AfB{\ell (\eta)}{FK}{-\rho^{-1}(i)}$ is an $\mathbb E$-triangle.

\vskip5pt
$(2)$ \ Let $\con AfBp{FK}\xi$ be an $\mathbb E$-triangle. Then $\con K{\rho(\xi)}AfB{-\ell ^{-1}(p)}$ is an $\mathbb E$-triangle.
\end{lemma}

	\begin{proof} \ $(1)$ \ By definition $(\ell(\eta))^\ast(\delta_K)=\eta$. By \Cref{lem:hs-1?h}$(1)'$ the following is a homotopic morphism of $\mathbb E$-triangles
			\begin{equation*}
\xymatrix@R=0.4cm{
  K\ar[r]^{i}\ar@{=}[d] & A\ar[r]^{f}\ar[d] & B\ar@{-->}[r]^{\eta}\ar[d]_-{{\ell(\eta)}} & {} \\
  K\ar[r] & 0\ar[r] & FK\ar@{-->}[r]^{{\delta_K }} &. {}
}
\end{equation*}
Thus  $A \xrightarrow{- f} B \xrightarrow{\ell(\eta)} FK \overset {i_\ast \delta _K}\dashrightarrow$ is an $\mathbb E$-triangle. By definition $i_\ast \delta_K = \rho ^{-1}(i)$.
Then by \Cref{isomorphisms}, $\con AfB{\ell (\eta)}{FK}{-\rho^{-1}(i)}$ is an $\mathbb E$-triangle.

\vskip5pt
$(2)$ \   By definition $(\rho(\xi))_\ast(\delta_K)=\xi$. By \Cref{lem:hs-f?1}(1) the following is a homotopic morphism of $\mathbb E$-triangles
\begin{equation*}
\xymatrix@R=0.4cm{
  K\ar[r]\ar[d]_-{{\rho(\xi )}} & 0\ar[r]\ar[d] & FK\ar@{-->}[r]^{{\delta_K}}\ar@{=}[d] & {} \\
  A\ar[r]^{f} & B\ar[r]^{p} & FK\ar@{-->}[r]^{\xi} &. {}
}
\end{equation*}
Thus $K \xrightarrow{\rho (\xi)} A \xrightarrow{- f} B \overset {p^\ast \delta _K}\dashrightarrow$ is an $\mathbb E$-triangle. By definition $p^\ast \delta_K = \ell^{-1} (p)$. Then by \Cref{isomorphisms}, $\con K{\rho(\xi)}AfB{-\ell ^{-1}(p)}$ is an $\mathbb E$-triangle.
	\end{proof}

	\begin{proposition}\label{thm:hs-square-s}
		Let $(\mathcal{C}, \mathbb E, \mathfrak s)$ be an extriangulated category. Suppose that there is a homotopic morphism of $\mathbb E$-triangles with $f,g\in S$
\begin{equation*}
\xymatrix@R=0.4cm{
  K\ar[r]^{i}\ar@{=}[d] & X\ar[r]^{f}\ar[d]_{u} & Y\ar@{-->}[r]^{{l_f}}\ar[d]^{v} & {} \\
  K\ar[r]^{j} & A\ar[r]^{g} & B\ar@{-->}[r]^{{l_g}} &. {}
}
\end{equation*}
Then there is a homotopic morphism of $\mathbb E$-triangles
\begin{equation*}
\xymatrix@R=0.4cm{
  X\ar[r]^{f}\ar[d]_{u} & Y\ar[r]^{p}\ar[d]^{v} & FK\ar@{-->}[r]^{{r_f}} & {} \\
  A\ar[r]^{g} & B\ar[r]^{q} & FK\ar@{=}[u]\ar@{-->}[r]^{{r_g}} & {}
}
\end{equation*}
such that $$u_\ast r_f=r_g, \ \ i_\ast l_g = - q^\ast r_f, \ \ \ell (l_f) = p, \ \ \ell (l_g) = q, \ \ \rho (r_f) = - i, \ \ \rho (r_g) = - j.$$
\end{proposition}

		\begin{proof} \ By \Cref{lem:rotation} there are $\mathbb E$-triangles $X \xrightarrow{f} Y \xrightarrow{p} FX \overset {r_f} \dashrightarrow$ and  $A \xrightarrow{g} B \xrightarrow{q} FX \overset {r_g} \dashrightarrow$
with $$\ell (l_f)=p,  \ \ -\rho ^{-1}(i)=r_f, \ \ \ell (l_g)=q, \ \ -\rho^{-1}(j)=r_g.$$ By functoriality of $\rho$ there is a commutative diagram
            \begin{equation*}
\xymatrix@R=0.4cm{
  \mathcal C(K, X)\ar[r]^{{\rho^{-1}}}\ar[d]_{{\mathcal C(K, u)}} & \mathbb E(FK, X)\ar[d]^{{{\mathbb E}(FK, u)}} \\
  \mathcal C(K, A)\ar[r]^{{\rho^{-1}}} & {\mathbb E}(FK, A)
}
\end{equation*}
Thus, $u_\ast r_f=-u_\ast(\rho^{-1}(i))=-\rho^{-1}(ui)=-\rho^{-1}(j)=r_g$. Since $v^\ast l_g = l_f$, by functoriality of $\ell$ one has $qv = p$. By definition $i_\ast l_g = i_\ast \ell^{-1}(q)=i_\ast q^\ast \delta_K=q^\ast i_\ast  \delta_K = q^\ast \rho^{-1}(i)=- q^\ast r_f$. Since $(1_K,u,v)$ is homotopic, $\con X{\dg{-f \\ u}}{Y\oplus A}{\dg{v,g}}B{i_\ast l_g}$ is an $\mathbb E$-triangle. Since $i_\ast l_g = - q^\ast r_f$, $(u,v, 1_{FK})$ is homotopic. \end{proof}

\subsection{$4 \times 4$ Lemma}  \ The following result is a generalization of the {\rm $4 \times 4$ Lemma} in triangulated categories.

\vskip 5pt

\begin{theorem}\label{thm:4x4-sss} \ Let $(\mathcal{C}, \mathbb E, \mathfrak s)$ be an extriangulated category. Suppose that
		\begin{equation*}
\xymatrix@R=0.4cm{
  X\ar[r]^{f}\ar[d]_-{\alpha} & Y\ar[r]^{g} & Z\ar@{-->}[r]^{\delta} \ar[d]^-{\gamma}& {} \\
  L\ar[r]^{u} & M\ar[r]^{v} & N\ar@{-->}[r]^{{\varepsilon }} & {}
}
\end{equation*}
is a diagram of $\mathbb E$-triangles such that  $\alpha_\ast \delta=\gamma^\ast \varepsilon$, and $\alpha$ and $\gamma$ are $\mathbb E$-inflations and $\mathbb E$-deflations.	
Then there is an $\mathbb E$-inflation and an $\mathbb E$-deflation $\beta$ such that $(\alpha, \beta, \gamma)$ is a homotopic morphism of $\mathbb E$-triangles. Moreover, there is a commutative diagram
\begin{equation*}
\xymatrix@R=0.4cm{
  K_\alpha\ar@{..>}[r]^{{f'}}\ar@{..>}[d]_{{i_\alpha }} & K_\beta\ar@{..>}[r]^{{g'}}\ar@{..>}[d]^{{i_\beta}} & K_\gamma\ar@{..>}[d]^{{i_\gamma}}\ar@{..>}[r] & {} \\
  X\ar[r]^{f}\ar[d]_-{\alpha} & Y\ar[r]^{g}\ar@{..>}[d]^-{\beta} & Z\ar@{-->}[r]^{\delta}\ar[d]^-{\gamma} & {} \\
  L\ar[r]^{u}\ar@{..>}[d]_{{p_\alpha}} & M\ar[r]^{v}\ar@{..>}[d]^{{p_\beta}} & N\ar@{-->}[r]^{{\varepsilon }}\ar@{..>}[d]^{{p_\gamma}} & {} \\
  C_\alpha\ar@{..>}[r]^{{u'}} & C_\beta\ar@{..>}[r]^{{v'}} & C_\gamma \ar@{..>}[r] & {}
}
\end{equation*}
with connecting morphism $z: K_\gamma\to C_\alpha$, such that the following are $\mathbb E$-triangles$:$
\vskip5pt
\begin{tasks}[style=enumerate, label=$(\arabic*)$, label-width=17.8pt](2)
		\task $\con {K_\alpha}{f'}{K_\beta}{g'}{K_\gamma}{\ell^{-1} (z)};$
		\task $\con {K_\beta}{g'}{K_\gamma}z{C_\alpha}{-\ell^{-1} (u')};$
		\task $\con {K_\gamma}z{C_\alpha}{u'}{C_\beta}{\rho^{-1}(g')};$
		\task $\con {C_\alpha}{u'}{C_\beta}{v'}{C_\gamma}{-\rho^{-1}(z)};$
		\task $\con {K_\alpha}{i_\alpha}X\alpha L{\ell^{-1}(p_\alpha)};$
		\task $\con {K_\beta}{i_\beta}Y\beta M{\ell^{-1}(p_\beta)};$
		\task $\con {K_\gamma}{i_\gamma}Z\gamma N{\ell^{-1}(p_\gamma)};$
		\task $\con X\alpha L{p_\alpha}{C_\alpha}{-\rho^{-1}(i_\alpha)};$
		\task $\con Y\beta M{p_\beta}{C_\beta}{-\rho^{-1}(i_\beta)};$
		\task $\con Z\gamma N{p_\gamma}{C_\gamma}{-\rho^{-1}(i_\gamma)}$
	\end{tasks}
\vskip10pt
with morphisms of $\mathbb E$-triangles $(i_\alpha,i_\beta,i_\gamma)$, $(p_\alpha,p_\beta,p_\gamma)$, $(f',f,u)$, $(f,u, u')$, $(g',g,v)$, $(g,v,v')$,
and that $C_\alpha=FK_\alpha$, $C_\beta=FK_\beta$, $C_\gamma=FK_\gamma$, $u'=Ff'$, $v'=Fg'$, where $F$ is the functor defined in \Cref{F}, and $\ell$ and $\rho$ are defined in \Cref{binaturality}.

\vskip5pt

In particular, there is a sequence
	\begin{equation*}		K_\alpha \xrightarrow{f'} K_\beta \xrightarrow{g'} K_\gamma \xrightarrow{z} C_\alpha \xrightarrow{u'} C_\beta \xrightarrow{v'} C_\gamma
\end{equation*}
such that any subsequent three objects form an $\mathbb E$-triangle as in $(1)$-$(4)$.
\end{theorem}
\begin{proof} \ By \Cref{thm:hmor}$(1)$ and $(1)'$, there is a homotopic morphism of $\mathbb E$-triangles $(\alpha, \beta, \gamma)$,
where $\beta: Y\to M$ is an $\mathbb E$-inflation and an $\mathbb E$-deflation. Then by definition there is a decomposition with $\beta = \beta_2 \beta_1:$
        \begin{equation}
\label{eq:4x4-sss--1}
\xymatrix@R=0.4cm{
  X\ar[r]^{f}\ar[d]_-{\alpha} & Y\ar[r]^{g}\ar[d]_{{\beta_1}} & Z\ar@{-->}[r]^{\delta}\ar@{=}[d] & {} \\
  L\ar[r]^{s}\ar@{=}[d] & E\ar[r]^{t}\ar[d]_{{\beta_2}} & Z\ar@{-->}[r]^{\kappa}\ar[d]^-{\gamma} & {} \\
  L\ar[r]^{u} & M\ar[r]^{v} & N\ar@{-->}[r]^{\varepsilon} & {}
}
\end{equation}
where $\con X{\dg{f \\\alpha}}{Y \oplus L}{\dg{\beta_1,-s}}E{t^\ast\delta}$ and $\con E{\dg{-t\\\beta_2}}{Z\oplus M}{\dg{\gamma , v}}N{s_\ast\varepsilon }$ are $\mathbb E$-triangles and  $\alpha_\ast\delta=\kappa =\gamma^\ast\varepsilon$. By \Cref{prop:hs-inflation}, $\beta_1$ and $\beta_2$ are $\mathbb E$-inflations and $\mathbb E$-deflations. By \Cref{prop:strict-et4op-2-h} and \Cref{prop:pb-2-h} there are commutative diagrams
\begin{equation}
\label{eq:4x4-sss-0}
\xymatrix@R=0.4cm{
  K_\alpha\ar@{=}[r]\ar[d]_{{i_\alpha}} & K_\alpha\ar[d]^{{fi_\alpha}} & {} & {} & {} & K_\gamma\ar@{=}[r]\ar[d]_{x} & K_\gamma\ar[d]^{{i_\gamma}} & {} \\
  X\ar[r]^{f}\ar[d]_-{\alpha} & Y\ar[r]^{g}\ar[d]^{{\beta_1}} & Z\ar@{-->}[r]^{\delta}\ar@{=}[d] & {} & L\ar[r]^{s}\ar@{=}[d] & E\ar[r]^{t}\ar[d]_{{\beta_2}} & Z\ar@{-->}[r]^{\kappa}\ar[d]^-{\gamma} & {} \\
  L\ar[r]^{s}\ar@{-->}[d]_{{l_\alpha}} & E\ar[r]^{t}\ar@{-->}[d]^{{\eta_1}} & Z\ar@{-->}[r]^{\kappa} & {} & L\ar[r]^{u} & M\ar[r]^{v}\ar@{-->}[d]_{{\mu_2}} & N\ar@{-->}[r]^{{\varepsilon }}\ar@{-->}[d]^{{l_\gamma}} & {} \\
  {} & {} & {} & {} & {} & {} & {} & {}
}
\end{equation}
where $l_\alpha=s^\ast{\eta_1}$, $(i_\alpha)_\ast\eta_1=t^\ast\delta$ and $\mu_2=v^\ast l_\gamma$, $x_\ast l_\gamma+s_\ast\varepsilon=0$. By \Cref{thm:hs-square-s} there are homotopic morphisms of $\mathbb E$-triangles $(f,s,1_{C_\alpha})$ and $(t,v,1_{C_\gamma})$

		\begin{equation}
\label{eq:4x4-lem-steps}
\xymatrix@R=0.4cm{
  X\ar[r]^-{\alpha}\ar[d]_{f} & L\ar[r]^{{p_\alpha}}\ar[d]^{s} & C_\alpha\ar@{-->}[r]^{{r_\alpha}}\ar@{=}[d] & {} & E\ar[r]^{{\beta_2}}\ar[d]_{t} & M\ar[r]^{{p_\gamma v}}\ar[d]^{v} & C_\gamma\ar@{-->}[r]^{{\eta_2}}\ar@{=}[d] & {} \\
  Y\ar[r]^{{\beta_1}} & E\ar[r]^{y} & C_\alpha\ar@{-->}[r]^{{\mu _1}} & {} & Z\ar[r]^-{\gamma} & N\ar[r]^{{p_\gamma}} & C_\gamma\ar@{-->}[r]^{{r_\gamma}} & {}
}
\end{equation}
    where $$C_\alpha=FK_\alpha, \ (i_\alpha)_\ast \eta_1 = -y^\ast r_\alpha, \ p_\alpha=\ell(l_\alpha), \ y=\ell(\eta_1), \ \rho(r_\alpha)=-i_\alpha, \ \rho(\mu_1)=-(fi_\alpha),$$
    $$C_\gamma=FK_\gamma, \ x_\ast l_\gamma = -(p_\gamma)^\ast \eta_2, \ p_\gamma v=\ell(\mu_2), \ p_\gamma=\ell(l_\gamma), \ \rho(\eta_2)=-x, \ \rho(r_\gamma)=-(i_\gamma).$$ Since $\beta$ is an $\mathbb E$-deflation, there is an $\mathbb E$-triangle $\con{K_\beta}{i_\beta}{Y}{\beta}{M}{l_\beta}$. By \Cref{lem:hs-?g1}$(2)$ there is $g': K_\beta\to K_\gamma$ such that $(g',\beta_1,1_M)$ is a homotopic morphism. By \Cref{prop:strict-et4op-2-h}$(2)$ there is an $\mathbb E$-triangle $\con {K_\alpha}{f'}{K_\beta}{g'}{K_\gamma}{x^\ast \eta_1}$,  as in the following commutative diagram:
\begin{equation}
\label{eq:4x4-sss-1}
\xymatrix@R=0.4cm{
  K_\alpha\ar@{..>}[r]^{{f'}}\ar@{=}[d] & K_\beta\ar@{..>}[r]^{{g'}}\ar[d]_{{i_\beta}} & K_\gamma\ar@{..>}[r]^{{x^\ast \eta_1}}\ar[d]^{x} & {} \\
  K_\alpha\ar[r]^{{fi_\alpha}} & Y\ar[r]^{{\beta_1}}\ar[d]_-{\beta} & E\ar@{-->}[r]^{{\eta_1}}\ar[d]^{{\beta_2}} & {} \\
  {} & M\ar@{=}[r]\ar@{-->}[d]_{{l_\beta}} & M\ar@{-->}[d]^{{\mu_2}} & {} \\
  {} & {} & {} & {}
}
\end{equation}
where $(\beta_2)^\ast l_\beta = (f')_\ast \eta_1$. By \Cref{prop:hs-inflation}, $g'$ is both an $\mathbb E$-inflation and an $\mathbb E$-deflation. Thus, by \Cref{thm:hs-square-s} there is an $\mathbb E$-triangle $\con {K_\beta}{g'}{K_\gamma}{z}{C_\alpha}{\theta}$

\begin{equation}
\label{eq:4x4-sss-2}
\xymatrix@R=0.4cm{
  K_\beta\ar[r]^{{g'}}\ar[d]_{{i_\beta}} & K_\gamma\ar@{..>}[r]^{z}\ar[d]^{x} & C_\alpha\ar@{..>}[r]^{\theta}\ar@{=}[d] & {} \\
  Y\ar[r]^{{\beta_1}}\ar[d]_-{\beta} & E\ar[r]^{y}\ar[d]^{{\beta_2}} & C_\alpha\ar@{-->}[r]^{{\mu_1}} & {} \\
  M\ar@{=}[r]\ar@{-->}[d]_{{l_\beta}} & M\ar@{-->}[d]^{{\mu_2}} & {} & {} \\
  {} & {} & {} & {}
}
\end{equation}
where $z=yx$, $(i_{\beta})_\ast \theta=\mu_1$, $(f')_\ast\eta_1=-y^\ast\theta$, $\ell(x^\ast \eta_1)=z$, $\rho(\theta)=-f'$.

\vskip5pt

Since $\beta$ is both an $\mathbb E$-inflation and an $\mathbb E$-deflation, by \Cref{lem:rotation} there is an $\mathbb E$-triangle $\con Y \beta M{p_\beta} {C_\beta} {r_\beta}$ arising from the $\mathbb E$-triangle $\con {K_\beta} {i_\beta} Y \beta M {l_\beta}$, where $C_\beta=FK_\beta$, $p_\beta=\ell(l_\beta)$ and $\rho(r_\beta)=-i_\beta$. By the $\mathbb E$-triangle $\con{K_\beta}{i_\beta}{Y}{\beta}{M}{l_\beta}$ and \Cref{prop:slr-2}$(1)$, $K_\beta \in \mathcal L ;$ and then by the $\mathbb E$-triangle $\con{K_\beta}{g'}{K_\gamma}{z}{C_\alpha}{\theta}$ and \Cref{prop:slr-2}$(2)$, $z$ is both an $\mathbb E$-inflation and an $\mathbb E$-deflation. By \Cref{lem:rotation} and \Cref{isomorphisms} there is an $\mathbb E$-triangle $\con{K_\gamma}{z}{C_\alpha}{u'}{C_\beta}{\rho^{-1}(g')}$ where $u'=-\ell(\theta)$. Since $(f')_\ast \delta_{K_\alpha}=\rho^{-1}(f')=-\theta=\ell^{-1}(u')=(u')^{\ast}\delta_{K_\beta}$, $u'=Ff'$ by definition.

\vskip5pt

Finally, by the $\mathbb E$-triangle $\con{K_\gamma}{i_\gamma}{Z}{\gamma}{N}{p_\gamma}$ and \Cref{prop:slr-2}(1), $K_\gamma \in \mathcal L;$ and then according to the $\mathbb E$-triangle $\con{K_\gamma}{z}{C_\alpha}{u'}{C_\beta}{\rho^{-1}(g')}$
and \Cref{prop:slr-2}(2), $u'$ is both an $\mathbb E$-inflation and an $\mathbb E$-deflation. Thus, by \Cref{lem:rotation} and \Cref{isomorphisms} there is an $\mathbb E$-triangle $\con {C_\alpha}{u'}{C_\beta}{v'}{C_\gamma}{-\rho^{-1}(z)}$ where $v'=\ell(\rho^{-1}(g'))$. Since $(g')_\ast \delta_{K_\beta}=\rho^{-1}(g')=\ell^{-1}(v')=(v')^{\ast}\delta_{K_\gamma}$, $v'=Fg'$ by definition.

\vskip10pt

Up to now the $\mathbb E$-triangles $(1)$-$(10)$ are obtained. It remains to verify the morphisms of $\mathbb E$-triangles.

\vskip5pt

By \Cref{eq:4x4-sss-1} one has $fi_\alpha=i_\beta f'$. Also one has
$gi_\beta\xlongequal{\text{\Cref{eq:4x4-sss-0}}}t\beta_1i_\beta \xlongequal{\text{\Cref{eq:4x4-sss-1}}} txg'\xlongequal{\text{\Cref{eq:4x4-sss-0}}}i_\gamma g'$, and
    \begin{align*}
        (i_\alpha)_\ast \ell^{-1}(z) & \xlongequal{\text{\Cref{eq:4x4-sss-2}}} (i_\alpha)_\ast x^\ast\eta_1=x^\ast(i_\alpha)_\ast\eta_1 \xlongequal{\text{\Cref{eq:4x4-sss-2}}} x^\ast t^\ast \delta =(tx)^\ast \delta  \xlongequal{\text{\Cref{eq:4x4-sss-0}}} (i_\gamma)^\ast \delta.
    \end{align*}
This shows that $(i_\alpha, i_\beta, i_\gamma)$ is a morphism of $\mathbb E$-triangles.

\vskip5pt

One has
    \begin{align*}
        u'p_\alpha & = -\ell (\theta) p_\alpha\xlongequal{\text{\Cref{eq:4x4-lem-steps}}} -\ell (\theta) y s \xlongequal{\text{\Cref{binaturality}}} - \ell (y^\ast\theta)s \\
        & = \ell (- y^\ast\theta)s  \xlongequal{\text{\Cref{eq:4x4-sss-1}}} \ell ((f')_\ast \eta_1)s\xlongequal{\text{\Cref{eq:4x4-sss-2}}}\ell ((\beta_2)^\ast l_\beta)s \xlongequal{\text{\Cref{binaturality}}} \ell (l_\beta) \beta_2 s\\
        & = p_\beta \beta_2 s \xlongequal{\text{\Cref{eq:4x4-sss--1}}}p_\beta u.
    \end{align*}
and
\begin{align*}
        u'p_\alpha & = -\ell (\theta) p_\alpha\xlongequal{\text{\Cref{eq:4x4-lem-steps}}} -\ell (\theta) y s \xlongequal{\text{\Cref{binaturality}}} - \ell (y^\ast\theta)s \\
        & = \ell (- y^\ast\theta)s  \xlongequal{\text{\Cref{eq:4x4-sss-1}}} \ell ((f')_\ast \eta_1)s\xlongequal{\text{\Cref{eq:4x4-sss-2}}}\ell ((\beta_2)^\ast l_\beta)s \xlongequal{\text{\Cref{binaturality}}} \ell (l_\beta) \beta_2 s\\
        & = p_\beta \beta_2 s \xlongequal{\text{\Cref{eq:4x4-sss--1}}}p_\beta u.
    \end{align*}
Note that
    \begin{align*}
        (p_\beta)^\ast (\rho^{-1}(g')) & \xlongequal{\text{\Cref{thm:lr-equivalence}$(2)$}}  (p_\beta)^\ast (g')_\ast\delta_{K_\beta}= (g')_\ast (p_\beta)^\ast\delta_{K_\beta}\\
        & \xlongequal{\text{\Cref{thm:lr-equivalence}$(1)$}} (g')_\ast \ell^{-1}(p_\beta)=(g')_\ast l_\beta\xlongequal{\text{\Cref{eq:4x4-sss-1}}}\mu_2,
    \end{align*}
one has
    \begin{align*}
        v'p_\beta & = \ell (\rho^{-1}(g')) p_\beta \xlongequal{\text{\Cref{binaturality}}} \ell ((p_\beta)^\ast (\rho^{-1}(g')))\\
        & =\ell(\mu_2) \xlongequal{\text{\Cref{eq:4x4-lem-steps}}} p_\gamma v.
    \end{align*}
This shows that $(p_\alpha, p_\beta, p_\gamma)$ is a morphism of $\mathbb E$-triangles.

\vskip5pt

Note that
    \begin{align*}
        (f')_\ast (\ell^{-1}(p_\alpha))& \xlongequal{\text{\Cref{eq:4x4-lem-steps}}}(f')_\ast l_\alpha \xlongequal{\text{\Cref{eq:4x4-sss-0}}} (f')_\ast s^\ast \eta_1 =  s^\ast (f')_\ast \eta_1\\
        & \xlongequal{\text{\Cref{eq:4x4-sss-1}}} s^\ast (\beta_2)^\ast l_\beta=(\beta_2 s)^\ast l_\beta \\
        & \xlongequal{\text{\Cref{eq:4x4-sss--1}}} u^\ast l_\beta =u^\ast (\ell^{-1}(p_\beta)).
    \end{align*}
Since $fi_\alpha=i_\beta f'$ and $u\alpha = \beta f$, $(f',f,u)$ is a morphism of $\mathbb E$-triangles.

\vskip5pt

 Note that $f_\ast (-\rho^{-1}(i_\alpha)) \xlongequal{\text{\Cref{eq:4x4-lem-steps}}}f_\ast r_\alpha \xlongequal{\text{\Cref{eq:4x4-lem-steps}}} \mu_1$, one has
    \begin{align*}
        \rho((u')^\ast (-\rho^{-1}(i_\beta))) & =\rho((Ff')^\ast (-\rho^{-1}(i_\beta))) \xlongequal{\text{\Cref{binaturality}}} -i_\beta f'  \\
        & \xlongequal{(1)}-fi_\alpha \xlongequal{\text{\Cref{eq:4x4-lem-steps}}} \rho(\mu_1).
    \end{align*}
    Hence $f_\ast (-\rho^{-1}(i_\alpha))=\mu_1=(u')^\ast (-\rho^{-1}(i_\beta))$. Since $u\alpha = \beta f$ and $u'p_\alpha=p_\beta u$, $(f,u,u')$ is a morphism of $\mathbb E$-triangles.

\vskip5pt

 Note that $(g')_\ast (\ell^{-1}(p_\beta))=(g')_\ast l_\beta \xlongequal{\text{\Cref{eq:4x4-sss-1}}}\mu_2  \xlongequal{\text{\Cref{eq:4x4-sss-0}}} v^\ast (\ell^{-1}(p_\gamma))$. Since $i_\gamma g' = g i_\beta$ and $\gamma g = v \beta$, $(g',g,v)$ is a morphism of $\mathbb E$-triangles.

\vskip5pt

 One has
     \begin{align*}
        \rho ((v')^\ast (-\rho^{-1}(i_\gamma))) & =\rho ((Fg')^\ast (-\rho^{-1}(i_\gamma)))\xlongequal{\text{\Cref{binaturality}}}-i_\gamma g'\\
        & \xlongequal{(2)}-gi_\beta \xlongequal{\text{\Cref{binaturality}}} \rho(g_\ast (-\rho^{-1}(i_\beta))) .
    \end{align*}
 Hence $(v')^\ast (-\rho^{-1}(i_\gamma)) = g_\ast (-\rho^{-1}(i_\beta))$. Since $\gamma g = v \beta$ and $v' p_\beta = p_\gamma v$, $(g,v,v')$ is a morphism of $\mathbb E$-triangles.

\vskip5pt

This completes the proof. \end{proof}

\subsection{$3 \times 3$ diagram and Horseshoe Lemma} $\,$\vskip5pt

\begin{definition}\label{def:3x3-diag} \ A \textit{$3 \times 3$ diagram} in an extriangulated category $(\mathcal{C}, \ \mathbb E, \ \mathfrak s)$ is a diagram of six $\mathbb E$-triangles
	\begin{equation*}
\label{eq:3x3}
\xymatrix@R=0.4cm{
  A_1\ar[r]^{{f_A}}\ar[d]_{{i_1}} & A_2\ar[r]^{{g_A}}\ar[d]^{{i_2}} & A_3\ar@{-->}[r]^{{\delta_A}}\ar[d]^{{i_3}} & {} \\
  B_1\ar[r]^{{f_B}}\ar[d]_{{p_1}} & B_2\ar[r]^{{g_B}}\ar[d]^{{p_2}} & B_3\ar@{-->}[r]^{{\delta_B}}\ar[d]^{{p_3}} & {} \\
  C_1\ar[r]^{{f_C}}\ar@{-->}[d]_{{\varepsilon _1}} & C_2\ar[r]^{{g_C}}\ar@{-->}[d]^{{\varepsilon _2}} & C_3\ar@{-->}[r]^{{\delta_C}}\ar@{-->}[d]^{{\varepsilon _3}} & {} \\
  {} & {} & {} & {}
}
\end{equation*}
such that $(i_1,i_2,i_3)$, $(p_1,p_2,p_3)$, $(f_A,f_B,f_C)$ and $(g_A,g_B,g_C)$ are morphisms of $\mathbb E$-triangles.
\end{definition}

\vskip5pt

\begin{theorem} \label{cor:3x3-lemma-inflation}\label{thm:3x3-lemma} \ Let $(\mathcal{C}, \mathbb E, \mathfrak s)$ be an extriangulated category. Suppose that
		\begin{equation*}
\xymatrix@R=0.4cm{
  A_1\ar[r]^{f_A}\ar[d]_-{i_1} & A_2\ar[r]^{g_A} & A_3\ar@{-->}[r]^{\delta_A} \ar[d]^-{i_3}& {} \\
  B_1\ar[r]^{f_B} & B_2\ar[r]^{g_B} & B_3\ar@{-->}[r]^{{\delta_B }} & {}
}
\end{equation*}
is a diagram of $\mathbb E$-triangles such that  $(i_1)_\ast \delta_A=(i_3)^\ast \delta_B$, and $i_1$ and $i_3$ are $\mathbb E$-inflations.	
Then there is an $\mathbb E$-inflation $i_2$ such that $(i_1, i_2, i_3)$ is a homotopic morphism of $\mathbb E$-triangles. Moreover, there is a $3\times 3$ diagram.

\end{theorem}
	\begin{proof} By \Cref{thm:hs-morphism} there is $i_2: A_2\to B_2$ such that $(i_1,i_2,i_3)$ is a homotopic morphism. That is, there is a decomposition
		\begin{equation}
\label{eq:pf-3x3-0}
\xymatrix@R=0.4cm{
  A_1\ar[r]^{{f_A}}\ar[d]_{{i_1}} & A_2\ar[r]^{{g_A}}\ar[d]^{{j_1}} & A_3\ar@{-->}[r]^{{\delta_A}}\ar@{=}[d] & {} \\
  B_1\ar[r]^{s}\ar@{=}[d] & E\ar[r]^{t}\ar[d]^{{j_2}} & A_3\ar@{-->}[r]^{\kappa}\ar[d]^{{i_3}} & {} \\
  B_1\ar[r]^{{f_B}} & B_2\ar[r]^{{g_B}} & B_3\ar@{-->}[r]^{{\delta_B}} & {}
}
\end{equation}
with $j_2j_1=i_2$	such that  $(i_1, j_1, 1)$ and $(1, j_2, i_3)$ are morphisms of $\mathbb E$-triangles, and
$\con{A_1}{\dg{f_A \\ i_1}}{A_2\oplus B_1}{(j_1, -s)}{E}{t^\ast \delta_A}$ and $\con{E}{\dg{-t \\ j_2}}{A_3\oplus B_3}{(i_3, g_B)}{B_3}{ s_\ast \delta_B}$ are $\mathbb E$-triangles,
and  $(i_1)_\ast \delta_A = \kappa =  (i_3)^\ast \delta_B$. By \Cref{prop:po-1-h}$(2)$ and \Cref{prop:strict-et4-2-h}$(2)$, there exist $q_1$ and $q_2$ making the diagrams below commute
\begin{equation}
\label{eq:pf-3x3-1}
\xymatrix@R=0.4cm{
  A_1\ar[r]^{{f_A}}\ar[d]_{{i_1}} & A_2\ar[r]^{{g_A}}\ar[d]^{{j_1}} & A_3\ar@{-->}[r]^{{\delta_A}}\ar@{=}[d] & {} & B_1\ar[r]^{s}\ar@{=}[d] & E\ar[r]^{t}\ar[d]^{{j_2}} & A_3\ar@{-->}[r]^{\kappa}\ar[d]^{{i_3}} & {} \\
  B_1\ar[r]^{s}\ar[d]_{{p_1}} & E\ar[r]^{t}\ar@{..>}[d]^{{q_1}} & A_3\ar@{-->}[r]^{\kappa} & {} & B_1\ar[r]^{{f_B}} & B_2\ar[r]^{{g_B}}\ar@{..>}[d]^{{q_2}} & B_3\ar@{-->}[r]^{{\delta_B}}\ar[d]^{{p_3}} & {} \\
  C_1\ar@{-->}[d]_{{\varepsilon _1}} & C_1\ar@{=}[l]\ar@{..>}[d]^{{(f_A)_\ast \varepsilon _1}} & {} & {} & {} & C_3\ar@{=}[r]\ar@{..>}[d]^{\theta} & C_3\ar@{-->}[d]^{{\varepsilon _3}} & {} \\
  {} & {} & {} & {} & {} & {} & {} & {}
}
\end{equation}
such that $t^\ast \delta_A + q^\ast \varepsilon_1 = 0$ and $(p_3)^\ast \theta = s_\ast \delta_B$.

\vskip5pt

By ET4 there are $\mathbb E$-triangles $\con {A_2}{i_2}{B_2}{p_2}{C_2}{\varepsilon_2}$ and $\con {C_1}{f_C}{C_2}{g_C}{C_3}{\delta_C}$ fit in the following commutative diagram satisfying that $\delta_C=(q_1)_\ast \theta$, $(f_C)^\ast (\varepsilon_2)=(f_A)_\ast(\varepsilon_1)$ and $(j_1)_\ast (\varepsilon_2)=(g_C)^\ast\theta$.

\begin{equation}
\label{eq:pf-3x3-2}
\xymatrix@R=0.4cm{
  A_2\ar[r]^{{j_1}}\ar@{=}[d] & E\ar[r]^{{q_1}}\ar[d]_{{j_2}} & C_1\ar@{-->}[r]^{{(f_A)_\ast \varepsilon _1}}\ar@{..>}[d]_{{f_C}} & {} \\
  A_2\ar@{..>}[r]^{{i_2}} & B_2\ar@{..>}[r]^{{p_2}}\ar[d]_{{q_2}} & C_2\ar@{..>}[r]^{{\varepsilon_2}}\ar@{..>}[d]_{{g_C}} & {} \\
  {} & C_3\ar@{=}[r]\ar@{-->}[d]_{\theta} & C_3\ar@{..>}[d]_{{\delta_C}} & {} \\
  {} & {} & {} & {}
}
\end{equation}
\vskip5pt

Putting  \Cref{eq:pf-3x3-0}, \Cref{eq:pf-3x3-1} and \Cref{eq:pf-3x3-2} together, one obtains the following diagram

\begin{equation*}\xymatrix@R=0.4cm{
  A_1\ar[r]^{{f_A}}\ar[d]_{{i_1}} & A_2\ar[r]^{{g_A}}\ar[d]^{{i_2}} & A_3\ar@{-->}[r]^{{\delta_A}}\ar[d]^{{i_3}} & {} \\
  B_1\ar[r]^{{f_B}}\ar[d]_{{p_1}} & B_2\ar[r]^{{g_B}}\ar[d]^{{p_2}} & B_3\ar@{-->}[r]^{{\delta_B}}\ar[d]^{{p_3}} & {} \\
  C_1\ar[r]^{{f_C}}\ar@{-->}[d]_{{\varepsilon _1}} & C_2\ar[r]^{{g_C}}\ar@{-->}[d]^{{\varepsilon _2}} & C_3\ar@{-->}[r]^{{\delta_C}}\ar@{-->}[d]^{{\varepsilon _3}} & {} \\
  {} & {} & {} & {}
}
\end{equation*} It is indeed a $3\times 3$ diagram. In fact, by assumption one has $f_Bi_1 = i_2f_A$,  $g_Bi_2 = i_3g_A$,  and $(i_1)_\ast\delta_A = (i_3)^\ast\delta_B$. One has
$$f_Cp_1 \xlongequal{\text{\Cref{eq:pf-3x3-1}}} f_Cq_1s\xlongequal{\text{\Cref{eq:pf-3x3-2}}} p_2j_2s\xlongequal{\text{\Cref{eq:pf-3x3-0}}} p_2f_B, \ \ \ \ \
g_Cp_2 \xlongequal{\text{\Cref{eq:pf-3x3-2}}} q_2 \xlongequal{\text{\Cref{eq:pf-3x3-1}}} p_3g_B$$
and
$$(p_1)_\ast\delta_B \xlongequal{\text{\Cref{eq:pf-3x3-1}}} (q_1)_\ast s_\ast\delta_B \xlongequal{\text{\Cref{eq:pf-3x3-1}}} (q_1)_\ast (p_3)^\ast\theta  = (p_3)^\ast (q_1)_\ast \theta \allowbreak \xlongequal{\text{\Cref{eq:pf-3x3-2}}} (p_3)^\ast\delta_C.$$
It follows from \Cref{eq:pf-3x3-2} that $(f_A)_\ast\varepsilon_1 = (f_C)^\ast \varepsilon_2;$ and one has
$$(g_A)_\ast\varepsilon_2 \xlongequal{\text{\Cref{eq:pf-3x3-0}}} t_\ast (j_1)_\ast \varepsilon_2 \xlongequal{\text{\Cref{eq:pf-3x3-2}}} t_\ast (g_C)^\ast \theta = (g_C)^\ast t_\ast \theta \allowbreak\xlongequal{\text{\Cref{eq:pf-3x3-1}}} (g_C)^\ast \varepsilon_3.$$
This completes the proof. \end{proof}

\vskip5pt

\begin{corollary} {\rm (Horseshoe Lemma)} \label{horseshoe} \ Let $(\mathcal{C}, \mathbb E, \mathfrak s)$ be an extriangulated category, and  $\con{A_1}{f_A}{A_2}{g_A}{A_3}{\delta_A}$ an $\mathbb E$-triangle.
Suppose that $i_1 : A_1 \to B_1$ and $i_3 : A_3 \to B_3$ are $\mathbb E$-inflations where $i_1$ factors through $f_A$. Then there is a $3\times 3$ diagram$:$
	\begin{equation*}
\xymatrix@R=0.4cm{
  A_1\ar[r]^{f_A}\ar[d]_-{i_1} & A_2\ar[r]^{g_A}\ar@{..>}[d] & A_3\ar@{-->}[r]^{\delta_A}\ar[d]^-{i_3} & {} \\
  B_1\ar[r]^-{{\dg{1 \\ 0}}}\ar[d] & B_1  \oplus B_3\ar[r]^-{{(0,1)}}\ar@{..>}[d] & B_3\ar@{-->}[r]^{0}\ar[d] & {} \\
  C_1\ar@{..>}[r]\ar@{-->}[d] & C_2\ar@{..>}[r]\ar@{..>}[d] & C_3\ar@{..>}[r]\ar@{-->}[d] & {} \\
  {} & {} & {} & {}
}
\end{equation*}
\end{corollary}

\subsection{Other variants of $4\times 4$ Lemma}

Since the class of $\mathbb E$-triangles is not closed under rotations, $4\times 4$ Lemma has some variants. All proofs are similar to \Cref{thm:4x4-sss}: one constructs the homotopic morphism $(\alpha, \beta, \gamma)$ by \Cref{thm:hmor} and then follows the proof of \Cref{thm:4x4-sss}. So, we just state the results without the proofs. The dual statements are also omitted.

\vskip5pt

\begin{proposition}\label{prop:4x4-I-?-S} \ Let $(\mathcal{C}, \mathbb E, \mathfrak s)$ be an extriangulated category. Suppose that
		\begin{equation*}
\xymatrix@R=0.4cm{
  X\ar[r]^{f}\ar[d]_-{\alpha} & Y\ar[r]^{g} & Z\ar@{-->}[r]^{\delta}\ar[d]^-{\gamma} & {} \\
  L\ar[r]^{u} & M\ar[r]^{v} & N\ar@{-->}[r]^{{\varepsilon }} & {}
}
\end{equation*}
is a diagram of $\mathbb E$-triangles such that $\alpha_\ast \delta=\gamma^\ast \varepsilon$, and $\alpha$ is an $\mathbb E$-inflation, and $\gamma$ is an $\mathbb E$-inflation and an $\mathbb E$-deflation. Then there is an $\mathbb E$-inflation $\beta$ such that $(\alpha, \beta, \gamma)$ is a homotopic morphism of $\mathbb E$-triangles. Moreover, there is a commutative diagram
\begin{equation*}
\xymatrix@R=0.4cm{
   &   & K_\gamma\ar@{..>}[d]^{{i_\gamma}} & {} \\
  X\ar[r]^{f}\ar[d]_-{\alpha} & Y\ar[r]^{g}\ar@{..>}[d]^-{\beta} & Z\ar@{-->}[r]^{\delta}\ar[d]^-{\gamma} & {} \\
  L\ar[r]^{u}\ar@{..>}[d]_{{p_\alpha}} & M\ar[r]^{v}\ar@{..>}[d]^{{p_\beta}} & N\ar@{-->}[r]^{{\varepsilon }}\ar@{..>}[d]^{{p_\gamma}} & {} \\
  C_\alpha\ar@{..>}[r]^{{u'}} & C_\beta\ar@{..>}[r]^{{v'}} & C_\gamma \ar@{..>}[r] & {}
}
\end{equation*}
    	 with a connecting morphism $z:K_\gamma\to C_\alpha$, such that the following are $\mathbb E$-triangles$:$
             \vskip5pt
\begin{tasks}[style=enumerate, label=$(\arabic*)$, label-width=17.8pt](2)
		\task $\con {K_\gamma}z{C_\alpha}{u'}{C_\beta}{\rho^{-1}(g')};$
		\task $\con {C_\alpha}{u'}{C_\beta}{v'}{C_\gamma}{-\rho^{-1}(z)};$
		\task $\con {K_\gamma}{i_\gamma}Z\gamma N{\ell^{-1}(p_\gamma)};$
		\task $\con X\alpha L{p_\alpha}{C_\alpha}{-\rho^{-1}(i_\alpha)};$
		\task $\con Y\beta M{p_\beta}{C_\beta}{-\rho^{-1}(i_\beta)};$
		\task $\con Z\gamma N{p_\gamma}{C_\gamma}{-\rho^{-1}(i_\gamma)}$
	\end{tasks}
        \vskip10pt
with morphisms of $\mathbb E$-triangles $(p_\alpha,p_\beta,p_\gamma)$, $(f,u, u')$, $(g,v,v')$,
and $C_\gamma=FK_\gamma$, where $F$ is the functor defined in \Cref{F}, and $\ell$ and $\rho$ are defined in \Cref{binaturality}.
    \vskip5pt
    In particular, there is a sequence
$ K_\gamma \xrightarrow{z} C_\alpha \xrightarrow{u'} C_\beta \xrightarrow{v'} C_\gamma$
	such that any subsequent three objects form an $\mathbb E$-triangle as in $(1)$-$(2)$.

\end{proposition}

\vskip5pt

\begin{proposition}\label{prop:4x4-S-S-?}  \ 	Let $(\mathcal{C}, \mathbb E, \mathfrak s)$ be an extriangulated category. Suppose that
		\begin{equation*}
\xymatrix@R=0.4cm{
  X\ar[r]^{f}\ar[d]_-{\alpha} & Y\ar[r]^{g}\ar[d]^-{\beta} & Z\ar@{-->}[r]^{\delta} & {} \\
  L\ar[r]^{u} & M\ar[r]^{v} & N\ar@{-->}[r]^{{\varepsilon }} & {}
}
\end{equation*}
is a diagram of $\mathbb E$-triangles such that  $u\alpha=\beta f$, and $\alpha$ and $\beta$ are $\mathbb E$-inflations and $\mathbb E$-deflations.	Then there is an $\mathbb E$-deflation $\gamma$ such that $(\alpha, \beta, \gamma)$ is a homotopic morphism of $\mathbb E$-triangles. Moreover, there is a commutative diagram
	\begin{equation*}
\xymatrix@R=0.4cm{
  K_\alpha\ar@{..>}[r]^{{f'}}\ar@{..>}[d]_-{{i_\alpha }} & K_\beta\ar@{..>}[r]^{{g'}}\ar@{..>}[d]^{{i_\beta}} & K_\gamma\ar@{..>}[d]^{{i_\gamma}}\ar@{..>}[r] & {} \\
  X\ar[r]^{f}\ar[d]_-{\alpha} & Y\ar[r]^{g}\ar[d]^-{\beta} & Z\ar@{-->}[r]^{\delta}\ar@{..>}[d]^-{\gamma} & {} \\
  L\ar[r]^{u}\ar@{..>}[d]_{{p_\alpha}} & M\ar[r]^{v}\ar@{..>}[d]^{{p_\beta}} & N\ar@{-->}[r]^{{\varepsilon }} & {} \\
  C_\alpha\ar@{..>}[r]^{{u'}} & C_\beta & {} & {}
}
\end{equation*}
with a connecting morphism $z:K_\gamma\to C_\alpha$, such that the following are $\mathbb E$-triangles$:$

            \vskip5pt
\begin{tasks}[style=enumerate, label=$(\arabic*)$, label-width=4ex](2)
		\task $\con {K_\alpha}{f'}{K_\beta}{g'}{K_\gamma}{\ell^{-1} (z)};$
		\task $\con {K_\beta}{g'}{K_\gamma}z{C_\alpha}{-\ell^{-1} (u')};$
		\task $\con {K_\gamma}z{C_\alpha}{u'}{C_\beta}{\rho^{-1}(g')};$
		\task $\con {K_\alpha}{i_\alpha}X\alpha L{\ell^{-1}(p_\alpha)};$
		\task $\con {K_\beta}{i_\beta}Y\beta M{\ell^{-1}(p_\beta)};$
		\task $\con {K_\gamma}{i_\gamma}Z\gamma N{\ell^{-1}(p_\gamma)};$
		\task $\con X\alpha L{p_\alpha}{C_\alpha}{-\rho^{-1}(i_\alpha)};$
		\task $\con Y\beta M{p_\beta}{C_\beta}{-\rho^{-1}(i_\beta)}$
	\end{tasks}
        \vskip10pt
        with morphisms of $\mathbb E$-triangles $(i_\alpha,i_\beta,i_\gamma)$, $(f',f,u)$, $(f,u, u')$, $(g',g,v)$,
and that $C_\alpha=FK_\alpha$, $C_\beta=FK_\beta$, $u'=Ff'$, where $F$ is the functor defined in \Cref{F}, and $\ell$ and $\rho$ are defined in \Cref{binaturality}.
\vskip5pt
In particular, there is a sequence
$K_\alpha \xrightarrow{f'} K_\beta \xrightarrow{g'} K_\gamma \xrightarrow{z} C_\alpha \xrightarrow{u'} C_\beta$
such that any subsequent three objects form an $\mathbb E$-triangle as in $(1)$-$(3)$.
\end{proposition}

\vskip5pt

\begin{proposition}\label{prop:4x4-S-D-?} \ Let $(\mathcal{C}, \mathbb E, \mathfrak s)$ be an extriangulated category. Suppose that
		\begin{equation*}
\xymatrix@R=0.4cm{
  X\ar[r]^{f}\ar[d]_-{\alpha} & Y\ar[r]^{g}\ar[d]^-{\beta} & Z\ar@{-->}[r]^{\delta} & {} \\
  L\ar[r]^{u} & M\ar[r]^{v} & N\ar@{-->}[r]^{{\varepsilon }} & {}
}
\end{equation*}
is a diagram of $\mathbb E$-triangles such that $u\alpha=\beta f$, and $\alpha$ is an $\mathbb E$-inflation and $\mathbb E$-deflation, and  $\beta$ is an $\mathbb E$-deflation. Then there is an $\mathbb E$-deflation $\gamma$ such that $(\alpha, \beta, \gamma)$ is a homotopic morphism of $\mathbb E$-triangles. Moreover, there is a commutative diagram
	\begin{equation*}
\xymatrix@R=0.4cm{
  K_\alpha\ar@{..>}[r]^{{f'}}\ar@{..>}[d]_{{i_\alpha }} & K_\beta\ar@{..>}[r]^{{g'}}\ar@{..>}[d]^{{i_\beta}} & K_\gamma\ar@{..>}[d]^{{i_\gamma}} \ar@{..>}[r]& {} \\
  X\ar[r]^{f}\ar[d]_-{\alpha} & Y\ar[r]^{g}\ar[d]^-{\beta} & Z\ar@{-->}[r]^{\delta}\ar@{..>}[d]^-{\gamma} & {} \\
  L\ar[r]^{u}\ar@{..>}[d]_{{p_\alpha}} & M\ar[r]^{v} & N\ar@{-->}[r]^{{\varepsilon }} & {} \\
  C_\alpha & {} & {} & {}
}
\end{equation*}
    	 with a connecting morphism $z:K_\gamma\to C_\alpha$, such that the following are $\mathbb E$-triangles$:$
             \vskip5pt
\begin{tasks}[style=enumerate, label=$(\arabic*)$, label-width=4ex](2)
		\task $\con {K_\alpha}{f'}{K_\beta}{g'}{K_\gamma}{\ell^{-1} (z)};$
		\task $\con {K_\beta}{g'}{K_\gamma}z{C_\alpha}{-\ell^{-1} (u')};$
		\task $\con {K_\alpha}{i_\alpha}X\alpha L{\ell^{-1}(p_\alpha)};$
		\task $\con {K_\beta}{i_\beta}Y\beta M{\ell^{-1}(p_\beta)};$
		\task $\con {K_\gamma}{i_\gamma}Z\gamma N{\ell^{-1}(p_\gamma)};$
		\task $\con X\alpha L{p_\alpha}{C_\alpha}{-\rho^{-1}(i_\alpha)}$
	\end{tasks}
        \vskip10pt
    with morphism of $\mathbb E$-triangles  $(i_\alpha,i_\beta,i_\gamma)$, $(f',f,u)$, $(g',g,v)$, and  $C_\alpha=FK_\alpha$, where $F$ is the functor defined in \Cref{F}, and $\ell$ and $\rho$ are defined in \Cref{binaturality}.
    \vskip5pt
    In particular, there is a sequence
$K_\alpha \xrightarrow{f'} K_\beta \xrightarrow{g'} K_\gamma \xrightarrow{z} C_\alpha$
	such that any subsequent three objects form an $\mathbb E$-triangle as in $(1)$-$(2)$.

\end{proposition}

\vskip5pt

\begin{proposition}\label{prop:4x4-I-S-?} \ Let $(\mathcal{C}, \mathbb E, \mathfrak s)$ be an extriangulated category. Suppose that
		\begin{equation*}
\xymatrix@R=0.4cm{
  X\ar[r]^{f}\ar[d]_-{\alpha} & Y\ar[r]^{g}\ar[d]^-{\beta} & Z\ar@{-->}[r]^{\delta} & {} \\
  L\ar[r]^{u} & M\ar[r]^{v} & N\ar@{-->}[r]^{{\varepsilon }} & {}
}
\end{equation*}
is a diagram of $\mathbb E$-triangles such that $u\alpha=\beta f$, and $\alpha$ is an $\mathbb E$-inflation, and that $\beta$ is an $\mathbb E$-inflation and $\mathbb E$-deflation.
Then there is an $\mathbb E$-deflation $\gamma$ such that $(\alpha, \beta, \gamma)$ is a homotopic morphism of $\mathbb E$-triangles. Moreover, there is a commutative diagram
	\begin{equation*}
\xymatrix@R=0.4cm{
  {} & K_\beta\ar@{..>}[r]^{{g'}}\ar@{..>}[d]^{{i_\beta}} & K_\gamma\ar@{..>}[d]^{{i_\gamma}} & {} \\
  X\ar[r]^{f}\ar[d]_-{\alpha} & Y\ar[r]^{g}\ar[d]^-{\beta} & Z\ar@{-->}[r]^{\delta}\ar@{..>}[d]^-{\gamma} & {} \\
  L\ar[r]^{u}\ar@{..>}[d]_-{{p_\alpha}} & M\ar[r]^{v}\ar@{..>}[d]^-{{p_\beta}} & N\ar@{-->}[r]^{{\varepsilon }} & {} \\
  C_\alpha\ar@{..>}[r]^{{u'}} & C_\beta & {} & {}
}
\end{equation*}
with a connecting morphism $z: K_\gamma \to C_\alpha$, such that the following are $\mathbb E$-triangles$:$
            \vskip5pt
\begin{tasks}[style=enumerate, label=$(\arabic*)$, label-width=4ex](2)
		\task $\con {K_\beta}{g'}{K_\gamma}z{C_\alpha}{-\ell^{-1} (u')};$
		\task $\con {K_\gamma}z{C_\alpha}{u'}{C_\beta}{\rho^{-1}(g')};$
		\task $\con {K_\beta}{i_\beta}Y\beta M{\ell^{-1}(p_\beta)};$
		\task $\con {K_\gamma}{i_\gamma}Z\gamma N{\ell^{-1}(p_\gamma)};$
		\task $\con X\alpha L{p_\alpha}{C_\alpha}{-\rho^{-1}(i_\alpha)};$
		\task $\con Y\beta M{p_\beta}{C_\beta}{-\rho^{-1}(i_\beta)}$
	\end{tasks}
        \vskip10pt
    with morphisms of $\mathbb E$-triangles $(f,u, u')$ and $(g',g,v)$, and  $C_\beta=FK_\beta$, where $F$ is the functor defined in \Cref{F}, and $\ell$ and $\rho$ are defined in \Cref{binaturality}.
    \vskip5pt
    In particular, there is a sequence
$	K_\beta \xrightarrow{g'} K_\gamma \xrightarrow{z} C_\alpha \xrightarrow{u'} C_\beta$
	such that any subsequent three objects form an $\mathbb E$-triangle as in $(1)$-$(2)$.
\end{proposition}

\vskip5pt

\begin{proposition}\label{prop:4x4-I-D-?} \	Let $(\mathcal{C}, \mathbb E, \mathfrak s)$ be an extriangulated category. Suppose that
		\begin{equation*}
\xymatrix@R=0.4cm{
  X\ar[r]^{f}\ar[d]_-{\alpha} & Y\ar[r]^{g}\ar[d]^-{\beta} & Z\ar@{-->}[r]^{\delta} & {} \\
  L\ar[r]^{u} & M\ar[r]^{v} & N\ar@{-->}[r]^{{\varepsilon }} & {}
}
\end{equation*}
is a diagram of $\mathbb E$-triangles such that $u\alpha=\beta f$, and $\alpha$ is an $\mathbb E$-deflation, and that $\beta$ is an $\mathbb E$-inflation. Then there is an $\mathbb E$-deflation $\gamma$ such that $(\alpha, \beta, \gamma)$ is a homotopic morphism of $\mathbb E$-triangles. Moreover, there is a commutative diagram
	\begin{equation*}
\xymatrix@R=0.4cm{
  {} & K_\beta\ar@{..>}[r]^{{g'}}\ar@{..>}[d]^{{i_\beta}} & K_\gamma\ar@{..>}[d]^{{i_\gamma}} & {} \\
  X\ar[r]^{f}\ar[d]_-{\alpha} & Y\ar[r]^{g}\ar[d]^-{\beta} & Z\ar@{-->}[r]^{\delta}\ar@{..>}[d]^-{\gamma} & {} \\
  L\ar[r]^{u}\ar@{..>}[d]_{{p_\alpha}} & M\ar[r]^{v} & N\ar@{-->}[r]^{{\varepsilon }} & {} \\
  C_\alpha  & {} & {} & {}
}
\end{equation*}
with a connecting morphism $z: K_\gamma\to C_\alpha$, such that the following are $\mathbb E$-triangles$:$
            \vskip5pt
\begin{tasks}[style=enumerate, label=$(\arabic*)$, label-width=4ex](2)
		\task $\con {K_\beta}{g'}{K_\gamma}z{C_\alpha}{-\ell^{-1} (u')};$
		\task $\con {K_\beta}{i_\beta}Y\beta M{\ell^{-1}(p_\beta)};$
		\task $\con {K_\gamma}{i_\gamma}Z\gamma N{\ell^{-1}(p_\gamma)};$
		\task $\con X\alpha L{p_\alpha}{C_\alpha}{-\rho^{-1}(i_\alpha)}$
	\end{tasks}
        \vskip10pt
    with $(g',g,v)$ a morphism of $\mathbb E$-triangles, and $C_\alpha=FK_\alpha$, where $F$ is the functor defined in \Cref{F}, and $\ell$ and $\rho$ are defined in \Cref{binaturality}.
\end{proposition}

\vskip5pt

We summarize all possible variants of $4\times 4$ Lemma in the following table. Each variant starts with two $\mathbb E$-triangles,
and two morphisms which belong to the set $$W = \{\text{$\mathbb E$-inflation and $\mathbb E$-deflation, $\mathbb E$-inflation,  $\mathbb E$-deflation}\}$$
satisfying the corresponding commutative relation. As in \Cref{thm:hmor}, logically, there are $27$ possible cases,
but only $15$ cases admit affirmative answer (the other $12$ cases do not hold in general, as shown in \Cref{ex:no-good-completion}). \Cref{thm:4x4-sss}, \Cref{thm:3x3-lemma} and \Cref{prop:4x4-I-?-S}-\ref{prop:4x4-I-D-?} claims the $8$ affirmative cases among them
(the rest $7$ cases are dual which are omitted).

\vskip5pt

For a better view, we include the following table, where  $(\alpha, \beta, \gamma)$ is a homotopic morphism of $\mathbb E$-triangles.

\begin{equation*}
\xymatrix@R=0.4cm{
  X\ar[r]^{f}\ar[d]_-{\alpha} & Y\ar[r]^{g}\ar[d]_-{\beta} & Z\ar@{-->}[r]^{\delta}\ar[d]_-{\gamma} & {} \\
  X'\ar[r]^{{f'}} & Y'\ar[r]^{{g'}} & Z'\ar@{-->}[r]^{{\delta'}} & .
}
\end{equation*}
The abbreviations i, d and i-d stand for $\mathbb E$-inflation, $\mathbb E$-deflation, and $\mathbb E$-inflation and $\mathbb E$-deflation, respectively. For instance,
(i,?,d) means that an $\mathbb E$-inflation $\alpha$ and an $\mathbb E$-deflation $\gamma$ are given; ``?'' means that whether there is a $\beta$ in the set $W$ such that $(\alpha, \beta, \gamma)$ is a homotopic morphism of $\mathbb E$-triangles, or equivalently, whether a variant of $4 \times 4$ Lemma exists; the answer ``implicit'' means that $\beta$ is neither an $\mathbb E$-inflation nor an $\mathbb E$-deflation in general, i.e., there is no corresponding variant of $4 \times 4$ Lemma.

\vskip5pt
\centerline {\bf Table of $\boldsymbol{4 \times 4}$ Lemma and its variants}
\vskip5pt
\begin{center}
    \renewcommand{\arraystretch}{1.3}
    \scriptsize
    \begin{tabular}{|c|c| c|c|c|}
      \hline
      \textbf{Prerequisite} & \textbf{Case} & $\boldsymbol{(\alpha, \beta, \gamma)}$ & \textbf{Completion} & \textbf{Reference} \\
      \hline
           \multirow{9}{*}{\makecell{Given $(\alpha,\gamma)$ \\\\ with $\alpha_\ast \delta = \gamma^\ast \delta'$}}
           & $(1)$ & (i-d, ?, i-d) & $\beta$ is i-d & {\Cref{thm:4x4-sss}} \\
            \cline{2-3}\cline{4-5}
           & $(2)$ & (i-d, ?, i) & \multirow{3}{*}{$\beta$ is i} & \Cref{thm:3x3-lemma} \\
            \cline{2-3}\cline{5-5}
           & $(3)$ & (i, ?, i-d) & & \Cref{prop:4x4-I-?-S} \\
            \cline{2-3}\cline{5-5}
           & $(4)$ & (i, ?, i) & & \Cref{thm:3x3-lemma} \\
            \cline{2-3}\cline{4-4}\cline{5-5}
           & $(2)'$ & (d, ?, i-d) &  \multirow{3}{*}{$\beta$ is d} &  {The dual of \Cref{thm:3x3-lemma}} \\
            \cline{2-3}\cline{5-5}
           & $(3)'$ & (i-d, ?, d) & & {The dual of \Cref{prop:4x4-I-?-S}} \\
            \cline{2-3}\cline{5-5}
           & $(4)'$ & (d, ?, d) & & {The dual of \Cref{thm:3x3-lemma}} \\
            \cline{2-3}\cline{4-4}\cline{5-5}
           & $(5)$ & (i, ?, d) & \multirow{2}{*}{Implicit} & \multirow{2}{*}{\Cref{ex:no-good-completion}} \\
            \cline{2-3}
           & $(6)$ & (d, ?, i) &  &  \\
            \hline
          \multirow{9}{*}{\makecell{Given $(\alpha, \beta)$ \\\\ with $f'\alpha=\beta f$}}
           & $(7)$ & (i-d, i-d, ?) & \multirow{4}{*}{$\gamma$ is d} & \Cref{prop:4x4-S-S-?} \\
            \cline{2-3}\cline{5-5}
           & $(8)$ & (i-d, d, ?) &  & \Cref{prop:4x4-S-D-?} \\
            \cline{2-3}\cline{5-5}
           & $(9)$ & (i, i-d, ?) &  &  \Cref{prop:4x4-I-S-?} \\
            \cline{2-3}\cline{5-5}
           & $(10)$ & (i, d, ?) &  & \Cref{prop:4x4-I-D-?} \\
            \cline{2-3}\cline{4-4}\cline{5-5}
           & $(11)$ & (i-d, i, ?) & \multirow{5}{*}{Implicit} & \multirow{5}{*}{\Cref{ex:no-good-completion}} \\
            \cline{2-3}
           & $(12)$ & (i, i, ?) &  & \\
            \cline{2-3}
           & $(13)$ & (d, i, ?) &  &\\
            \cline{2-3}
           & $(14)$ & (d, i-d, ?) &  &  \\
            \cline{2-3}
           & $(15)$ & (d, d, ?) &  &  \\
            \hline
          \multirow{9}{*}{\makecell{Given $(\beta, \gamma)$ \\\\with $g'\beta=\gamma g$}}
           & $(7)'$ & (?, i-d, i-d) & \multirow{4}{*}{$\alpha$ is i} & {The dual of \Cref{prop:4x4-S-S-?}} \\
            \cline{2-3}\cline{5-5}
           & $(8)'$ & (?, i, i-d) &  & {The dual of  \Cref{prop:4x4-S-D-?}} \\
            \cline{2-3}\cline{5-5}
           & $(9)'$ & (?, i-d, d) &  & {The dual of \Cref{prop:4x4-I-S-?}}\\
            \cline{2-3}\cline{5-5}
           & $(10)'$ & (?, i, d) &  & {The dual of \Cref{prop:4x4-I-D-?}}\\
            \cline{2-3}\cline{4-4}\cline{5-5}
           & $(11)'$ & (?, d, i-d) & \multirow{5}{*}{Implicit} &  \multirow{5}{*}{The dual of \Cref{ex:no-good-completion}} \\
            \cline{2-3}
           & $(12)'$ & (?, d, d) &  & \\
            \cline{2-3}
           & $(13)'$ & (?, d, i) &  &\\
            \cline{2-3}
           & $(14)'$ & (?, i-d, i) &  &   \\
            \cline{2-3}
           & $(15)'$ & (?, i, i) &  &  \\
            \hline
    \end{tabular}
  \end{center}
\vskip5pt

\vskip5pt

\section{Homotopic morphisms and (middling) good morphisms}\label{4x4_triangulated}

In this subsection one considers a triangulated category $(\mathcal T, \Sigma, \Delta)$.
Following Neeman \cite[Definition 1.9]{Nee91} a morphism $(\alpha,\beta,\gamma)$ of distinguished triangles
\begin{equation*}
\xymatrix@R=0.4cm{
  X\ar[r]^{f}\ar[d]_-{\alpha} & Y\ar[r]^{g}\ar[d]^-{\beta} & Z\ar[r]^{h}\ar[d]^-{\gamma} & \Sigma X\ar[d]^{{\Sigma\alpha}} \\
  X'\ar[r]^{{f'}} & Y'\ar[r]^{{g'}} & Z'\ar[r]^{{h'}} & \Sigma X'
}
\end{equation*}
is {\it good} if the mapping cone
    \begin{equation*}        Y \oplus X'
		\xrightarrow{\scalebox{0.8}{$\dg{-g&0\\\beta&f'}$}}
		Z \oplus Y'
		\xrightarrow{\scalebox{0.8}{$\dg{-h&0\\\gamma&g'}$}}
		\Sigma X \oplus Z'
		\xrightarrow{\scalebox{0.8}{$\dg{-\Sigma f&0\\\Sigma \alpha&h'}$}}
		\Sigma Y \oplus \Sigma X'
\end{equation*}
is a distinguished triangle;  and  it is {\it middling good} if there is a $4 \times 4$ diagram
\begin{equation}\label{eq:4-origin}
\xymatrix@R=0.4cm{
  X\ar[r]^{f}\ar[d]_-{\alpha} & Y\ar[r]^{g}\ar[d]^-{\beta} & Z\ar[r]^{h}\ar[d]^-{\gamma} & \Sigma X\ar[d]^{{\Sigma \alpha}} \\
  X'\ar[r]^{{f'}}\ar[d]_-{{\alpha'}} & Y'\ar[r]^{{g'}}\ar[d]^{{\beta'}} & Z'\ar[r]^{{h'}}\ar[d]^{{\gamma'}} & \Sigma X'\ar[d]^{{\Sigma \alpha '}} \\
  X''\ar[r]^{{f''}}\ar[d]_-{{\alpha''}} & Y''\ar[r]^{{g''}}\ar[d]^{{\beta''}} & Z''\ar[r]^{{h''}}\ar[d]^{{\gamma''}} & \Sigma X''\ar[d]^{{-\Sigma \alpha''}} \\
  \Sigma X\ar[r]^{{\Sigma f}} & \Sigma Y\ar[r]^{{\Sigma g}} & \Sigma Z\ar[r]^{{-\Sigma h}} & \Sigma ^2X
}
\end{equation}
where all the squares commute, except the right bottom square satisfies $(\Sigma \alpha'') h''=-(\Sigma h) \gamma''$,
and all rows and columns are distinguished triangles. See also \cite[Definitions 2.5 and 2.3]{CF22}. By \cite[Theorem 2.3]{Nee91} a good morphism is middling good.

\vskip5pt

When viewing $\mathcal T$ as an extriangulated category, every morphism is both an $\mathbb E$-inflation and an $\mathbb E$-deflation,  $\mathcal L=\mathcal R=\mathcal T$,
and the functor $F$ can be $\Sigma$. Note that in this case for $f\in \mathcal T(X, Y)$, $\ell(f)=f$ and $\rho(f)=\Sigma^{-1}f$.

\vskip5pt

\begin{proposition}\label{middlinggood} \ A homotopic morphism of  distinguished triangles  is middling good.
\end{proposition}

	\begin{proof} \ Let $(\alpha, \beta, \gamma)$ be a homotopic morphism of distinguished triangles.
\begin{equation*}
\xymatrix@R=0.4cm{
  X\ar[r]^{f}\ar[d]_-{\alpha} & Y\ar[r]^{g}\ar[d]^-{\beta} & Z\ar[r]^{h}\ar[d]^-{\gamma} & \Sigma X\ar[d]^{{\Sigma\alpha}} \\
  X'\ar[r]^{{f'}} & Y'\ar[r]^{{g'}} & Z'\ar[r]^{{h'}} & \Sigma X'
}
\end{equation*}
By \Cref{thm:4x4-sss} there is a diagram
        \begin{equation}
\label{tri44}
\xymatrix@R=0.4cm{
  X''\ar[r]^-{{f''}}\ar[d]_-{{\alpha''}} & Y''\ar[r]^-{{g''}}\ar[d]^-{{\beta''}} & Z''\ar[r]^-{{h''}}\ar[d]^-{{\gamma''}} & \Sigma X''\ar[d]^-{{\Sigma \alpha''}} \\
  X\ar[r]^-{f}\ar[d]_-{\alpha} & Y\ar[r]^-{g}\ar[d]^-{\beta} & Z\ar[r]^-{h}\ar[d]^-{\gamma} & \Sigma X\ar[d]^-{{\Sigma \alpha}} \\
  X'\ar[r]^-{{f'}}\ar[d]_-{{\alpha'}} & Y'\ar[r]^-{{g'}}\ar[d]^-{{\beta'}} & Z'\ar[r]^-{{h'}}\ar[d]^-{{\gamma'}} & \Sigma X'\ar[d]^-{{\Sigma \alpha '}} \\
  \Sigma X''\ar[r]^-{{\Sigma f''}} & \Sigma Y''\ar[r]^-{{\Sigma g''}} & \Sigma Z''\ar[r]^-{{-\Sigma h''}} & \Sigma^2 X''
}
\end{equation}
    where all nine squares are commutative, the left three columns and all four rows are distinguished triangles. By rotating the first row of the diagram \Cref{tri44} to the bottom, there is a $4\times 4$ diagram \Cref{eq4}$:$

    \begin{equation}\label{eq4}
\xymatrix@R=0.6cm{
  X\ar[r]^{f}\ar[d]_-{\alpha} & Y\ar[r]^{g}\ar[d]^-{\beta} & Z\ar[r]^{h}\ar[d]^-{\gamma} & \Sigma X\ar[d]^{{\Sigma \alpha}} \\
  X'\ar[r]^{{f'}}\ar[d]_-{{\alpha'}} & Y'\ar[r]^{{g'}}\ar[d]^{{\beta'}} & Z'\ar[r]^{{h'}}\ar[d]^{{\gamma'}} & \Sigma X'\ar[d]^{{\Sigma \alpha '}} \\
 \Sigma X''\ar[r]^{{\Sigma f''}}\ar[d]_-{{-\Sigma\alpha''}} & \Sigma Y''\ar[r]^{{\Sigma g''}}\ar[d]^{{-\Sigma\beta''}} & \Sigma Z''\ar[r]^{{-\Sigma h''}}\ar[d]^{{-\Sigma\gamma''}} & \Sigma^2 X''\ar[d]^{{\Sigma^2 \alpha''}} \\
  \Sigma X\ar[r]^-{{\Sigma f}} & \Sigma Y\ar[r]^-{{\Sigma g}} & \Sigma Z\ar[r]^-{{-\Sigma h}} & \Sigma ^2X
}
\end{equation}
    where all the squares commute, except the right bottom square satisfies $(\Sigma^2 \alpha'') (-\Sigma h'')=-(-\Sigma h) (-\Sigma \gamma'')$,
and all rows and columns are distinguished triangles.
By definition $(\alpha, \beta,\gamma)$ is middling good.
	\end{proof}

\Cref{middlinggood} together with \Cref{thm:hs-morphism} recover the classical {\rm $4 \times 4$ Lemma} in triangulated categories.

\begin{corollary} \label{44intri} \ For distinguished triangles $X\xrightarrow{f}Y\xrightarrow{g}Z\xrightarrow{h}\Sigma X$,  $X'\xrightarrow{f'}Y'\xrightarrow{g'}Z'\xrightarrow{h'}\Sigma X'$ with $\beta f=f'\alpha$, there is $\gamma: Z\to Z'$ such that $(\alpha, \beta,\gamma)$ is a middling good morphism.
    \begin{equation*}
\xymatrix@R=0.4cm{
  X\ar[r]^{f}\ar[d]_-{\alpha} & Y\ar[r]^{g}\ar[d]^-{\beta} & Z\ar[r]^{h}\ar@{..>}[d]^-{\gamma} & \Sigma X\ar[d]^{{\Sigma \alpha}} \\
  X'\ar[r]^{{f'}} & Y'\ar[r]^{{g'}} & Z'\ar[r]^{{h'}} & \Sigma X'
}
\end{equation*}
\end{corollary}

\begin{proof} \ By the axiom of triangulated category there is $\gamma': Z\to Z'$ such that $(\alpha,\beta,\gamma')$ is a morphism of distinguished triangles. There is $\gamma: Z\to Z'$ such that $(\alpha,\beta,\gamma)$ is a homotopic morphism by \Cref{thm:hs-morphism}. Then $(\alpha, \beta,\gamma)$ is middling good by \Cref{middlinggood}.
\end{proof}

The following proposition gives a sufficient condition for a homotopic morphism to be good.

\vskip5pt

\begin{proposition}\label{good} \  Let $(\alpha, \beta, \gamma)$ be a homotopic morphism of distinguished triangles
\begin{equation*}
\xymatrix@R=0.4cm{
  X\ar[r]^{f}\ar[d]_-{\alpha} & Y\ar[r]^{g}\ar[d]^-{\beta} & Z\ar[r]^{h}\ar[d]^-{\gamma} & \Sigma X\ar[d]^{{\Sigma\alpha}} \\
  X'\ar[r]^{{f'}} & Y'\ar[r]^{{g'}} & Z'\ar[r]^{{h'}} & \Sigma X'
}
\end{equation*}
with $\mathcal T(Z', \Sigma X)=0$. Then it is good.
\end{proposition}

      \begin{proof} \ Since $(\alpha, \beta, \gamma)$ is a homotopic morphism, there is a decomposition with two distinguished triangles.
\begin{equation*}
  \xymatrix@R=0.4cm{
  X\ar[r]^{f}\ar[d]_-{\alpha} & Y\ar[r]^{g}\ar[d]^{{\beta_1}} & Z\ar[r]^{h} & \Sigma X\ar[d]^{{\Sigma \alpha }} \\
  X'\ar[r]^{s}\ar@{=}[d] & E\ar[r]^{t}\ar[d]_{{\beta_2}} & Z\ar@{=}[u]\ar[r]^{{h'\gamma}}\ar[d]^-{\gamma} & \Sigma X'\ar@{=}[d] \\
  X'\ar[r]^{{f'}} & Y'\ar[r]^{{g'}} & Z'\ar[r]^{{h'}} & \Sigma X'
}
\end{equation*}
    where $\tri X{\dg{f \\ \alpha}}{Y \oplus X'}{\dg{\beta_1, -s}}E{ht}$ and $\tri E{\dg{-t \\ \beta_2}}{Z \oplus Y'} {(\gamma, g')} {Z'} {(\Sigma s)h'}$ are distinguished triangles, $\beta = \beta_2\beta_1$ and $h'\gamma=(\Sigma \alpha)h$. It follows that $ht (-s\Sigma^{-1}h')= -h(t s)\Sigma^{-1}h'= 0$ and $\dg {-\Sigma t\\ \Sigma \beta_2} \dg{-\Sigma \beta_1 , \Sigma s}=\dg{-\Sigma g & 0\\ -\Sigma \beta & \Sigma f'}$. Then the octahedral axiom gives rise to a distinguished triangle
\begin{equation}\label{triangle}
    Z\oplus Y'\xrightarrow{\scalebox{0.8}{$\dg {a & b\\ \gamma & g'}$}} \Sigma X\oplus Z'\xrightarrow{\scalebox{0.8}{$\dg {-\Sigma f & c\\ -\Sigma \alpha & d}$}} \Sigma Y\oplus \Sigma X' \xrightarrow{\scalebox{0.8}{$\dg{-\Sigma g & 0\\ -\Sigma \beta & \Sigma f'}$}} \Sigma Z\oplus \Sigma Y'
\end{equation}
for some morphisms $a$, $b$, $c$ and $d$. See the following diagram.

\begin{equation*}
\xymatrix@R=0.8cm@C=1.4cm{
  \Sigma^{-1}Z'\ar[r]^(0.55){{-s(\Sigma^{-1}h')}}\ar@{=}[d] &  E\ar[r]^(0.4){\scalebox{0.8}{${\dg {-t\\ \beta_2}}$}}\ar[d]_{ht} & Z \oplus Y'\ar[r]^(0.6){{\dg{\gamma , g'}}}\ar@{..>}[d]_{\scalebox{0.8}{${\dg {a & b\\ \gamma & g'}}$}} & Z'\ar@{=}[d] \\
  \Sigma^{-1}Z'\ar[r]^{0} & \Sigma X\ar[r]^(0.4){\scalebox{0.8}{${\dg{1\\ 0}}$}}\ar[d]_{\scalebox{0.8}{${\dg {-\Sigma f\\ -\Sigma \alpha}}$}} & \Sigma X \oplus Z'\ar[r]^(0.6){{\dg{0 , 1}}}\ar@{..>}[d]_{\scalebox{0.8}{${\dg {-\Sigma f & c\\ -\Sigma \alpha & d}}$}} & Z'\ar[d]^{{-(\Sigma s)h'}} \\
  {} & \Sigma Y \oplus \Sigma X'\ar@{=}[r]\ar[d]_{\scalebox{0.8}{${\dg{-\Sigma \beta_1 , \Sigma s}}$}} & \Sigma Y \oplus \Sigma X'\ar[r]^(0.6){\scalebox{0.8}{${\dg{-\Sigma \beta_1 , \Sigma s}}$}}\ar@{..>}[d]_{\scalebox{0.8}{${\dg{\Sigma g & 0\\ -\Sigma \beta & \Sigma f'}}$}} & \Sigma E \\
  {} & \Sigma E\ar[r]_(0.4){\scalebox{0.8}{${\dg {-\Sigma t\\ \Sigma \beta_2}}$}} & \Sigma Z\oplus \Sigma Y' & {}
}
\end{equation*}

By \Cref{triangle} one has $\dg {a & b\\ \gamma & g'} \dg{ -g & 0\\  \beta & - f'}=0$, and then $bf'=0$. Thus $b: Y'\to \Sigma X$ factors through $g': Y'\to Z'$, i.e.,  $b=ug'$ for some $u: Z'\to \Sigma X$. By the assumption $\mathcal T(Z', \Sigma X)=0$, thus $u=0$ and $b= 0$. Since $\dg{1\\ 0}ht=\dg {a & b\\ \gamma & g'} \dg{-t\\ \beta_2}$, $ht=-at+b\beta_2=-at$, and hence $(a+h)t=0$. Then $a+h: Z\to \Sigma X$ factors through $h'\gamma: Z\to \Sigma X'$, i.e.,  $a+h=vh'\gamma$ for some $v: \Sigma X'\to \Sigma X$. Since $vh'\in \mathcal T(Z', \Sigma X)=0$, $a+h=0$. Thus $\dg {a & b\\ \gamma & g'}=\dg {-h & 0\\ \gamma & g'}$.

\vskip5pt

By \Cref{triangle} one has $\dg{\Sigma g & 0\\ -\Sigma \beta & \Sigma f'} \dg {-\Sigma f & c\\ -\Sigma \alpha & d}=0$, and then $(\Sigma g)c=0$. Thus $c: Z'\to \Sigma Y$ factors through $\Sigma f:\Sigma X\to \Sigma Y$,
i.e.,  $c=fu'$ for some $u': Z'\to \Sigma X$, which is the zero map. So $c=0$. Since $-(\Sigma s )h'\dg{1, 0}=\dg{-\Sigma \beta_1 , \Sigma s } \dg {-\Sigma f & c\\ -\Sigma \alpha & d}$,  $-(\Sigma s )h'=-(\Sigma \beta_1) c+(\Sigma s)d=(\Sigma s)d$, and hence $(\Sigma s)(d+h')=0$. Then $d+h': Z'\to \Sigma X'$ factors through $h'\gamma=(\Sigma \alpha) h: Z\to \Sigma X'$, i.e., $d+h'=v'(\Sigma \alpha) h $ for some $v': Z'\to Z$. Since $h v'\in \mathcal T(Z', \Sigma X)=0$, $d+h'=0$. Thus $\dg {-\Sigma f & c\\ -\Sigma \alpha & d}=\dg {-\Sigma f & 0\\ -\Sigma \alpha & -h'}$.

\vskip5pt

Therefore the distinguished triangle \Cref{triangle} reads as
\begin{equation*}    Z\oplus Y' \xrightarrow{\scalebox{0.8}{$\dg{ -g & 0\\  \beta & - f'}$}}   Z\oplus Y'\xrightarrow{\scalebox{0.8}{$\dg {-h & 0\\ \gamma & g'}$}} \Sigma X\oplus Z'\xrightarrow{\scalebox{0.8}{$\dg {-\Sigma f & 0\\ -\Sigma \alpha & -h'}$}} \Sigma Y\oplus \Sigma X'
\end{equation*}
By the following isomorphisms of distinguished triangles
\begin{equation*}
\xymatrix@R=0.8cm@C=1.4cm{
  Z\oplus  Y'\ar[r]^{\scalebox{0.8}{$ \dg{ -g & 0\\ \beta & -f'}$}}\ar[d]^{\cong}_{\scalebox{0.8}{${\dg{1 & 0\\ 0 & -1}}$}} & Z \oplus Y'\ar@{..>}[r]^{\scalebox{0.8}{$ \dg {-h & 0\\ \gamma & g'} $}}\ar@{=}[d] & \Sigma X \oplus Z'\ar@{..>}[r]^{\scalebox{0.8}{$ \dg {-\Sigma f & 0\\ -\Sigma \alpha & -h'}$}}\ar@{=}[d] & \Sigma Y \oplus \Sigma X'\ar[d]^{\scalebox{0.8}{${\dg{1 & 0\\ 0 & -1}}$}}_{\cong} \\
  Z\oplus  Y'\ar[r]^{ \scalebox{0.8}{$\dg{ -g & 0\\ \beta & f'}$}} & Z \oplus Y'\ar@{..>}[r]^{ \scalebox{0.8}{$\dg {-h & 0\\ \gamma & g'}$}} & \Sigma X \oplus Z'\ar@{..>}[r]^{\scalebox{0.8}{$\dg {-\Sigma f & 0\\ \Sigma \alpha & h'}$}} & \Sigma Y \oplus \Sigma X'
}
\end{equation*}
one sees that $(\alpha, \beta,\gamma)$ is a good morphism.
    \end{proof}

\section{Weakly idempotent completness}

Several equivalent conditions of weakly idempotent complete extriangulated categories will be shown, inspired by Heller's axiom \cite[Appendix B]{Buh10} and \cite[Proposition 1.31]{Gil25};
on the other hand, some variants of these conditions will be discussed when an extriangulated category is not assumed to be weakly idempotent complete.

\subsection{Equivalent characterizations} \ A morphism  $f': A'\to B'$ is a \textit{retract} of morphism $f: A\to B$ if there is a commutative diagram
    \begin{equation*}
\xymatrix@R=0.4cm{
  A'\ar[r]^{i}\ar[d]_{{f'}} & A\ar[r]^{p}\ar[d]^{f} & A'\ar[d]^{{f'}} \\
  B'\ar[r]^{j} & B\ar[r]^{q} & B'
}
\end{equation*}
such that  $pi=1_{A'}$ and $qj=1_{B'}$.

\vskip5pt

\begin{proposition}\label{thm:wic} \ Let $(\mathcal{C}, \mathbb E, \mathfrak{s})$ be an extriangulated category. The following statements are equivalent:
\vskip5pt
$(1)$ \  $\mathcal C$ is weakly idempotent complete.
\vskip5pt
$(2)$ \ For  $f: X\to Y$ and $g: Y\to Z$, if $gf$ is an $\mathbb E$-inflation, then so is $f$.
\vskip5pt
$(2)'$ \  For $f: X\to Y$ and $g: Y\to Z$, if $gf$ is an $\mathbb E$-deflation, then so is $g$.
\vskip5pt
$(3)$ \  The class of $\mathbb E$-inflations is closed under retracts.
\vskip5pt
$(3)'$ \  The class of  $\mathbb E$-deflations is closed under retracts.	
\vskip5pt
$(4)$ \  For $i: X\to Y$,  $Z\in \mathcal C$, if $\dg{i \\ 0}: X\to Y\oplus Z$ is an $\mathbb E$-inflation, then so is $i$.
\vskip5pt
$(4)'$   For $p: X\to Y$, $Z\in \mathcal C$, if $(p,0): X\oplus Z\to Y$ is an $\mathbb E$-deflation, then so is $p$.
\vskip5pt
$(5)$ \  {\rm ($3 \times 3$ Lemma)} \ Suppose that there are four $\mathbb E$-triangles $\con{A_1}{f_A}{A_2}{g_A}{A_3}{\delta_A}$, $\con{B_1}{f_B}{B_2}{g_B}{B_3}{\delta_B}$, $\con{A_2}{i_2}{B_2}{p_2}{C_2}{\varepsilon_2}$ and $\con{A_3}{i_3}{B_3}{p_3}{C_3}{\varepsilon_3}$ such that $g_B   i_2 = i_3 g_A$. Then there is a $3 \times 3$ diagram
        \begin{equation*}
\xymatrix@R=0.4cm{
  A_1\ar[r]^{{f_A}}\ar@{..>}[d]_{{i_1}} & A_2\ar[r]^{{g_A}}\ar[d]^{{i_2}} & A_3\ar@{-->}[r]^{{\delta_A}}\ar[d]^{{i_3}} & {} \\
  B_1\ar[r]^{{f_B}}\ar@{..>}[d]_{{p_1}} & B_2\ar[r]^{{g_B}}\ar[d]^{{p_2}} & B_3\ar@{-->}[r]^{{\delta_B}}\ar[d]^{{p_3}} & {} \\
  C_1\ar@{..>}[r]^{{f_C}}\ar@{..>}[d]_-{{\varepsilon _1}} & C_2\ar@{..>}[r]^{{g_C}}\ar@{-->}[d]^{{\varepsilon _2}} & C_3\ar@{..>}[r]^{{\delta_C}}\ar@{-->}[d]^{{\varepsilon _3}} & {} \\
  {} & {} & {} & {}
}
\end{equation*}
\end{proposition}

	\begin{proof} \ Either $(2)$ or $(2)'$ implies $(1)$ (\cite[Proposition 3.33]{Msa22}), and $(1)$ implies $(2)$ and $(2)'$ (\cite[Proposition 2.7]{Kla22}). Hence $(1)$, $(2)$ and $(2)'$ are equivalent.

    \vskip5pt

    From $(1)$ and $(2)$ one gets $(3)$. Let $f'$ be a retract of some $\mathbb E$-inflation $f:$
		\begin{equation*}
\xymatrix@R=0.4cm{
  X'\ar[r]^{i}\ar[d]_{{f'}} & X\ar[r]^{p}\ar[d]^{f} & X'\ar[d]^{{f'}} \\
  Y'\ar[r]^{j} & Y\ar[r]^{q} & Y'
}
\end{equation*}
		By $(1)$,  $i$ has a cokernel, thus $i$ is an $\mathbb E$-inflation in a split $\mathbb E$-triangle. Note that $jf'=fi$ is an $\mathbb E$-inflation, by $(2)$ $f'$ is an $\mathbb E$-inflation. Thus one obtains $(3)$.

    \vskip5pt

    From $(3)$ one gets $(4)$. Consider $i: X\to Y$ and $Z\in \mathcal C$ given in $(4)$. The following commutative diagram shows that $i$ is a retract of $\dg{i \\ 0}:$
    \begin{equation*}
\xymatrix@R=0.5cm{
  X\ar@{=}[r]\ar[d]_{i} & X\ar@{=}[r]\ar[d]^{{\dg{i \\ 0}}} & X\ar[d]^{i} \\
  Y\ar[r]^(0.4){{\dg{1 \\ 0}}} & Y\oplus Z\ar[r]^(0.6){{\dg{1 , 0}}} & Y.
}
\end{equation*}
By assumption $i$ is an $\mathbb E$-inflation. Thus one obtains $(4)$.

    \vskip5pt

    From $(4)$ one gets $(2)$. Consider $f: X\to Y$ and $g: Y\to Z$ with $gf$ an $\mathbb E$-inflation. By \Cref{prop:comp-infldefl}$(1)$, $\dg{f \\ gf} : X\to Y\oplus Z$ is an $\mathbb E$-inflation. Hence $\dg{f \\ 0} = \dg{1 & 0 \\ -g & 1} \dg{f \\ gf}$ is an $\mathbb E$-inflation. By assumption $f$ is an $\mathbb E$-inflation. Thus one obtains $(2)$.

    \vskip5pt

	Note that from $(2)$ and $(2)'$ one gets $(5)$ (\cite[Lemma 5.9]{NP19}).
	
    \vskip5pt
	From $(5)$ one gets $(1)$. Consider the retraction $p: Y \to X$ with a right inverse $i$. Since $\dg{0 & p \\ i & 1-ip}\dg{0 & p \\ i & 1-ip} = \dg{1 & 0\\ 0 & 1}$, $\dg{0 & p \\ i & 1-ip}$ is an isomorphism. By \Cref{isomorphisms} $\con Y {\dg{p \\ 1-ip}}{X\oplus Y}{\dg{0 , p}}X 0$ is a split $\mathbb E$-triangle:
\begin{equation*}
\xymatrix@C=1.1cm{
  Y\ar[r]^(0.4){{\dg{0 \\ 1}}} & X \oplus Y\ar[r]^(0.6){{\dg{1 , 0}}}\ar[d]^-{\scalebox{0.8}{${\dg{0 & p \\ i & 1-ip}}$}}_-{\cong} & X\ar@{-->}[r]^{0} & {} \\
  Y\ar@{=}[u]\ar[r]^(0.4){\scalebox{0.8}{${\dg{p \\ 1-ip}}$}} & X \oplus Y\ar[r]^(0.6){{\dg{0 , p}}} & X\ar@{=}[u]\ar@{-->}[r]^{0}. & {}
}
\end{equation*}
By assumption there exists a $3 \times 3$ diagram
		\begin{equation*}
\xymatrix@R=0.4cm@C=1.1cm{
  X\ar@{=}[r]\ar@{..>}[d] & X\ar[r]\ar[d]^{\scalebox{0.8}{${\dg{1 \\ 0}}$}} & 0\ar[d]\ar@{-->}[r]^{0} & \\
  Y\ar[r]^(0.4){\scalebox{0.8}{${\dg{p \\ 1-ip}}$}}\ar@{..>}[d] & X \oplus Y\ar[r]^(0.6){{\dg{0,p}}}\ar[d]^{{\dg{0,1}}} & X\ar@{=}[d] \ar@{-->}[r]^{0} &\\
  Z\ar@{..>}[r] \ar@{..>}[d] & Y\ar@{..>}[r]^{p}\ar@{-->}[d]^{0} & X \ar@{..>}[r] \ar@{-->}[d]^{0} & \\
  & & & \\
}
\end{equation*}
		$p$ is an $\mathbb E$-deflation, which has a kernel by \Cref{lem:infl-section}$(3)$. Thus one obtains $(1)$.
\vskip5pt
Similarly $(3)'$ and $(4)'$ are equivalent to the others. This completes the proof.
		\end{proof}

\begin{remark} \ As shown in Case $(15)'$ of the table in Section 5.4,  \Cref{thm:wic}$(5)$ may fail if the category is not weakly idempotent complete.
\end{remark}

\subsection{Extriangulated categories without weakly idempotent completeness}

\vskip5pt

\begin{proposition}\label{prop:comp-infldefl-2} \ Let $(\mathcal{C}, \mathbb E, \mathfrak s)$ be an extriangulated category, and $f: X \to Y$ and $g: Y \to Z$ morphisms.

\vskip5pt

$(1)$ \ If $gf$ is an $\mathbb E$-inflation, then $f$ is a retract of an $\mathbb E$-inflation.

\vskip5pt

$(1)'$ \ If $gf$ is an $\mathbb E$-deflation, then $g$ is a retract of an $\mathbb E$-deflation.
\end{proposition}

    \begin{proof} \ One justifies $(1)$, and the assertion $(1)'$ can be proved dually. Since $gf$ is an $\mathbb E$-inflation, $\dg {f \\ gf}: X\to Y\oplus Z$ is also an $\mathbb E$-inflation by \Cref{prop:comp-infldefl}$(1)$.
As an isomorphism,  $\dg {1_Y & 0\\-g & 1_Z}: Y\oplus Z\to Y\oplus Z$ is an $\mathbb E$-inflation. Thus $\dg {1_Y & 0\\-g & 1_Z}\dg {f \\ gf}=\dg {f \\ 0}$ is an $\mathbb E$-inflation. By the commutative diagram

\begin{equation*}
\xymatrix@R=0.8cm{
  X\ar@{=}[r]\ar[d]_-{f} & X\ar@{=}[r]\ar[d]^{\dg {f\\ 0}} & X\ar[d]^{f} \\
  Y\ar[r]^(0.4){\dg{1\\ 0}} & Y\oplus Z\ar[r]^(0.6){{\dg{1,0}}} & Y
}
\end{equation*}
one sees that $f$ is a retract of the $\mathbb E$-inflation $\dg {f \\ 0}$. \end{proof}

\begin{lemma}\label{lem:hs-retraction}
	Let $(\mathcal{C}, \mathbb E, \mathfrak s)$ be an extriangulated category. Suppose that there is a homotopic square
	\begin{equation*}
\xymatrix@R=0.4cm{
  A_1\ar[r]^{u}\ar[d]_{f} & B_1\ar[d]^{g} \\
  A_2\ar[r]^{v} & B_2
}
\end{equation*}
If $v$ is a retraction then so is $u;$ and if $u$ is a section then so is $v$.
\end{lemma}

	\begin{proof} \ Suppose that there is $i: B_2\to A_2$ such that $vi=1_{B_2}$. Then $vig=g=g 1_{B_1}$. By \Cref{lem:weak}$(1)$ there exists $s : B_1 \to A_1$ such that $us=1_{B_1}$, i.e., $u$ is a retraction.
		\begin{equation*}
\xymatrix@R=0.4cm{
  B_1\ar@{..>}[dr]^{s}\ar@{=}@/^1pc/[drr]\ar@/_1pc/[ddr]_{ig} & {} & {} \\
  {} & A_1\ar[r]^{u}\ar[d]_{f} & B_1\ar[d]^{g} \\
  {} & A_2\ar[r]^{v} & B_2
}
\end{equation*}
	The other assertion is dually proved.
	\end{proof}

    \begin{lemma}\label{lem:hs-retract-inf} \ Let $(\mathcal{C}, \mathbb E, \mathfrak s)$ be an extriangulated category, $f': A'\to B'$ a retract of an $\mathbb E$-inflation $f:A\to B$.
    Then there exists an $\mathbb E$-inflation $\overline f$, a section $\overline j: B' \to \overline B$ and a retraction $\overline q: \overline B \to B'$
    such that
    \begin{equation*}
\xymatrix@R=0.4cm{
  A'\ar@{=}[r]\ar[d]_{{f'}} & A'\ar@{=}[r]\ar[d]^{{\overline f}} & A'\ar[d]^{{f'}} \\
  B'\ar[r]^{{\overline j}} & \overline B\ar[r]^{{\overline q}} & B'
}
\end{equation*}
commutes and $\overline{q}\overline{j}=1_{B'}$.\end{lemma}
    \begin{proof} \ By definition there is a commutative diagram with $pi=1_{A'}$ and $jq=1_{B'}:$
    \begin{equation*}
\label{eq:ret-inf-1}
\xymatrix@R=0.4cm{
  A'\ar[r]^{i}\ar[d]_{{f'}} & A\ar[r]^{p}\ar[d]^{f} & A'\ar[d]^{{f'}}\\
  B'\ar[r]^{j} & B\ar[r]^{q} & B'
}
\end{equation*} Consider an $\mathbb E$-triangle $\con AfBgC\delta$. Let $\con {A'}{\overline f}{\overline B}{\overline g}{\overline C}{p_\ast\delta}$ be an $\mathbb E$-triangle realizing $p_\ast \delta$.
By \Cref{lem:hs-f?1}$(1)$ there exist $t: B \to \overline B$ such that the left square in the following commutative diagram is a homotopic square.
        \begin{equation*}
\label{eq:ret-inf-2}
\xymatrix@R=0.4cm{
  {} & {} & {} & {} \\
  A\ar[r]^{f}\ar[d]_{p} & B\ar[r]^{g}\ar@{..>}[d]_{t}\ar@/^0.6pc/[ddr]^{q} & C\ar@{-->}[r]^{\delta}\ar@{=}[d] & {} \\
  A'\ar[r]^{{\overline f}}\ar@/_1pc/[drr]^{{f'}} & \overline B\ar[r]^{{\overline g}}\ar@{..>}[dr]^{{\overline q}} & C\ar@{-->}[r]^{{p_\ast \delta}} & {} \\
  {} & {} & B' & {}
}
\end{equation*}
 By \Cref{lem:weak}$(2)$ there exists a $\overline q : \overline B \to B'$ such that $f' = \overline q \overline f$ and $q = \overline q t$.
 Take $\overline j:= tj \in \mathcal{C}(B', \overline B)$. Then $\overline q \overline j = \overline q tj = qj = 1_{B'}$. Also
$$\overline f = \overline fpi =  t fi = t jf' = \overline j f'.$$
        This completes the proof.
    \end{proof}

\vskip5pt

\begin{proposition}\label{cor:infl-ret}
    Let $(\mathcal{C}, \mathbb E, \mathfrak s)$ be an extriangulated category. Then $f' : A' \to B'$ is a retract of some $\mathbb E$-inflation if and only if there exists $\overline B\in \mathcal C$ such that $\dg{f' \\ 0} : A' \to B' \oplus \overline B$ is an $\mathbb E$-inflation.
	\end{proposition}

    \begin{proof} \ The ``if'' part is straightforward. Conversely, suppose that $f' : A' \to B'$ is a retract of an $\mathbb E$-inflation $f : A \to B$.
    By \Cref{lem:hs-retract-inf} there exists an $\mathbb E$-inflation $\overline f$, a section $\overline j: B' \to \overline B$ and a retraction $\overline q: \overline B \to B'$
    such that
    \begin{equation*}
\xymatrix@R=0.4cm{
  A'\ar@{=}[r]\ar[d]_{{f'}} & A'\ar@{=}[r]\ar[d]^{{\overline f}} & A'\ar[d]^{{f'}} \\
  B'\ar[r]^{{\overline j}} & \overline B\ar[r]^{{\overline q}} & B'
}
\end{equation*}
commutes and $\overline{q}\overline{j}=1_{B'}$. By \Cref{prop:comp-infldefl}$(1)$, $\dg{f' \\ \overline f}: A' \to B' \oplus \overline B$ is an $\mathbb E$-inflation. Thus $\dg{f' \\ 0} = \dg{1 & 0\\ -\overline j & 1} \dg{f' \\ \overline f}$ is also an $\mathbb E$-inflation.
    \end{proof}

\begin{remark} \ The result above in fact gives an understanding of the weakly idempotent completion of an extriangulated category,  as defined in \cite[Definition 3.26]{Msa22}.
\end{remark}

\vskip5pt

\begin{corollary}\label{homotopyretract}
	Let $(\mathcal{C}, \mathbb E, \mathfrak s)$ be an extriangulated category. Suppose that there is a homotopic square
	\begin{equation*}
\xymatrix@R=0.4cm{
  A_1\ar[r]^{u}\ar[d]_{f} & B_1\ar[d]^{g} \\
  A_2\ar[r]^{v} & B_2
}
\end{equation*}
\vskip5pt
$(1)$ \ $f$ is a retract of an $\mathbb E$-inflation if and only if $g$ is a retract of an $\mathbb E$-inflation.

    \vskip5pt
$(1)'$ \  $f$ is a retract of an $\mathbb E$-deflation if and only if $g$ is a retract of an $\mathbb E$-deflation.

\end{corollary}
\begin{proof} \ One only proves $(1)$. Suppose that $f$ is a retract of some $\mathbb E$-inflation. By \Cref{cor:infl-ret} there is $C\in \mathcal C$ such that $\dg{f\\ 0}: A_1\to A_2\oplus C$ is an $\mathbb E$-inflation. Note that there is an $\mathbb E$-triangle $\con {A_1}{\dg{u\\f}}{B_1\oplus A_2}{\dg{g,-v}}{B_2}\delta$. By ET2, $\con {A_1}{\dg{u\\f\\0}}{B_1\oplus A_2\oplus C}{\dg{g&-v&0\\0&0&-1}}{B_2\oplus C}\delta$ is also an $\mathbb E$-triangle. By definition
\begin{equation*}
\xymatrix@R=0.7cm{
  A_1\ar[r]^{u}\ar[d]_{ \dg{f\\ 0}} & B_1\ar[d]^{ \dg{g\\ 0}} \\
  A_2\oplus C\ar[r]^{\dg{v& 0\\0& 1}} & B_2\oplus C
}
\end{equation*}
is a homotopic square. By \Cref{prop:hs-inflation}, $\dg{g\\ 0}$ is an $\mathbb E$-inflation. It follows from \Cref{cor:infl-ret} that $g$ is a retract of an $\mathbb E$-inflation. \end{proof}

\vskip5pt

Using \Cref{prop:comp-infldefl-2} and \Cref{homotopyretract}, one can get a similar version of $3\times 3$ Lemma (\cite[Lemma 5.9]{NP19}) in an extriangulated category, without assuming that it is weakly idempotent complete.

\vskip5pt

\begin{proposition}\label{prop:deflation-retract}
	Let $(\mathcal{C}, \mathbb E, \mathfrak s)$ be an extriangulated category. Suppose that there are four $\mathbb E$-triangles $\con{A_1}{f_A}{A_2}{g_A}{A_3}{\delta_A}$, $\con{B_1}{f_B}{B_2}{g_B}{B_3}{\delta_B}$, $\con{A_2}{i_2}{B_2}{p_2}{C_2}{\varepsilon_2}$ and $\con{A_3}{i_3}{B_3}{p_3}{C_3}{\varepsilon_3}$ such that $g_B i_2 = i_3 g_A$. Then there exist $i_1: A_1\to B_1$ and $g_C: C_2\to C_3$
	\begin{equation*}
\xymatrix@R=0.4cm{
  A_1\ar[r]^{{f_A}}\ar@{..>}[d]_{{i_1}} & A_2\ar[r]^{{g_A}}\ar[d]^{{i_2}} & A_3\ar@{-->}[r]^{{\delta_A}}\ar[d]^{{i_3}} & {} \\
  B_1\ar[r]^{{f_B}} & B_2\ar[r]^{{g_B}}\ar[d]^{{p_2}} & B_3\ar@{-->}[r]^{{\delta_B}}\ar[d]^{{p_3}} & {} \\
  {} & C_2\ar@{..>}[r]^{{g_C}}\ar@{-->}[d]^{{\varepsilon _2}} & C_3\ar@{-->}[d]^{{\varepsilon _3}} & {} \\
  {} & {} & {} & {}
}
\end{equation*}
	such that
	\vskip5pt
$(1)$ \ $i_1$ is a retract of some $\mathbb E$-inflation, $g_C$ is a retract of some $\mathbb E$-deflation$;$
	\vskip5pt
$(2)$ \ $(i_1,i_2,i_3)$ and $(g_A,g_B,g_C)$ are homotopic morphisms of $\mathbb E$-triangles.
\vskip5pt

Moreover, one can always complete a $3 \times 3$ diagram from the above conditions if and only if $\mathcal C$ is weakly idempotent complete.
\end{proposition}

\begin{proof} \ By ET3$^{\mathrm{op}}$ there is $i: A_1\to B_1$ such that $(i, i_2, i_3)$ is a morphism of $\mathbb E$-triangles. Then by \Cref{thm:hs-morphism} there is $i_1: A_1\to B_1$ such that $(i_1,i_2,i_3)$ is a homotopic morphism.
By the definition of a homotopic morphism, there is a commutative diagram
\begin{equation*}
\xymatrix@R=0.4cm{
  A_1\ar[r]^{{f_A}}\ar[d]_{{i_1}} & A_2\ar[r]^{{g_A}}\ar@{..>}[d]^{{j_1}} & A_3\ar@{-->}[r]^{{\delta_A}}\ar@{=}[d] & {} \\
  B_1\ar[r]^{s}\ar@{=}[d] & E\ar[r]^{t}\ar@{..>}[d]^{{j_2}} & A_3\ar@{-->}[r]^{{(i_3)^\ast \delta_B}}\ar[d]^{{i_3}} & {} \\
  B_1\ar[r]^{{f_B}} & B_2\ar[r]^{{g_B}} & B_3\ar@{-->}[r]^{{\delta_B}} & {}
}
\end{equation*}
 with $i_2=j_2j_1$. By \Cref{prop:comp-infldefl-2}, $j_1$ is a retract of some $\mathbb E$-inflation, and hence by \Cref{homotopyretract}, $i_1$ is also a retract of some $\mathbb E$-inflation. The construction of $g_C$ is similar.

 By \Cref{thm:wic}$(5)$, one can always complete a $3 \times 3$ diagram from the above conditions if and only if $\mathcal C$ is weakly idempotent complete.
 This completes the proof.
	\end{proof}

  A weakly idempotent complete extriangulated category has the $\rm (d,i\text{-}d,?)$-variant of $4\times 4$ Lemma.

\begin{proposition} \ 	Let $(\mathcal{C}, \mathbb E, \mathfrak s)$ be a weakly idempotent complete extriangulated category. Suppose that
		\begin{equation*}
\xymatrix@R=0.4cm{
  X\ar[r]^{f}\ar[d]_-{\alpha} & Y\ar[r]^{g}\ar[d]^-{\beta} & Z\ar@{-->}[r]^{\delta} & {} \\
  L\ar[r]^{u} & M\ar[r]^{v} & N\ar@{-->}[r]^{{\varepsilon }} & {}
}
\end{equation*}
is a diagram of $\mathbb E$-triangles such that  $u\alpha=\beta f$, $\alpha$ is an $\mathbb E$-deflation and $\beta$ is an $\mathbb E$-inflations and an $\mathbb E$-deflations.	Then $\alpha$ is also an $\mathbb E$-inflation, and there is an $\mathbb E$-deflation $\gamma$ such that $(\alpha, \beta, \gamma)$ is a homotopic morphism of $\mathbb E$-triangles. Moreover, there is a commutative diagram
	\begin{equation*}
\xymatrix@R=0.4cm{
  K_\alpha\ar@{..>}[r]^{{f'}}\ar@{..>}[d]_-{{i_\alpha }} & K_\beta\ar@{..>}[r]^{{g'}}\ar@{..>}[d]^{{i_\beta}} & K_\gamma\ar@{..>}[d]^{{i_\gamma}}\ar@{..>}[r] & {} \\
  X\ar[r]^{f}\ar[d]_-{\alpha} & Y\ar[r]^{g}\ar[d]^-{\beta} & Z\ar@{-->}[r]^{\delta}\ar@{..>}[d]^-{\gamma} & {} \\
  L\ar[r]^{u}\ar@{..>}[d]_{{p_\alpha}} & M\ar[r]^{v}\ar@{..>}[d]^{{p_\beta}} & N\ar@{-->}[r]^{{\varepsilon }} & {} \\
  C_\alpha\ar@{..>}[r]^{{u'}} & C_\beta & {} & {}
}
\end{equation*}
with a connecting morphism $z:K_\gamma\to C_\alpha$, such that the following are $\mathbb E$-triangles$:$

            \vskip5pt
\begin{tasks}[style=enumerate, label=$(\arabic*)$, label-width=4ex](2)
		\task $\con {K_\alpha}{f'}{K_\beta}{g'}{K_\gamma}{\ell^{-1} (z)};$
		\task $\con {K_\beta}{g'}{K_\gamma}z{C_\alpha}{-\ell^{-1} (u')};$
		\task $\con {K_\gamma}z{C_\alpha}{u'}{C_\beta}{\rho^{-1}(g')};$
		\task $\con {K_\alpha}{i_\alpha}X\alpha L{\ell^{-1}(p_\alpha)};$
		\task $\con {K_\beta}{i_\beta}Y\beta M{\ell^{-1}(p_\beta)};$
		\task $\con {K_\gamma}{i_\gamma}Z\gamma N{\ell^{-1}(p_\gamma)};$
		\task $\con X\alpha L{p_\alpha}{C_\alpha}{-\rho^{-1}(i_\alpha)};$
		\task $\con Y\beta M{p_\beta}{C_\beta}{-\rho^{-1}(i_\beta)}$
	\end{tasks}
        \vskip10pt
        with morphisms of $\mathbb E$-triangles $(i_\alpha,i_\beta,i_\gamma)$, $(f',f,u)$, $(f,u, u')$, $(g',g,v)$,
and that $C_\alpha=FK_\alpha$, $C_\beta=FK_\beta$, $u'=Ff'$, where $F$ is the functor defined in \Cref{F}, and $\ell$ and $\rho$ are defined in \Cref{binaturality}.
\vskip5pt
In particular, there is a sequence
$K_\alpha \xrightarrow{f'} K_\beta \xrightarrow{g'} K_\gamma \xrightarrow{z} C_\alpha \xrightarrow{u'} C_\beta$
such that any subsequent three objects form an $\mathbb E$-triangle as in $(1)$-$(3)$.
\end{proposition}

\begin{proof}
  Since $\beta f$ is an $\mathbb E$-inflation, $\alpha$ is an $\mathbb E$-inflation by \Cref{prop:comp-infldefl}$(1)$. The remaining part of the proof follows from \Cref{prop:4x4-S-S-?}.
\end{proof}

\begingroup
\makeatletter\let\addcontentsline\@gobblethree\makeatother

\endgroup
\end{document}